\documentclass[11pt]{article}

\usepackage{amsmath,amssymb,amsfonts,amsthm,enumitem,stmaryrd}
\usepackage{pstricks,pst-node,xy,float}
\xyoption{all}

\numberwithin{equation}{section}
\newtheorem{thm}{Theorem}[section]
\newtheorem{prp}[thm]{Proposition}
\newtheorem{lmm}[thm]{Lemma}   
\newtheorem{crl}[thm]{Corollary}

\newtheorem{mythm}{Theorem}

\newtheorem{myconj}{Conjecture}

\newtheorem{thmlet}{Theorem}

\theoremstyle{definition}
\newtheorem{eg}[thm]{Example}
\newtheorem{rmk}[thm]{Remark}

\newcounter{zr}

\def\sf#1{\textsf{#1}}

\def\BE#1{\begin{equation}\label{#1}}  
\def\EE{\end{equation}}
\def\e_ref#1{(\ref{#1})}

\def\ov#1{\overline{#1}}
\def\un#1{\underline{#1}}

\def\wt#1{\widetilde{#1}}
\def\wh#1{\widehat{#1}}

\def\lan{\langle}  
\def\ran{\rangle}
\def\lr#1{\lan{#1}\ran}
\def\llrr#1{\lan\!\lan{#1}\ran\!\ran}
\def\blr#1{\big\lan{#1}\big\ran}
\def\bllrr#1{\big\lan\!\big\lan{#1}\big\ran\!\big\ran}
\def\bblr#1{\bigg\lan{#1}\bigg\ran}
\def\lrbr#1{\llbracket{#1}\rrbracket}
\def\blrbr#1{\big\llbracket{#1}\big\rrbracket}
\def\LRbr#1{\left\llbracket{#1}\right\rrbracket}
\def\Bgbr#1{\Bigg\llbracket{#1}\Bigg\rrbracket}
\def\smsize#1{\begin{small}#1\end{small}}

\def\lra{\longrightarrow}
\def\Lra{\Longrightarrow}
\def\Llra{\Longleftrightarrow}

\def\al{\alpha}
\def\be{\beta}
\def\de{\delta}
\def\ep{\epsilon}
\def\ga{\gamma}

\def\la{\lambda}
\def\om{\omega}
\def\si{\sigma}
\def\th{\theta}
\def\ve{\epsilon}

\def\ze{\zeta}

\def\De{\Delta}
\def\Ga{\Gamma}

\def\cA{\mathcal A}
\def\c{\mathbf c}
\def\C{\mathbb C}
\def\cC{\mathcal C}
\def\nc{\textnormal{c}}

\def\bfd{\mathbf d}

\def\E{\mathbf e}

\def\bE{\mathbb E}
\def\F{\mathcal F}

\def\fL{\mathfrak L}
\def\H{\mathcal H}

\def\I{\mathfrak i}
\def\cI{\mathcal I}
\def\bfI{\mathbf I}

\def\M{\mathfrak M}

\def\cM{\mathcal M}
\def\N{\mathcal N}
\def\bfN{\mathbf N}

\def\P{\mathbb P}
\def\cP{\mathcal P}
\def\Q{\mathbb Q}
\def\R{\mathbb R}

\def\fR{\mathfrak R}
\def\bfS{\mathbf S}
\def\cS{\mathcal S}
\def\bft{\mathbf t}
\def\T{\mathbb T}

\def\V{\mathcal V}
\def\cY{\mathcal Y}
\def\Z{\mathbb Z}
\def\cZ{\mathcal Z}

\def\a{\mathbf{a}}
\def\fa{\mathfrak a}
\def\b{\mathbf{b}}
\def\d{\mathfrak d}
\def\g{\mathfrak g}
\def\bfg{\mathbf g}
\def\tnd{\textnormal{d}}
\def\p{\mathbf{p}}
\def\fp{\mathfrak p}

\def\s{\mathbf{s}}
\def\x{\mathbf{x}}
\def\0{\mathbf{0}}

\def\i{\infty}
\def\hb{\hbar}

\def\Aut{\textnormal{Aut}}
\def\Fl{\textnormal{Fl}}
\def\tnd{\textnormal{d}}

\def\ev{\textnormal{ev}}
\def\LHS{\textnormal{LHS}}
\def\val{\textnormal{val}}
\def\Edg{\textnormal{Edg}}

\def\Ver{\textnormal{Ver}}
\def\vir{\textnormal{vir}}

\def\1{\mathbf 1}

\def\dset{\llfloor d\rrfloor}
\def\nset{\llfloor n\rrfloor}
\def\flr#1{\llfloor{#1}\rrfloor}

\def\eset{\emptyset}

\def\Res#1{\underset{#1}{\fR}}

\begin{document}

\title{Energy Bounds and Vanishing Results for\\
the Gromov-Witten Invariants of the Projective Space}
\author{Aleksey Zinger\thanks{Partially supported by NSF grant 1500875 and MPIM}}

\date{\today}
\maketitle

\begin{abstract}
\noindent
We describe generating functions for arbitrary-genus Gromov-Witten invariants
of the projective space with any number of marked points explicitly.
The structural portion of this description gives rise to uniform
energy bounds and vanishing results for these invariants.
They suggest deep conjectures relating Gromov-Witten invariants of symplectic manifolds to
the energy of pseudo-holomorphic maps and the expected dimension of their moduli space.
\end{abstract}
\tableofcontents

\section{Introduction}
\label{intro_sec}

\noindent
Gromov-Witten (or \sf{GW-}) invariants of a smooth projective variety 
(or more generally of a symplectic manifold)~$X$
are certain counts of (pseudo-holomorphic) curves in~$X$.
These invariants are known or conjectured to possess many striking properties
which are often completely unexpected from the classical point of view.
For example, physical considerations suggest that these invariants are uniformly bounded
by the symplectic area of the curves being counted; see Conjecture~\ref{GWbound_cnj}.
We confirm this conjecture for the complex projective space~$\P^{n-1}$ in all genera
by applying the explicit formula of Theorem~\ref{main_thm} in Section~\ref{Mainform_subs}.
We also use this theorem to confirm the vanishing predictions of Conjecture~\ref{GWeq0_cnj}
for~$\P^{n-1}$.\\

\noindent
Generating functions for the 1-pointed genus~0 GW-invariants 
of semi-positive projective complete intersections $X\!\subset\!\P^{n-1}$
are explicitly computed in~\cite{Gi,LLY}.
The resulting formulas in particular confirm the mirror symmetry prediction of~\cite{CdGP} 
for the genus~0 GW-invariants of a quintic threefold,
i.e.~a degree~5 hypersurface in~$\P^4$.
By~\cite{BeKl,bcov0}, generating functions for 2-pointed genus~0 
GW-invariants of hypersurfaces are explicit transforms of 
the 1-pointed genus~0 functions;
these results are extended to projective  complete intersections in \cite{Ch,PoZ}
and to complete intersections in toric varieties in~\cite{Popa}.
It is shown in~\cite{g0ci} that generating functions  for $N$-pointed genus~0 
GW-invariants of  projective  complete intersections, with $N\!\ge\!3$,
are also explicit transforms of the 1-pointed genus~0 functions.
Combined with \cite{bcov1,bcov1ci}, this implies the same for
generating functions  for $N$-pointed genus~1 
GW-invariants of  projective complete intersections.
We show in this paper that a natural generating function for 
the $N$-pointed genus~$g$ GW-invariants of~$\P^{n-1}$ 
is an explicit transform of the 1-pointed genus~0 generating function
as well;
see Theorem~\ref{main_thm} in Section~\ref{Mainform_subs} and
Theorem~\ref{equiv_thm} in Section~\ref{equiv_sec}.\\

\noindent
Throughout the paper $n,g,N\!\in\!\Z^+$ are fixed integers,
with $g$ and $N$ denoting the genus of the curves being counted and  the number of marked points,
respectively.
Let
$$\Z^{\ge0}=\{0\}\sqcup\Z^+ \qquad\hbox{and}\qquad [N]=\{1,2,\ldots,N\}.$$
For $d\!\in\!\Z^{\ge0}$, we denote by $\ov\M_{g,N}(\P^{n-1},d)$
the moduli space of stable $N$-marked genus~$g$ degree~$d$ maps to~$\P^{n-1}$. 
For each $s\!=\!1,\ldots,N$, let
$$\ev_s\!: \ov\M_{g,N}(\P^{n-1},d)\lra \P^{n-1}, \qquad
\psi_s\equiv c_1(L_s^*)\in H^2\big(\ov\M_{g,N}(\P^{n-1},d)\big),$$
be the evaluation map 
and the first Chern of  the universal cotangent line bundle at the $s$-th marked point,
respectively.
Denote by  $H\!\in\!H^2(\P^{n-1})$ the hyperplane class.\\

\noindent
The main theorem of this paper, Theorem~\ref{main_thm} stated at the beginning of
Section~\ref{mainthm_sec}, 
provides a closed formula for the $N$-pointed genus~$g$ version of 
the standard (1-pointed genus~0) Givental's $J$-function.
This is a generating function for the genus~$g$ GW-invariants 
\BE{GWXadfn_e}\blr{\tau_{b_1}H^{c_1},\ldots,\tau_{b_N}H^{c_N}}_{g,d}^{\P^{n-1}}
\equiv \int_{[\ov\M_{g,N}(\P^{n-1},d)]^{\vir}}    
\prod_{s=1}^{s=N}\!\!\big(\psi_s^{b_s}\ev_s^*H^{c_s}\big)\EE
of $\P^{n-1}$.
The most basic positive-genus case of Theorem~\ref{main_thm} 
is equivalent to Theorem~\ref{g1N1_thm} below.

\setcounter{mythm}{-1}

\begin{mythm}\label{g1N1_thm}
For all $n,d,c\!\in\!\Z^{\ge0}$ with $c\!<\!n$ and $n\!\ge\!2$,
$$\blr{\tau_{nd+1-c}H^c}^{\P^{n-1}}_{1,d}=
\Bgbr{\frac{n(1\!+\!2d\!-\!n\!+\!2w)}{48}
\frac{(d\!+\!w)^{n-2}}{\prod\limits_{r=1}^d\!(r\!+\!w)^n}}_{w;n-1-c}\,,$$
with $\blrbr{f}_{w;r}$ denoting the coefficient of $w^r$ in the power series
expansion of a function \hbox{$f\!=\!f(w)$} around $w\!=\!0$.
\end{mythm}

\noindent
This theorem is obtained in Section~\ref{g1N1pf_subs}.
While the precise statement of Theorem~\ref{main_thm} is quite involved in general,
its qualitative corollaries, Theorems~\ref{GWbound_thm} and~\ref{GWeq0_thm} below, 
are quite simple to state; they are established in Section~\ref{applpf_subs}.

\begin{mythm}\label{GWbound_thm}
For all $n\!\in\!\Z^+$ and $g\!\in\!\Z^{\ge0}$,
there exists $C_{n,g}\!\in\!\R^+$ such~that 
$$\bigg|\frac{\blr{b_1!\,\tau_{b_1}H^{c_1},\ldots,b_N!\,\tau_{b_N}H^{c_N}}_{g,d}^{\P^{n-1}}}{N!}
\bigg| \le C_{n,g}^{d+N}$$
for all $N\!\in\!\Z^+$  and $d,b_1,\ldots,b_N,c_1,\ldots,c_N\!\in\!\Z^{\ge0}$.
\end{mythm}

\noindent
In the basic $d\!=\!0$ case, the invariants of Theorem~\ref{GWbound_thm} become
\BE{d0GWs_e}\blr{\tau_{b_1}H^{c_1},\ldots,\tau_{b_N}H^{c_N}}_{g,0}^{\P^{n-1}}
=\int_{\ov\cM_{g,N}\times\P^{n-1}}
e\big(\bE_g^*\!\otimes\!T\P^{n-1}\big)\prod_{i=1}^N\big(\psi_i^{b_i}H^{c_i}\big),\EE
where 
$$\bE_g\lra\ov\cM_{g,N}$$
is the Hodge vector bundle of holomorphic differentials over
the Deligne-Mumford moduli space of genus~$g$ curves
with $N$~marked points.
By~\e_ref{d0GWs_e}, \e_ref{CgIdfn_e0}, and induction via~\e_ref{StrDil_e},
$$\bigg|\frac{\blr{b_1!\,\tau_{b_1}H^{c_1},\ldots,b_N!\,\tau_{b_N}H^{c_N}}_{g,0}^{\P^{n-1}}}{N!}\bigg|
\le C_{n,g}2^N$$
for some $C_{n,g}\!\in\!\R^+$ determined by the numbers $C_{g;n;I}$
in~\e_ref{CgIdfn_e0} and by the top-dimensional intersections of $\la$ and $\psi$-classes on
the moduli spaces $\ov\cM_{g,N}$ with $N\!\le\!6g\!-\!6$.
The base~2 above can be replaced by any number arbitrarily close to~1 
at the cost of increasing~$C_{n,g}$.
The $d\!=\!0$ case of Theorem~\ref{GWbound_thm} is thus straightforward.
The case of Theorem~\ref{GWbound_thm} with $n\!=\!3$, $b_i\!=\!0$ for all~$i$,
and $c_i\!=\!2$ for all~$i$ is consistent with the asymptotic prediction
of \cite[Footnote~2]{FI} for the number~$n_{g,d}$ of genus~$g$ degree~$d$ curves in~$\P^2$ 
passing through $3d\!-\!1\!+\!g$ general points.

\begin{mythm}\label{GWeq0_thm}
Suppose $n,g,N\!\in\!\Z^+$ with $2g\!+\!N\!\ge\!3$
and $(b_s)_{s\in[N]},(c_s)_{s\in[N]}\!\in\!(\Z^{\ge0})^N$.
If there exists $S\!\subset\![N]$ such~that 
$$b_s\!+\!c_s<n~~\forall\,s\!\in\!S \qquad\hbox{and}\qquad
\sum_{s\in S}b_s>3(g\!-\!1)\!+\!N,$$ 
then $\blr{\tau_{b_1}H^{c_1},\ldots,\tau_{b_N}H^{c_N}}_{g,d}^{\P^{n-1}}=0$.
\end{mythm}

\noindent
Theorems~\ref{GWbound_thm} and~\ref{GWeq0_thm} are potential indications of
fundamental properties of GW-invariants that are out of reach of the current methods.
Their statements have natural intrinsic extensions to more general symplectic manifolds,
formulated in the two conjectures below.
The exponent $\lr{\om,\be}$  in Conjecture~\ref{GWbound_cnj}
is the energy of the $J$-holomorphic maps of class~$\be$, while
$\lr{\om,\be}\!+\!N$ is essentially the energy of the induced ``graph map".
A symplectic manifold~$(X,\om)$ is called \sf{monotone with minimal Chern number $\nu\!\in\!\R^+$}
in Conjecture~\ref{GWeq0_cnj}
if 
$$c_1(X)=\la[\om]\in H^2(X;\R)$$ 
for some $\la\!\in\!\R^+$ and $\nu$ is the minimal value of $c_1(X)$
on the homology classes representable by non-constant $J$-holomorphic maps $\P^1\!\lra\!X$
for every $\om$-compatible almost complex structure on~$X$.

\begin{myconj}[{\cite[Conjecture~1]{g0ci}}]\label{GWbound_cnj}
Suppose $(X,\om)$ is a compact symplectic manifold and $g\!\in\!\Z$.
For all $H_1,\ldots,H_k\!\in\!H^*(X)$, there exists $C_{X,g}\!\in\!\R^+$ such~that 
$$\bigg|\frac{\blr{b_1!\,\tau_{b_1}H_{c_1},\ldots,b_N!\,\tau_{b_N}H_{c_N}}_{g,\be}^X}{ N!}
\bigg|
\le C_{X,g}^{\lr{\om,\be}+N}
\qquad\forall\,\be\!\in\!H_2(X),\,N,b_s\!\ge\!0,\,c_s\!\in\![k].$$
\end{myconj}

\begin{myconj}\label{GWeq0_cnj}
Suppose $(X,\om)$ is a compact monotone symplectic manifold with minimal 
Chern number~$\nu$, 
$$g,N\in\Z^{\ge0} ~~\hbox{with}~  2g\!+\!N\ge3,\quad
(b_s)_{s\in[N]},(c_s)_{s\in[N]}\in (\Z^{\ge0})^N, 
\quad\hbox{and}~~ H_s\!\in H^{2c_s}(X)~~\forall\,s\!\in\![N].$$
If there exists $S\!\subset\![N]$ such~that
$$b_s\!+\!c_s<\nu~~\forall\,s\!\in\!S
\qquad\hbox{and}\qquad \sum_{s\in S}b_s>3(g\!-\!1)\!+\!N,$$
then $\blr{\tau_{b_1}H_1,\ldots,\tau_{b_N}H_N}_{g,\be}^X=0$.
\end{myconj}

\noindent
Theorems~\ref{GWbound_thm} and~\ref{GWeq0_thm} establish Conjectures~\ref{GWbound_cnj} 
and~\ref{GWeq0_cnj} for $X\!=\!\P^{n-1}$.
Theorems~1 and~2 in~\cite{g0ci} establish 
the $g\!=\!0$ cases of these conjectures for complete intersections $X\!\subset\!\P^n$
with each $H_i$ being a power of~$H$.
Conjecture~\ref{GWbound_cnj} for a Calabi-Yau threefold~$X$ corresponds to 
the string theory presumption that the partition functions determined by 
GW-invariants have positive radii of convergence.
If~$X$ is a Calabi-Yau  intersection 3-fold in~$\P^n$, 
this is equivalent to the existence of $C_{X,g}\!\in\!\R^+$ such~that 
\BE{CYbnd_e}\big|\blr{}_{g,d}^X\big|\le C_{X,g}^d\qquad\forall~d\!\in\!\Z^+\,.\EE
For a Calabi-Yau~$X$ of (complex) dimension at least~4,
the GW-invariants of genus~2 and higher vanish.
Conjecture~\ref{GWbound_cnj} then reduces to its cases for the genus~0 GW-invariants
with primary insertions ($b_i\!=\!0$ for all~$i$)
and for the genus~1 GW-invariants with no insertions.
For complete intersections $X\!\subset\!\P^n$,
such genus~0 bounds with each $H_i$ being a power of~$H$
are implied by the mirror formulas established  in~\cite{Gi,LLY};
these mirror formulas and bounds extend to many other GIT quotients.
The required genus~1 bounds for complete intersections $X\!\subset\!\P^n$ 
are implied by the genus~1 mirror formulas established in~\cite{bcov1,bcov1ci}.\\

\noindent
The virtual localization theorem of~\cite{GP} reduces the computation 
of positive-genus GW-invariants of~$\P^{n-1}$ to a sum over weighted graphs.
We use the approach of~\cite{bcov1} for breaking such graphs at special nodes 
to show that a generating function for the $N$-pointed genus~$g$ GW-invariants of~$\P^n$
is a linear combination of $N$-fold products of derivatives of 
a generating function for the 1-pointed genus~0 GW-invariants
with coefficients that are polynomials of total degree at most~$3(g\!-\!1)\!+\!N$.
In contrast to the application of this approach in~\cite{g0ci}
to compute $N$-pointed  genus~0 
GW-invariants of complete intersections $X\!\subset\!\P^n$ with $N\!\ge\!3$,
the present application requires dealing with graphs containing loops
and understanding the structure of Hodge integrals over~$\ov\cM_{g,N}$ 
as $N$ increases.
While we describe two explicit ways of computing the relevant coefficients,
the final formulas become rather complicated as~$g$ and~$N$ increase.
Nevertheless, our qualitative description of these coefficients suffices
to deduce Theorem~\ref{GWbound_thm} and to immediately obtain Theorem~\ref{GWeq0_thm}.\\

\noindent
The approach in this paper can be used to compute twisted $N$-pointed genus~$g$ GW-invariants
of~$\P^n$, but these do not correspond to the usual GW-invariants of the associated complete intersection
for $g\!\ge\!1$.
There are two necessary inputs for doing~so.
The first is Proposition~\ref{HodgeInt_prp}, which concerns the structure of Hodge integrals
only and is thus directly applicable in all cases.
The second input is Proposition~\ref{Fexp_prp}, which provides an asymptotic expansion
of the mirror hypergeometric function corresponding to the standard Givental's $J$-function.
The approach of~\cite{ZaZ} can be used directly to determine the power series~$\xi$ and~$\Phi_b$
appearing in such expansions in the cases of twisted invariants;
in the cases relevant to the projective complete intersections, 
they are determined in~\cite{bcov1ci}.\\

\noindent
In principle, all genus~$g$ GW-invariants of~$\P^{n-1}$ can be determined via 
\cite[Theorem~1]{Gi01}.
However, it is unclear how feasible it is to obtain such qualitative conclusions
as our Theorems~\ref{GWbound_thm} and~\ref{GWeq0_thm} from~\cite{Gi01}.
The $g\!=\!0$ case of Theorem~\ref{GWbound_thm}, i.e.~\cite[Theorem~1]{g0ci},
in fact confirmed a conjecture of R.~Pandharipande.
This statement was part of the idea of~\cite{MaP} to establish the bounds~\e_ref{CYbnd_e}
in all genera by reducing them to the $n\!=\!4$ bounds of Theorem~\ref{GWbound_thm} 
via a degeneration scheme of~\cite{MaP0} and reducing 
the latter bounds to the $g\!=\!0$ case via~\cite{Gi01}. 
As far as we are aware, the approach of establishing 
the bounds of Theorem~\ref{GWbound_thm} from the $g\!=\!0$ case 
 via~\cite{Gi01} has not been completed~yet.\\

\noindent
The main theorem of this paper, Theorem~\ref{main_thm}, 
is stated at the beginning of Section~\ref{mainthm_sec}.
Two descriptions of the structure coefficients appearing in this theorem are given
in Section~\ref{Mainform_subs} after the two necessary inputs are introduced in
Sections~\ref{HodgeInt_subs0} and~\ref{HG_subs}.
We compute these coefficients  in the $(g,N)\!=\!(1,1)$ case explicitly in
Section~\ref{g1N1pf_subs} and establish Theorem~\ref{g1N1_thm}.
Theorems~\ref{GWbound_thm} and~\ref{GWeq0_thm} are proved in Section~\ref{applpf_subs}.
The former is deduced directly from the structural description of Theorem~\ref{main_thm};
its proof makes no use of Sections~\ref{HodgeInt_subs0}-\ref{Mainform_subs}.
Theorem~\ref{main_thm} is an immediate consequence of its equivariant version,
Theorem~\ref{equiv_thm}, stated in Section~\ref{equivGW_subs}.
Section~\ref{outline_subs} applies the virtual localization theorem of~\cite{GP} 
to the equivariant $N$-point genus~$g$ Givental's $J$-function appearing 
of Theorem~\ref{equiv_thm}, reducing it to a sum of rational functions in
the equivariant weights over infinitely graphs. 
Section~\ref{pfoutline_subs} describes two approaches for breaking these graphs 
at special vertices and reducing the associated infinite sum to 
a sum of finitely many power series.
The necessary equivariant inputs for the resulting finite sums are collected
in Section~\ref{prelim_subs}.
After a quick preparation in Section~\ref{prelimcomp_subs}, 
Sections~\ref{ClFormComp_subs} and~\ref{RecFormComp_subs} implement the two approaches outlined 
in Section~\ref{pfoutline_subs} and establish Theorem~\ref{equiv_thm} 
with the two descriptions of the structure coefficients of Section~\ref{Mainform_subs}.
Sections~\ref{HodgeInt_subs} and~\ref{HodgeIntGS_subs} establish the key combinatorial identities 
involving Hodge integrals that are used in the proof of Theorem~\ref{equiv_thm}, 
Propositions~\ref{HodgeInt_prp} and~\ref{HodgeIntGS_prp}, respectively.

\section{Main theorem and some applications}
\label{mainthm_sec}

\noindent
For  $n,N\!\in\!\Z^{\ge0}$, let
$$\llfloor{n}\rrfloor=\big\{p\!\in\!\Z^{\ge0}\!:\,p\!<\!n\big\}, \qquad
\P^{n-1}_N\!=\!(\P^{n-1})^N\,.$$
For each $s\!=\!1,\ldots,N$, we~set
$$H_s=\pi_s^*H\in H^2\big(\P^{n-1}_N\big),$$
where $\pi_s\!:\P^{n-1}_N\!\lra\!\P^{n-1}$ is the projection onto 
the $s$-th coordinate.
If in addition \hbox{$g,d\!\in\!\Z^{\ge0}$},
the virtual fundamental class of $\ov\M_{g,N}(\P^{n-1},d)$ determines 
a cohomology push-forward 
$$\ev_*^d\equiv \big\{\ev_1\!\times\!\ldots\!\times\!\ev_N\big\}_*\!:
H^*\big(\ov\M_{g,N}(\P^{n-1},d)\big)\lra H^*\big(\P^{n-1}_N\big).$$
With $\un\hb\!=\!(\hb_1,\ldots,\hb_N)$, 
$\un\hb^{-1}\!=\!(\hb_1^{-1},\ldots,\hb_N^{-1})$, 
and $\un{H}\!=\!(H_1,\ldots,H_N)$, let 
\BE{Zdfn_e} 
Z^{(g)}\big(\un\hb,\un{H},q\big)= \sum_{d=0}^{\i}q^d
\ev_*^d\Bigg\{\prod_{s=1}^{s=N}\!\!\frac{1}{\hb_s\!-\!\psi_s}\Bigg\}
\in H^*(\P^{n-1}_N)\big[\un\hb^{-1}\big]\big[\big[q\big]\big].\EE
This power encodes all descendant $N$-pointed genus~$g$  GWs of~$\P^{n-1}$
defined in~\e_ref{GWXadfn_e}.\\

\noindent
For $\b\!\equiv\!(b_s)_{s\in[N]}\!\in\!(\Z^{\ge0})^N$, a tuple $\hb$ as above, 
and $p\!\in\!\Z^{\ge0}$, let
$$|\b|=\sum_{s=1}^{s=N}\!b_s, \quad
\un\hb^{-\b}=\prod_{s=1}^{s=N}\!(\hb_s^{-1})^{b_s}\,, \quad
F_p(w,q)=\sum_{d=0}^{\i}q^d\frac{(w\!+\!d)^pw^{nd-p}}{\prod\limits_{r=1}^d(w\!+\!r)^n}
\in \Q(w)\big[\big[q\big]\big]\,.$$
For  $\p\!\equiv\!(p_s)_{s\in[N]}\!\in\!\nset^N$
and a tuple $\un{H}\!=\!(H_s)_{s\in[N]}$ of formal variables, define
\BE{Depdfn_e}
w_s=\frac{H_s}{\hb_s}\,,\quad q_s=\frac{q}{H_s^n}\,, \quad
\De_{\p}(\un\hb,\un{H},q)=\prod_{s=1}^{s=N}
\frac{H_s^{p_s}}{\hb_s}F_{p_s}(w_s,q_s)\in 
\Q[\un{H}]\big[\big[\un\hb^{-1}\big]\big]\big[\big[q\big]\big]\,.\EE

\begin{thmlet}\label{main_thm}
Suppose $n,N\!\in\!\Z^+$ and $g\!\in\!\Z^{\ge0}$ with $n\!\ge\!2$ and $2g\!+\!N\!\ge\!3$.
The generating function~\e_ref{Zdfn_e} for the $N$-pointed genus~$g$ GW-invariants of~$\P^{n-1}$ 
is given~by 
\BE{mainthm_e} 
Z^{(g)}\big(\un\hb,\un{H},q\big)= 
\sum_{\p\in\nset^N}
\sum_{\begin{subarray}{c}\b\in(\Z^{\ge0})^N\\ |\b|\le 3(g-1)+N  \end{subarray}}
\sum_{d=0}^{\i}
\nc_{g;\p,\b}^{(d)}q^d\un\hb^{-\b}\De_{\p}(\un\hb,\un{H},q),\EE
with the coefficients $\nc_{g;\p,\b}^{(d)}\!\in\!\Q$ as specified in Section~\ref{Mainform_subs}.
\end{thmlet}

\subsection{Structure of Hodge integrals}
\label{HodgeInt_subs0}

\noindent
One of the two inputs determining the structure coefficients $\nc_{g;\p,\b}^{(d)}$
in Theorem~\ref{main_thm} are properties of the Hodge integrals on the Deligne-Mumford
moduli spaces~$\ov\cM_{g,N}$ arising from the string and dilaton equations.
This input is provided by Proposition~\ref{HodgeInt_prp} below;
it is proved in Section~\ref{HodgeInt_subs}.\\

\noindent
Let $g\!\in\!\Z^{\ge0}$.
For a tuple $I\!\equiv\!(i_1,i_2,\ldots)$ in $\Z^{\i}$ 
or in  \hbox{$\Z^g\!\subset\!\Z^{\i}$}, define 
$$|I|=\sum_{k=1}^{\i}i_k\in\Z\,, \qquad 
\|I\|=\sum_{k=1}^{\i}ki_k\in\Z\,, \qquad
\mu_g(I)=3(g\!-\!1)\!-\!\|I\|\in\Z.$$
If in addition $m\!\in\!\Z^{\ge0}$ with $2g\!+\!m\!\ge\!3$ and
$I\!\in\!(\Z^{\ge0})^g$, let
$$\la_{g;I}=\prod_{k=1}^g\!c_k(\bE_g)^{i_k}\in H^{2\|I\|}\big(\ov\cM_{g,m}\big)\,.$$
If $m'\!\in\!\Z^{\ge0}$ and $\b'\!\in\!(\Z^{\ge0})^{m'}$,
denote by $\b\b'$ the $(m\!+\!m')$-tuple obtained 
by adjoining~$\b'$ to~$\b$ at the~end.\\

\noindent
If $\b\!\equiv\!(b_k)_{k\in[m]}\in\!(\Z^{\ge0})^m$, let
\begin{gather}
\label{DMdfn_e}
\bllrr{\la_{g;I};\tau_{\b}} =
\begin{cases}
\int\limits_{\ov\cM_{g,|\b|-\mu_g(I)}}\!\!\!\!\!\!\!\!\!
\la_{g;I}\!
\prod\limits_{k=1}^m\!\!\psi_k^{b_k},&\hbox{if}~|\b|\!\ge\!\mu_g(I)\!+\!m;\\
0,&\hbox{if}~|\b|\!<\!\mu_g(I)\!+\!m;
\end{cases}\\
\label{DMdfn_e2}
\bllrr{\la_{g;I};\wt\tau_{\b}}
=\bigg(\prod_{k=1}^m\!b_k!\bigg)\bllrr{\la_{g;I};\tau_{\b}}\,.
\end{gather}
For $\c\!\equiv\!(\c_r)_{r\in\Z^+}\!\in\!(\Z^{\ge0})^{\i}$, let
\begin{gather*}
S(\c)=\big\{(r,j)\!\in\!\Z^+\!\times\!\Z^+\!\!:
(r,j)\!\in\!\{r\}\!\times\![c_r]~\forall\,r\!\in\!\Z^+\big\},\\
A_{I,\b;\c}^{(g)}=
\sum_{\b'\in(\Z^{\ge0})^{S(\c)}}\!\!\!\Bigg(\!\! (-1)^{|\b'|}
\frac{\llrr{\la_{g;I};\wt\tau_{\b\b'}}}
{(|\b|\!+\!|\b'|\!-\!\mu_g(I)\!-\!m\!-\!|\c|)!}
\prod_{(r,j)\in S(\c)}\!\!\binom{r}{b_{r,j}'}\!\!\Bigg).
\end{gather*}
In particular, $|S(\c)|\!=\!|\c|$ and the numerator above vanishes whenever
the argument of the factorial in the denominator is negative.

\begin{prp}\label{HodgeInt_prp}
Let $g,m\!\in\!\Z^{\ge0}$  with $2g\!+\!m\!\ge\!3$ and $I\!\in\!(\Z^{\ge0})^g$.
There exists a collection 
$$A_{I;\c}^{(g,\un\ep)}\in\Q \qquad\hbox{with}\quad
\un\ep\in(\Z^{\ge0})^m,~\c\!\in\!(\Z^{\ge0})^{\i},$$
which is invariant under the permutations of the components of~$\un\ep$ such~that
\BE{DMstr_e} 
A_{I,\b;\c}^{(g)}=(-1)^{\|\c\|}\hspace{-.2in}
\sum_{\begin{subarray}{c} \un\ep\in (\Z^{\ge0})^m\\
|\un\ep|\le\mu_g(I)+m\end{subarray}} \hspace{-.2in}
A_{I;\c}^{(g,\un\ep)}
\binom{|\b|\!-\!|\un\ep|}{\mu_g(I)\!+\!m\!+\!|\c|\!-\!\|\c\|\!-\!|\un\ep|}
\!\!\prod_{k=1}^m\!\!\binom{b_k}{\ep_k}\EE
for all $\b\!\in\!(\Z^{\ge0})^m$ and $\c\!\in\!(\Z^{\ge0})^{\i}$. 
These collections can be chosen so that there exists $C_g\!\in\!\R$ such~that 
\BE{DMstr_e2}\big|A_{I;\c}^{(g,\un\ep)} \big|\le 
C_g\,2^{\|\c\|}\big(3(g\!-\!1)\!+\!m\!+\!|\c|\big)!\EE
for all $I\!\in\!(\Z^{\ge0})^g$, $\c\!\in\!(\Z^{\ge0})^{\i}$,  
$\un\ep\!\in\!(\Z^{\ge0})^m$, and $m\!\in\!\Z^{\ge0}$
with $2g\!+\!m\!\ge\!3$.
\end{prp}

\noindent 
For example, 
\BE{AgIb0_e2}
\frac{\llrr{\la_{g;I};\wt\tau_{\b}}}
{(|\b|\!-\!\mu_g(I)\!-\!m)!}=A_{I,\b;\0}^{(g)}
=\sum_{\begin{subarray}{c} \un\ep\in (\Z^{\ge0})^m\\
|\un\ep|\le\mu_g(I)+m\end{subarray}} \hspace{-.2in}
A_{I;\0}^{(g,\un\ep)}
\binom{|\b|\!-\!|\un\ep|}{\mu_g(I)\!+\!m\!-\!|\un\ep|}
\!\!\prod_{k=1}^m\!\!\binom{b_k}{\ep_k}\,.\EE
If $r\!\in\!\Z^+$ and  $e_r\!\in\!(\Z^{\ge0})^{\i}$ is the $r$-th standard coordinate vector, 
then
\begin{equation*}\begin{split}
&\sum_{b'=0}^{b'=r}\!\!\Bigg(\!\! (-1)^{b'}
\frac{\llrr{\la_{g;I};\wt\tau_{\b b'}}}
{(|\b|\!+\!b'\!-\!\mu_g(I)\!-\!m\!-\!1)!}\binom{r}{b'}\!\!\Bigg)
=A_{I,\b;e_r}^{(g)}\\
&\hspace{2in}=
(-1)^r\!\!\!\!\!\!\!\!\!\!
\sum_{\begin{subarray}{c} \un\ep\in (\Z^{\ge0})^m\\
|\un\ep|\le\mu_g(I)+m\end{subarray}} \hspace{-.2in}
A_{I;e_r}^{(g,\un\ep)}
\binom{|\b|\!-\!|\un\ep|}{\mu_g(I)\!+\!m\!+\!1\!-\!r\!-\!|\un\ep|}
\!\!\prod_{k=1}^m\!\!\binom{b_k}{\ep_k}\,.
\end{split}\end{equation*}  
Proposition~\ref{HodgeInt_prp} is established in Section~\ref{HodgeInt_subs};
all numbers $A_{I;\c}^{(0,\un\ep)}$ and some numbers $A_{I;\c}^{(1,\un\ep)}$
are determined in Examples~\ref{AgepIc_eg01} and~\ref{AgepIc_eg1}.
The bound~\e_ref{DMstr_e2} is used in the proof of Theorem~\ref{GWbound_thm}.

\subsection{Asymptotic expansions}
\label{HG_subs}

\noindent
The second input determining the structure coefficients $\nc_{g;\p,\b}^{(d)}$
in Theorem~\ref{main_thm} is the asymptotic expansion of 
the hypergeometric~series 
\BE{Fdfn_e}
F(w,q)\equiv \sum_{d=0}^{\i}q^d 
\frac{w^{n d}}{\prod\limits_{r=1}\limits^{r=d}((w\!+\!r)^n\!-\!w^n)}
\in\Q(w)\big[\big[q\big]\big]\EE
as $w\!\lra\!\i$ provided by Proposition~\ref{Fexp_prp} below.
By \cite{Gi,LLY}, this hypergeometric~series encodes
the 1-marked genus~0 GW-invariants of~$\P^{n-1}$.\\

\noindent
For $n\!\in\!\Z^+$, let
\BE{Ldfn_e0}
L(q)=(1\!+\!q)^{1/n}\in 1+q\Q[[q]] \,.\EE
For $m,j\!\in\!\Z$, we define $\H_{m,j}\in\Q(u)$ recursively by
\begin{equation*}\begin{split}
\H_{m,j}&\equiv0 \quad\hbox{unless}~~0\le j\le m,\qquad
\H_{0,0}\equiv1;\\
\H_{m,j}(u)&\equiv\H_{m-1,j}(u)+
\big(u\!-\!1\big)
\bigg(\frac{\tnd}{\tnd u}+\frac{m\!-\!j}{nu}\bigg)\H_{m-1,j-1}(u)
\quad\hbox{if}~ m\ge1,~0\le j\le m. 
\end{split}\end{equation*}
For example, 
\BE{Hlow_e} 
\H_{m,0}(u)=1, \quad \H_{m,1}(u)=\binom{m}{2}\frac{u\!-\!1}{n u},\quad
\H_{m,2}(u)=\binom{m}{3}\!\bigg(\!\!n+\frac{3m\!-\!5}{4}(u\!-\!1)\!\!\bigg)
\frac{u\!-\!1}{n^2u^2}\EE
for $m\!\ge\!0$.
Finally, we define differential operators $\fL_1,\ldots,\fL_n$ on $\Q[[q]]$ by
\BE{fLkdfn_e}
D=q\frac{\tnd}{\tnd q}\,, \qquad
\fL_k=L^n\sum_{i=0}^k\binom{n}{i}\H_{n-i,k-i}(L^n)D^i\,.\EE
By~\e_ref{Ldfn_e0} and~\e_ref{Hlow_e}, the first two operators~are 
\begin{gather*}\begin{split}
\fL_1&=n L^n D+\frac{n\!-\!1}{2}\big(L^n\!-\!1\big)
=nL^n\Bigg\{L^{\frac{1-n}{2}}DL^{\frac{n-1}{2}}\Bigg\},\\
\fL_2&=\binom{n}{2}L^nD^2+\binom{n\!-\!1}{2}(L^n\!-\!1)D
+\frac{(n\!-\!1)(n\!-\!2)}{24nL^n}
\big((3n\!-\!5)L^n\!+\!n\!+\!5\big)(L^n\!-\!1)\,.
\end{split}\end{gather*}

\begin{prp}[{\cite[Proposition~2.1]{g0ci}}]\label{Fexp_prp}
The power series~$F$ of~\e_ref{Fdfn_e} admits an asymptotic expansion
\BE{F0exp_e} F(w,q)\sim e^{\xi(q)w}
\sum_{b=0}^{\i}\Phi_b(q)w^{-b}
\qquad\hbox{as}\quad w\lra\i,\EE
with $\xi,\Phi_1,\ldots\!\in\!q\Q[[q]]$ and $\Phi_0\!\in\!1\!+\!q\Q[[q]]$ 
determined by the first-order ODEs  
\BE{PhiODE_e}
1+D\xi=L\,,\qquad \fL_1\Phi_b+\frac{1}{L}\fL_2\Phi_{b-1}
+\ldots +\frac{1}{L^{n-1}}\fL_n\Phi_{b+1-n}=0,\EE
where $\Phi_b\equiv0$ for $b\!<\!0$.
\end{prp}

\noindent
For example,
\BE{F0exp_e2}\begin{split}
\Phi_0(q)&= L(q)^{-(n-1)/2}=(1\!+\!q)^{-(n-1)/2n},\\
\Phi_1(q)&=\frac{n\!-\!1}{24n}\Big(n\!+\!(n\!+\!1)L(q)^{-1}\!-\!(2n\!+\!1)L(q)^{-(n+1)}\Big)
\Phi_0(q).
\end{split}\EE

\vspace{.2in}

\noindent
By Proposition~\ref{Fexp_prp}, for each $p\!\in\!\Z^{\ge0}$ there is an asymptotic expansion
\BE{Fpexp_e} 
\big\{1\!+\!w^{-1}D\big\}^pF(w,p) 
\sim e^{\xi(q)w}\sum_{b=0}^{\i}\Phi_{p;b}(q)w^{-b}
\qquad\hbox{as}\quad w\lra\i,\EE
with   $\Phi_{p;0}\!\in\!1\!+\!q\Q[[q]]$ and
$\Phi_{p;1},\Phi_{p;2}\ldots\!\in\!q\Q[[q]]$ described~by
$$\Phi_{0;b}=\Phi_b, \quad
\Phi_{p;-1}\equiv0,\quad
\Phi_{p;b}=L\Phi_{p-1;b}+D\Phi_{p-1;b-1} \qquad\forall\,p\!\in\!\Z^+,\,b\!\in\!\Z^{\ge0}.$$
We set $\Phi_{p;b}\!=\!0$ if $b\!<\!0$.
For example, 
\BE{F0exp_e2b}\Phi_{p;0}=L^p\Phi_0, \qquad
\Phi_{p;1}=L^p\Phi_1-\frac{(n\!-\!p)p}{2n}L^{p-1}\big(1\!-\!L^{-n}\big)\Phi_0;\EE
the second identity above follows by induction from the first one
along with~\e_ref{Ldfn_e0} and~\e_ref{F0exp_e2}.\\

\noindent
For $g\!\in\!\Z^{\ge0}$, let $(C_{g;n;I})_{I\in(\Z^{\ge0})^g}$ be a tuple of integers 
such~that 
\BE{CgIdfn_e0}
e\big(\bE_g^*\!\otimes\!T\P^{n-1}\big)
=\sum_{I\in(\Z^{\ge0})^g}\!\!\!\!\!C_{g;n;I}\la_{g;I}H^{(n-1)g-\|I\|}
\in H^*\big(\ov\cM_{g,m}\!\times\!\P^{n-1}\big)\EE
for all $m\!\in\!\Z^{\ge0}$ with $2g\!+\!m\!\ge\!3$.
Let  $g,m\!\in\!\Z^{\ge0}$ with $2g\!+\!m\!\ge\!3$, $I\!\in\!(\Z^{\ge0})^g$,  
$\c\!\in\!(\Z^{\ge0})^{\i}$, and \hbox{$\un\ep\!\in\!(\Z^{\ge0})^m$}. 
With $A_{I;\c}^{(g,\un\ep)}\!\in\!\Q$ as in Proposition~\ref{HodgeInt_prp},
define
\begin{gather}
\label{hatAdfn_e}
\wh{A}_{I;(c_1,c_2,\ldots)}^{(g,\un\ep)}=
A_{I;(0,c_1,c_2,\ldots)}^{(g,\un\ep)}\prod_{k=1}^m\!\frac{1}{\ep_k!}\,,\\
\label{PsimiCdfn_e}
\Phi_{I;\c}^{(g,\un\ep)}(q)=
\frac{(-1)^{\mu_g(I)+m+|\c|}C_{g;n;I}\wh{A}_{I;\c}^{(g,\un\ep)}}{\Phi_0(q)^{2g-2}}
\prod_{r=1}^{\i}\frac{1}{c_r!}\!\bigg(\frac{\Phi_r(q)}{(r\!+\!1)!\,\Phi_0(q)}\!\bigg)^{\!\!c_r}\,.
\end{gather}

\begin{eg}\label{PhiIc_eg}
Using that the rank of $\bE_g$ is~$g$ and Euler's sequence for~$\P^{n-1}$, we obtain
\BE{CgnI_e} C_{0;n;()}=1, \qquad C_{1;n;(0)}=n, 
\qquad C_{1;n;(1)}=-\frac{(n\!-\!1)n}{2}\,.\EE 
For $\0\!\in\!(\Z^{\ge0})^m$, these statements and Example~\ref{AgepIc_eg01} give
\begin{equation*}\begin{split}
\Phi_{();\c}^{(0,\0)}(q)&=(-1)^{m-3+|\c|}\big(m\!-\!3\!-\!|\c|\big)!
\Phi_0(q)^2\prod_{r=1}^{\i}\frac{1}{c_r!}
\!\bigg(\frac{\Phi_r(q)}{(r\!+\!1)!\,\Phi_0(q)}\!\bigg)^{\!\!c_r},\\
\Phi_{(1);\c}^{(1,\un\0)}(q)&=(-1)^{m+|\c|}
\frac{(n\!-\!1)n(m\!-\!1\!-\!|\c|)!}{48}\prod_{r=1}^{\i}\frac{1}{c_r!}
\!\bigg(\frac{\Phi_r(q)}{(r\!+\!1)!\,\Phi_0(q)}\!\bigg)^{\!\!c_r}\,.
\end{split}\end{equation*}
All other power series $\Phi_{I;\c}^{(g,\un\ep)}$ with $g\!=\!0$ and 
$(g,I)\!=\!(1,(1))$ vanish.
The first expression above equals $\Phi_{m-3,\c}(q)$ in \cite[(2.30)]{g0ci}.
For $\un\ep\!\in\!\{0,1\}^m$, Example~\ref{AgepIc_eg1} gives
$$\Phi_{(0);\0}^{(1,\un\ep)}(q)=
(-1)^m \frac{n(m\!-\!|\un\ve|)!}{24}\cdot\begin{cases}
0,&\hbox{if}~|\un\ve|\!=\!0;\\
1,&\hbox{if}~|\un\ve|\!=\!1;\\
-(|\un\ve|\!-\!2)!,&\hbox{if}~|\ve|\!\ge\!2.\end{cases}$$
If in addition $r\!\in\!\Z^+$ and $e_r\!\in\!(\Z^{\ge0})^{\i}$ is 
the $r$-th standard coordinate vector, then
$$\Phi_{(0);e_r}^{(1,\un\ep)}(q)=
(-1)^{m+1} \frac{n(m\!-\!|\un\ve|)!}{24(r\!+\!1)!}
\frac{\Phi_r(q)}{\Phi_0(q)}\cdot\begin{cases}
(r\!+\!1),&\hbox{if}~|\un\ve|\!=\!0;\\
(m\!-\!1\!-\!r),&\hbox{if}~|\un\ve|\!=\!1;\\
-(|\un\ve|\!-\!2)!((|\un\ve|\!-\!1)r\!+\!m),
&\hbox{if}~|\ve|\!\ge\!2.
\end{cases}$$
The power series $\Phi_{I;\c}^{(1,\un\ep)}$  
with $\un\ep\!\not\in\!\{0,1\}^m$ for any $m\!\in\!\Z^+$ vanish.
\end{eg}

\subsection{The structure coefficients}
\label{Mainform_subs}

\noindent
We now describe the coefficients $\nc_{g;\p,\b}^{(d)}$ in~\e_ref{mainthm_e}
explicitly in two ways.
The first description provides a closed formula for these coefficients
as sums over connected trivalent $N$-marked genus~$g$ graphs;
see~\e_ref{ncCdfn2_e}.
The second description provides a recursive definition of these coefficients
which reduces the value of $3g\!+\!N$ (or equivalently of the dimension of~$\ov\cM_{g,N}$)
with the base case provided by~\e_ref{coeff2_e}, when this value is~2;
see~\e_ref{coeffdfn_e}.
We also show that these coefficients satisfy
\BE{ncvan_e}
\nc_{g;\p,\b}^{(d)}\neq0 \quad\Lra\quad
|\b|\le 3(g\!-\!1)\!+\!N, ~~
|\p|\!-\!|\b|\!+\!nd=(n\!-\!4)(g\!-\!1)\!+\!(n\!-\!2)N.\EE
It is fairly straightforward to see that the two descriptions are equivalent.
This also follows from the two variations of the main localization computation
in Sections~\ref{ClFormComp_subs} and~\ref{RecFormComp_subs}.
For a ring~$R$, \hbox{$\Phi\!\in\!R[[q]]$}, and $d\!\in\!\Z$, let
$$\blrbr{\Phi}_{q;d}\in R$$
denote the coefficient of~$q^d$ (with $\lrbr{\Phi}_{q;d}\!\equiv\!0$ if $d\!<\!0$).\\

\noindent
Let $S$ be a finite set.
An \sf{$S$-marked graph} is a tuple
\BE{GaNdfn_e} \Ga\equiv 
\big(\g\!:\!\Ver\!\lra\!\Z^{\ge0},
\eta\!:S\!\sqcup\!\Fl\!\lra\!\Ver,\Edg\big),\EE
where $\Ver$ and $\Fl$ are finite sets (of \sf{vertices} and \sf{flags}, respectively)
and $\Edg$ is a partition of $\Fl$ into subsets~$e$ with $|e|\!=\!2$.
For $N\!\in\!\Z^{\ge0}$, an \sf{$N$-marked graph} is an $[N]$-marked graph. 
Figure~\ref{trivalent_fig} depicts some 2-marked graphs~$\Ga$,
representing each vertex of~$\Ga$ by a~dot and each edge by a curve between its vertices.
The number next to a vertex~$v$, if any, is~$\g(v)$; we omit it if $\g(v)\!=\!0$.
The elements of the~set $[N]\!=\![2]$ are shown in bold face and 
are linked by line segments to their images under~$\eta$.\\

\noindent
An \sf{equivalence} between an $S$-marked graph as in~\e_ref{GaNdfn_e}
and another $S$-marked graph
$$ \Ga'\equiv \big(\g'\!:\!\Ver'\!\lra\!\Z^{\ge0},
\eta'\!:S\!\sqcup\!\Fl'\!\lra\!\Ver',\Edg'\big)$$
is a pair of bijections $h_{\Ver}\!:\Ver\!\lra\!\Ver'$ and $h_{\Fl}\!:\Fl\!\lra\!\Fl'$
such that  
$$\g=\g'\!\circ\!h_{\Ver}, \quad 
h_{\Ver}\!\circ\!\eta|_S=\eta'|_S, \quad
h_{\Ver}\!\circ\!\eta|_{\Fl}=\eta'\!\circ\!h_{\Fl},\quad
h_{\Fl}(e)\in\Edg'~\forall\,e\!\in\!\Edg.$$
We denote by $\Aut(\Ga)$ the group of \sf{automorphisms}, i.e.~self-equivalences, of~$\Ga$.
For example, \hbox{$\Aut(\Ga)\!=\!2$} for the second and third graphs in Figure~\ref{trivalent_fig}.\\

\noindent
For $\Ga$ as in~\e_ref{GaNdfn_e} and $f\!\in\!\Fl$, 
we denote by $e_f\!\in\!\Edg$ the unique element of~$\Edg$
containing~$f$.
For each $v\!\in\!\Ver$, let
\begin{gather*}
g_v=\g(v), \quad S_v=S\!\cap\!\eta^{-1}(v), ~~
\Fl_v(\Ga)=\Fl\!\cap\!\eta^{-1}(v), ~~
\ov\Fl_v(\Ga)=\eta^{-1}(v)\subset S\!\sqcup\!\Fl, \\
\val_{\Ga}(v)=2\big(g_v\!-\!1)+\big|\ov\Fl_v(\Ga)\big|,\quad
m_v(\Ga)=3\big(g_v\!-\!1\big)+\big|\ov\Fl_v(\Ga)\big|.
\end{gather*}
A vertex $v\!\in\!\Ver$ of~$\Ga$ is \sf{trivalent}~if $\val_{\Ga}(v)\!>\!0$.
The graph~$\Ga$ is \sf{trivalent} if all its vertices are trivalent.\\

\begin{figure}
\begin{pspicture}(7,-1.1)(10,1.5)
\psset{unit=.3cm}
\pscircle*(30,0){.3}\rput(29.2,-.2){\smsize{$1$}}
\psline[linewidth=.04](30,0)(32,2)\rput(32.5,2){\smsize{$\bf 1$}}
\psline[linewidth=.04](30,0)(32,-2)\rput(32.5,-2){\smsize{$\bf 2$}}
\pscircle(38,0){1}\pscircle*(39,0){.3}
\psline[linewidth=.04](39,0)(41,2)\rput(41.5,2){\smsize{$\bf 1$}}
\psline[linewidth=.04](39,0)(41,-2)\rput(41.5,-2){\smsize{$\bf 2$}}
\pscircle*(52,0){.3}\pscircle*(48,0){.3}
\psline[linewidth=.04](48,0)(46,-2)\rput(45.5,-2){\smsize{$\bf 2$}}
\psline[linewidth=.04](52,0)(54,2)\rput(54.5,2){\smsize{$\bf 1$}}
\pnode(52,0){A}\pnode(48,0){B}
\nccurve[ncurv=.7,nodesep=.1,angleA=210,angleB=-30]{-}{A}{B}
\nccurve[ncurv=.7,nodesep=.1,angleA=150,angleB=30]{-}{A}{B}
\pscircle*(62,0){.3}\pscircle*(58,0){.3}\psline[linewidth=.06](58,0)(62,0)
\psline[linewidth=.04](62,0)(64,-2)\rput(64.5,-2){\smsize{$\bf 2$}}
\psline[linewidth=.04](62,0)(64,2)\rput(64.5,2){\smsize{$\bf 1$}}
\rput(57.2,-.2){\smsize{$1$}}
\pscircle(68,0){1}
\pscircle*(73,0){.3}\pscircle*(69,0){.3}\psline[linewidth=.06](69,0)(73,0)
\psline[linewidth=.04](73,0)(75,-2)\rput(75.5,-2){\smsize{$\bf 2$}}
\psline[linewidth=.04](73,0)(75,2)\rput(75.5,2){\smsize{$\bf 1$}}
\end{pspicture}
\caption{The trivalent 2-marked genus~1 graphs}
\label{trivalent_fig}
\end{figure}
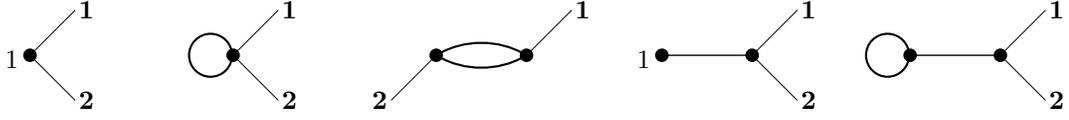

\noindent
A graph~$\Ga$  as in~\e_ref{GaNdfn_e} is \sf{connected} if for all $v,v'\!\in\!\Ver$ distinct 
there exist
\begin{gather*}
m\!\in\!\Z^+,~f_1^-,f_1^+,\ldots,f_m^-,f_m^+\!\in\!\Fl\qquad\hbox{s.t.}\\
\eta\big(f_1^-\big)\!=\!v,~\eta\big(f_m^+\big)\!=\!v',~
\eta\big(f_i^+\big)\!=\!\eta\big(f_{i+1}^-\big)~\forall\,i\!\in\![m\!-\!1],~
e_{f_i^-}=e_{f_i^+}~\forall\,i\!\in\![m].
\end{gather*}
For a connected graph~$\Ga$ as in~\e_ref{GaNdfn_e}, we define
\BE{gGAcond_e}
g_{\Ga}\equiv1\!-\!|\Ver|\!+\!|\Edg|\in\Z^{\ge0} \qquad\hbox{and}\qquad
\fa(\Ga)\equiv g_{\Ga}+\sum_{v\in\Ver}\!\!\g(v)\EE
to be the \sf{arithmetic genus} of the underlying graph without the map~$\g$
and the \sf{arithmetic genus} of the graph~$\Ga$ itself, respectively.
For such a graph,
\BE{mvsum_e}\sum_{v\in\Ver}\!\!|\ov\Fl_v(\Ga)|=2\,|\Edg|+|S|, \qquad 
\sum_{v\in\Ver}\!\!m_v(\Ga)+|\Edg|=3\big(\fa(\Ga)\!-\!1\big)+|S|;\EE
the first equality above does not depend on $\Ga$ being connected.
For $g,N\!\in\!\Z^{\ge0}$, we denote by $\cA_{g,N}$ 
the set of (equivalence classes of) connected trivalent $N$-marked genus~$g$  graphs. 
The two elements of~$\cA_{1,1}$ and the five elements of~$\cA_{1,2}$ 
are shown in Figures~\ref{g1N1_fig} and~\ref{trivalent_fig}, respectively.\\

\noindent
Let $\Ga$ be an $N$-marked graph as in~\e_ref{GaNdfn_e} with $S\!=\![N]$.
For $\b\!\in\!(\Z^{\ge0})^{\ov\Fl_v(\Ga)}$ and \hbox{$f\!\in\!\ov\Fl_v(\Ga)$},
we denote by $b_f\!\in\!\Z^{\ge0}$ the component of~$\b$ corresponding to~$f$.
For 
$$\b\in\big(\Z^{\ge0}\big)^{\ov\Fl(\Ga)}, \qquad  
\c\in \prod_{v\in\Ver}\!\!\!\!\big(\Z^{\ge0}\big)^{\i}, 
\qquad\hbox{and}\qquad
\bfI\in \prod_{v\in\Ver}\!\!\!\!\big(\Z^{\ge0}\big)^{g_v}\,,$$
we similarly denote their respective components by
$$b_f\!\in\!\Z^{\ge0}~~\hbox{for}~f\!\in\!\ov\Fl(\Ga) \quad\hbox{and}\quad
\b_v\!\in\!(\Z^{\ge0})^{\ov\Fl_v(\Ga)},~
\c_v\!\in\!(\Z^{\ge0})^{\i},~I_v\!\in\!(\Z^{\ge0})^{g_v}
~~\hbox{for}~v\!\in\!\Ver\,.$$
Define
\BE{wtcAstaredfn_e}\begin{split}
\wt\cA^{\star}(\Ga)&=\big\{(\b,\un\ep,\c,\bfI)\!\in\!
(\Z^{\ge0})^{\ov\Fl(\Ga)}\!\!\times\!\!(\Z^{\ge0})^{\ov\Fl(\Ga)}\!\!\times\!\!\!
\prod_{v\in\Ver}\!\!\!\!\big(\Z^{\ge0}\big)^{\i}
\!\!\times\!\!\!
\prod_{v\in\Ver}\!\!\!(\Z^{\ge0})^{g_v}\!:\\
&\hspace{1.5in}
|\b_v|\!+\!|\un\ep_v|\!+\!\|\c_v\|\!=\!\mu_{g_v}(I_v)\!+\!|\ov\Fl_v(\Ga)|~\forall\,v\!\in\!\Ver\big\}.
\end{split}\EE

\vspace{.15in}

\noindent
Choose a labeling of the flags so that each edge~$e$ of~$\Ga$ contains
flags with opposite signs.
For each vertex $v\!\in\!\Ver$ of~$\Ga$, denote~by
$$\Fl_v^{\pm}(\Ga)\subset\Fl_v(\Ga)$$ 
the subset of positive/negative flags at~$v$.
For $p\!\in\!\nset$, let 
$$\wh{p}=n\!-\!1\!-\!p \in \nset\,.$$
For $d\!\in\!\Z^{\ge0}$, $\p\!\in\!\nset^N$, and $\b\!\in\!(\Z^{\ge0})^N$, 
we define
\BE{cSGadfn_e}
\cS_{\Ga}(d,\p,\b)\subset (\Z^{\ge0})^{\Ver}\times \nset^{\Edg}\times(\Z^{\ge0})^{\Edg}\EE
to be the subset of triples $(\bfd,\p',\b')$ such~that
\begin{gather}\notag
\sum_{v\in\Ver}\!\!d_v=d,\\
\label{sumcond_e2}
\begin{split}
&\sum_{s\in S_v}\!\!\big(\wh{p}_s\!+\!b_s\big)
+\!\!\sum_{f\in\Fl_v^-(\Ga)}\!\!\!\!\!\!\big(p_{e_f}'\!+\!b_{e_f}'\big)
+\!\!\sum_{f\in\Fl_v^+(\Ga)}\!\!\!\!\!\!\big(\wh{p}_{e_f}'\!-\!1\!-\!b_{e_f}'\big)\\
&\hspace{1in}
=(n\!-\!4)(1\!-\!g_v)+\big|\ov\Fl_v(\Ga)\big|+n d_v \qquad\forall~v\!\in\!\Ver.
\end{split}
\end{gather}
By~\e_ref{gGAcond_e}  and the first statement in~\e_ref{mvsum_e}, 
this set is empty unless the second property on the right-hand side of~\e_ref{ncvan_e} 
with $g\!=\!\fa(\Ga)$  holds.\\

\noindent
For $(\p,\b)\!\in\nset^N\!\times\!(\Z^{\ge0})^N$ and $d\!\in\!\Z^{\ge0}$, let
\BE{ncCdfn2_e}\begin{split}
&\nc_{g;\p,\b}^{(d)} =
\sum_{\Ga\in\cA_{g,N}}\!\frac{1}{|\Aut(\Ga)|}
\sum_{\begin{subarray}{c}(\bfd,\p',\b')\in\cS_{\Ga}(d,\p,\b)\\
(\b'',\un\ep,\c,\bfI)\in\wt\cA^{\star}(\Ga)\end{subarray}}  
\hspace{-.3in}
(-1)^{|\b|+|\b'|}\!\!\!
\prod_{v\in\Ver}\!\\
&\hspace{.2in}
\left\llbracket \!\Phi_{I_v;\c_v}^{(g_v,\un\ep_v)}(q)\!\!
\prod_{s\in S_v}\!\frac{\Phi_{\wh{p}_s;b_s''+\ep_s-b_s}\!(q)}{b_s''!\,\Phi_0(q)}
\prod_{f\in\Fl_v^-(\Ga)}
\!\!\!\!\!\!\frac{\Phi_{p_{e_f}';b_f''+\ep_f-b_{e_f}'}\!(q)}{b_f''!\Phi_0(q)}
\prod_{f\in\Fl_v^+(\Ga)}\!\!\!\!\!
\frac{\Phi_{\wh{p}_{e_f}';b_f''+\ep_f+1+b_{e_f}'}\!(q)}{b_f''!\,\Phi_0(q)}
\right\rrbracket_{q;d_v}\,.
\end{split}\EE
Since 
$$\sum_{s\in S_v}\!b_s\le 
\sum_{s\in S_v}\!\big(b_s''\!+\!\ep_s\big)\le 3(g_v\!-\!1)+\big|\ov\Fl_v(\Ga)\big|
\equiv m_v(\Ga)$$
if the corresponding factor on the second line of~\e_ref{ncCdfn2_e} is nonzero, 
the second identity in~\e_ref{mvsum_e}  implies the first property in~\e_ref{ncvan_e}.\\

\noindent
We next give a recursive description of the coefficients 
$\nc_{g;\p,\b}^{(d,0)}\!\equiv\!\nc_{g;\p,\b}^{(d)}$ appearing in Theorem~\ref{main_thm}.
For $N\!\in\!\Z^{\ge0}$, a tuple $\p\!\equiv\!(p_s)_{s\in[N]}$ of integers, and $S\!\subset\![N]$, 
let $\p|_S$ be the $S$-tuple consisting of the elements of~$\p$ indexed by~$S$ and
$$|\p|_S\equiv \big|\p|_S\big|\equiv\sum_{s\in S}p_s\,.$$

\vspace{.2in}

\noindent
For $m\!\in\!\Z^{\ge0}$, let $\cP_{g,N}^{(m)}$ denote the collection 
of tuples $(g',(g_i,S_i,N_i)_{i\in[m]})$ such~that
\begin{gather*}
g'\!\in\!\Z^{\ge0}, \quad
g_i\!\in\!\Z^{\ge0},~N_i\!\in\!\Z^+~~\forall\,i\!\in\![m],\quad
[N]=\bigsqcup_{i=1}^{i=m}\!S_i,  \quad
2g'\!+\!\sum_{i=1}^{i=m}\!N_i\ge3,\\
(0,\{N\},1)\in\big\{(g_i,S_i,N_i)\!:i\!\in\![m]\big\}, \quad
g_i\!+\!|S_i|\!+\!|N_i|\ge2~~\forall\,i\!\in\![m],
\quad
g'\!+\!\sum_{i=1}^{i=m}(g_i\!+\!N_i)=g\!+\!m\,.
\end{gather*}
We write elements of $\cP_{g,N}^{(m)}$ as tuples $(g',\bfg,\bfS,\bfN)$
and denote the components of the $m$-tuples $\bfg$, $\bfS$, and~$\bfN$ by 
$g_i$, $S_i$, and~$N_i$, respectively.
We note that
$$3g_i\!+\!|S_i|\!+\!N_i\!<3g\!+\!N \qquad\forall\,i\!\in\![m]$$
if $(g',(g_i,S_i,N_i)_{i\in[m]})$ is an element of $\cP_{g,N}^{(m)}$.

\begin{rmk}\label{graphbr_rmk}
Let $\Ga$ be an $N$-marked genus~$g$ graph as in~\e_ref{GaNdfn_e} with $S\!=\![N]$
and $v\!\equiv\!\eta(N)$ be the vertex to which the last marked point is attached.
Breaking~$\Ga$ at~$v$ and replacing the flags at~$v$ with marked points, 
we~obtain $m$~connected strands similarly to Figure~\ref{strands_fig} 
on page~\pageref{strands_fig} for some $m\!\in\!\Z^+$.
Each of the strands is of some genus~$g_i$, carries the subset $S_i\!\subset\![N]$
of the original marked points, and $N_i\!\in\!\Z^+$ additional marked points
arising from the flags at~$v$.
Then,
$$\big(\g(v),(g_i,S_i,N_i)_{i\in[m]}\big)\in\cP_{g,N}^{(m)}\,.$$
Furthermore, every element of $\Ga$ describes a type of collections of strands
obtained by breaking an $N$-marked genus~$g$ graph~$\Ga$ at the vertex 
$v\!\equiv\!\eta(N)$ to which the last marked point is attached.
\end{rmk}

\noindent
For $m$, $g'$, and $\bfN$ as above, define 
\BE{Phimcdfn_e0}\begin{split}
\wt\cA^{\star}_{g',\bfN}&=\big\{(\b,\un\ep,\c,I)\!\in\!
\prod_{i=1}^{i=m}\!(\Z^{\ge0})^{N_i}\!\!\times\!\!
\prod_{i=1}^{i=m}\!(\Z^{\ge0})^{N_i}\!\!\times\!\!\big(\Z^{\ge0}\big)^{\i}
\!\!\times\!\!(\Z^{\ge0})^{g'}\!\!:\\
&\hspace{2.5in} |\b|\!+\!|\un\ep|\!+\!\|\c\|\!=\!\mu_{g'}(I)\!+\!|\bfN|\big\}.
\end{split}\EE
If in addition $d,t\!\in\!\Z$, let
\begin{gather}
\label{Phimcdfn_e}
\cS_{g',\bfN}(d,t)=\big\{(\p,\b)\!\in\!
\prod_{i=1}^{i=m}\!\nset^{N_i}\!\times\!\prod_{i=1}^{i=m}\!\Z^{N_i}\!:\,
|\p|\!-\!|\b|=(n\!-\!4)(1\!-\!g')\!+\!2|\bfN|\!+\!n(d\!+\!t)\big\}\,;
\end{gather}
the number $|\b|$ above denotes the sum of the three components of~$\b$
(not of their absolute values).
For tuples $\p$ and $\b$ as above and $i\!\in\![m]$, we denote
by $\p_i\!\in\!\nset^{N_i}$ and $\b_i\!\in\!\Z^{N_i}$ 
the $i$-th components of~$\p$ and~$\b$.
If in addition $f\!\in\![N_i]$, let $p_{i;f}\!\in\!\nset$ and $b_{i;f}\!\in\!\Z$
denote the $f$-th components of~$\p_i$ and $\b_i$, respectively.\\

\noindent
For any $p,p'\!\in\!\nset$ and $b,b',d,t\!\in\!\Z$, let 
\BE{coeff2_e}\nc_{0;(p,p'),(b,b')}^{(d,t)}=
\begin{cases}
(-1)^b\,,
&\hbox{if}~b\!\ge\!0,\,b\!+\!b'\!=\!-1,\,d,t\!=\!0,\,p\!+\!p'\!=\!n\!-\!1;\\
0,&\hbox{otherwise}.
\end{cases}\EE
For all $g\!\in\!\Z^{\ge0}$, $N\!\in\!\Z^+$ with $2g\!+\!N\!\ge\!3$,
$N$-tuples $\p\!\in\!\nset^N$ and $\b\!\in\!(\Z^{\ge0})^N$,
and $d,t\!\in\!\Z^{\ge0}$, we inductively define
\BE{coeffdfn_e}\begin{split}
\nc_{g;\p,\b}^{(d,t)}&=\!\!\!
\sum_{\begin{subarray}{c}m,d'\in\Z^{\ge0}\\ t'\in\Z\end{subarray}}
\!\frac{1}{m!} \!\!\!\!\!
\sum_{\begin{subarray}{c}(g',\bfg,\bfS,\bfN)\in\cP_{g,N}^{(m)}\\
(\b'',\un\ep,\c,I)\in\wt\cA^{\star}_{g',\bfN}\\
(\p',\b')\in\cS_{g',\bfN}(d',t')\end{subarray}} 
\sum_{\begin{subarray}{c}\bfd,\bft\in(\Z^{\ge0})^m\\
|\bfd|=d-d',|\bft|=t-t'\end{subarray}}  \!\!\!\!
\Bigg(\prod_{i=1}^{i=m}\frac{\nc_{g_i;\p|_{S_i}\p_i',\b|_{S_i}\b_i'}^{(d_i,t_i)}}{N_i!}\\
&\hspace{1.5in}
\times\Bgbr{\Phi_{I;\c}^{(g',\un\ep)}(q)\!\prod_{i=1}^{i=m}\prod_{f\in[N_i]}\!\!
\frac{\Phi_{p_{i;f}';b_{i;f}''+\ep_{i;f}+1+b_{i;f}'}(q)}
{b_{i;f}''!\Phi_0(q)}\!}_{q;d'}\Bigg),
\end{split}\EE
with $\nc_{g_i;\p|_{S_i}\p_i',\b|_{S_i}\b_i'}^{(d_i,t_i)}\!\equiv\!0$ if 
$b_{i;f}'\!<\!0$ for some $f\!\in\![N_i]$ and 
$2g_i\!+\!|S_i|\!+\!N_i\!\ge\!3$
(if the last sum is~2, $\nc_{g_i;\p|_{S_i}\p_i',\b|_{S_i}\b_i'}^{(d_i,t_i)}$ is given
by~\e_ref{coeff2_e}).\\

\noindent
If the summand in~\e_ref{coeffdfn_e} above does not vanish, then
$$-b_{i;f}'\le b_{i;f}''\!+\!\ep_{i;f}\!+\!1
~~\forall\,f\!\in\![N_i],\,i\!\in\![m] \quad\Lra\quad
-|\b'|\le 3g'\!-\!3\!+\!2|\bfN|;$$
the last implication makes use of the condition $(\b'',\un\ep,\c,I)\in\wt\cA^{\star}_{g',\bfN}$.
By induction, the non-vanishing coefficients $\nc_{g;\p,\b}^{(d,t)}$ thus satisfy 
the bound in~\e_ref{ncvan_e}.
Furthermore,
\BE{ncvan_e2}
\nc_{g;\p,\b}^{(d,t)}\neq0 \quad\Lra\quad
|\p|\!-\!|\b|\!+\!n(d\!+\!t)=(n\!-\!4)(g\!-\!1)\!+\!(n\!-\!2)N.\EE
Thus, the coefficients $\nc_{g;\p,\b}^{(d,t)}\!\equiv\!\nc_{g;\p,\b}^{(d,0)}$
defined by~\e_ref{coeff2_e} and~\e_ref{coeffdfn_e} satisfy~\e_ref{ncvan_e}.\\

\noindent
Since $\Phi_{I;\c}^{(g',\un\ep)}\!\in\!q\Q[[q]]$  unless $\c\!=\!\0$ and
$\Phi_{p';b'}\!\in\!q\Q[[q]]$  unless \hbox{$b'\!=\!0$}, 
all nonzero contributions in the $d\!=\!0$ case of~\e_ref{coeffdfn_e} arise
from the elements $(g',\bfg,\bfS,\bfN)$ of $\cP_{g,N}^{(m)}$ so that 
$g_i\!=\!0$ and $|S_i|,|N_i|\!=\!1$ for every $i\!\in\![m]$.
These conditions imply that $g'\!=\!g$ and $m\!=\!N$.
By the first statement after~\e_ref{Fpexp_e} and~\e_ref{PsimiCdfn_e},
$$\blrbr{\Phi_{p';0}}_{q;0}=1~\forall\,p'\!\in\!\nset, \quad
\blrbr{\Phi_{I;\0}^{(g,\un\ep)}(q)}_{q;0}=
(-1)^{\mu_g(I)+N}C_{g;n;I}\wh{A}_{I;\0}^{(g,\un\ep)}
~\forall\,I\!\in\!(\Z^{\ge0})^g,\,\un\ep\!\in\!(\Z^{\ge0})^N.$$
Combining these observations with~\e_ref{coeffdfn_e} and~\e_ref{coeff2_e}, 
we obtain 
\BE{coeffnon0_e2}\begin{split}
\nc_{g;\p,\b}^{(0,t)}&=
\sum_{(\b'',\un\ep,\0,I)\in\wt\cA^{\star}_{g,[N]}}
\sum_{\begin{subarray}{c}\p'\in\nset^N\\
|\p'|=(n-1)(1-g)+nt+\|I\|\end{subarray}}
 \hspace{-.55in}(-1)^{\mu_g(I)+N}C_{g;n;I}\wh{A}_{I;\0}^{(g,\un\ep)}
\prod_{s=1}^{s=N}(-1)^{b_s}
\frac{\de_{\wh{p}_s,p_s'}\de_{b_s,b_s''+\ep_s}}{b_s''!}\\
&=\sum_{\begin{subarray}{c}I\in(\Z^{\ge0})^g\\ \mu_g(I)=|\b|-N \end{subarray}}
\!\!\!\!\!\!\!\!\!\!
\de_{|\p|+\|I\|+nt,(n-1)(N-1+g)} \!\!\!\!\!\!\!
\sum_{\un\ep\in(\Z^{\ge0})^N}\!\!\!\!\!\!C_{g;n;I}A_{I;\0}^{(g,\un\ep)}
\prod_{s=1}^{s=N}\frac{1}{b_s!}\binom{b_s}{\ep_s}\\
&=\de_{|\p|-|\b|+nt,(n-4)(g-1)+(n-2)N}\!\!\!\!\!\!\!\!\!\!
\sum_{\begin{subarray}{c}I\in(\Z^{\ge0})^g\\ 
|\b|+\|I\|=3(g-1)+N\end{subarray}}
\hspace{-.42in} C_{g;n;I}\llrr{\la_{g;I};\tau_{\b}}
\end{split}\EE
where $\de_{a,b}$ is the Kronecker delta function
(equals~1 if $a\!=\!b$ and~0 otherwise);
the last equality above follows from~\e_ref{AgIb0_e2}.
The $t\!=\!0$ case of the above statement can also be obtained 
from~\e_ref{ncCdfn2_e}.\\

\noindent
The coefficients $\nc_{g;\p,\b}^{(d)}$ must be invariant under the permutations of
$[N]$ (same permutations in the components of~$\p$ and~$\b$).
This is not apparent from either of the above two descriptions of these coefficients,
even in the extremal cases;
thus, this is a consequence of the proofs of Theorem~\ref{main_thm}
in Section~\ref{maincomp_sec}.

\subsection{Proof of Theorem~\ref{g1N1_thm}}
\label{g1N1pf_subs}

\noindent 
The statement of Theorem~\ref{g1N1_thm} is equivalent~to
\BE{g1N1_e0}
\sum_{p=0}^{n-1}
 \bblr{\frac{H^{n-1-p}}{\hb\!-\!\psi_1}}^{\!\!\P^{n-1}}_{\!\!1,d}\!\!\!\!H^p
=\frac{n}{24}\frac{(H\!+\!d\hb)^{n-1}}{\hb^2\!\prod\limits_{r=1}^d\!(H\!+\!r\hb)^n}
-\frac{(n\!-\!1)n}{48}\frac{(H\!+\!d\hb)^{n-2}}{\hb\!\prod\limits_{r=1}^d\!(H\!+\!r\hb)^n},\EE
with the identity holding modulo $H^n$ and as a power series in $\hb^{-1}$.
We deduce this identity from Theorem~\ref{main_thm} below.\\

\noindent
By the defining property of the cohomology pushforward,
the generating function~\e_ref{Zdfn_e} with $(g,N)\!=\!(1,1)$ is given~by
$$Z^{(1)}\big(\hb,H,q\big)=\sum_{d=0}^{\i}\!\bigg(\!q^d
\sum_{p=0}^{n-1}
 \bblr{\frac{H^{n-1-p}}{\hb\!-\!\psi_1}}^{\!\!\P^{n-1}}_{\!\!1,d}\!\!\!\!H^p\!\!\bigg)
\in H^*(\P^{n-1})\big[\hb^{-1}\big]\big[\big[q\big]\big].$$
By~\e_ref{mainthm_e},
\BE{g1N1_e}Z^{(1)}\big(\hb,H,q\big)=\sum_{p\in\nset}
\sum_{b\in\{0,1\}}\sum_{d=0}^{\i}\nc_{1;p,b}^{(d)}q^d\hb^{-b}\De_p(\hb,H,q)\EE 
with the coefficients $\nc_{1;p,b}^{(d)}\!\in\!\Q$ 
determined by either the closed formula~\e_ref{ncCdfn2_e} or 
the recursion~\e_ref{coeffdfn_e}.
In order to compute these coefficients, we will make use~of
\begin{gather}\label{g1N1_e3a}
\Phi_{();\0}^{(0,(\0_3))}=\Phi_0^2, ~~
\Phi_{(1);\0}^{(1,(0))}=-\frac{(n\!-\!1)n}{48},  ~~
\Phi_{(0);\0}^{(1,(0))}=0,  ~~
\Phi_{(0);\0}^{(1,(1))}=-\frac{n}{24},  ~~
\Phi_{(0);e_1}^{(1,(0))}=\frac{n}{24}\frac{\Phi_1}{\Phi_0},\\
\label{g1N1_e3b}
\sum_{p=0}^{n-1}\frac{\Phi_{p;0}}{\Phi_0}\frac{\Phi_{\wh{p};1}}{\Phi_0}
=nL^{n-1}\frac{\Phi_1}{\Phi_0}-\frac{n^2\!-\!1}{12}L^{n-2}\big(1\!-\!L^{-n}\big),
\end{gather}
where $\0_3\!\in\!(\Z^{\ge0})^3$ and $\0\!\in\!(\Z^{\ge0})^{\i}$ are the zero vectors and
$e_1\!=\!(1,0,\ldots)\!\in\!\Z^{\i}$; see Example~\ref{PhiIc_eg} and~\e_ref{F0exp_e2b}.\\

\noindent
The set $\cA_{1,1}$ of connected trivalent 1-marked genus~1 graphs consists of two elements,
$\Ga_0$ and~$\Ga_1$; they are depicted in Figure~\ref{g1N1_fig} along
with the corresponding collections~\e_ref{wtcAstaredfn_e}. 
The associated collections~\e_ref{cSGadfn_e} are 
\begin{equation*}\begin{split}
\cS_{\Ga_0}(d,p,b)&=\begin{cases}
\{(d,\cdot,\cdot)\},&\hbox{if}~d\!=\!0,\,p\!=\!n\!-\!2\!+\!b;\\
\eset,&\hbox{otherwise};
\end{cases}\\
\cS_{\Ga_1}(d,p,b)&=\begin{cases}
\{d\}\!\times\!\nset\!\times\!\Z^{\ge0},&\hbox{if}~d\!=\!0,\,p\!=\!n\!-\!2\!+\!b;\\
\eset,&\hbox{otherwise}.
\end{cases}
\end{split}\end{equation*}
In particular, $\nc_{1;p,b}^{(d)}\!=\!0$ unless $d\!=\!0$ and $p\!=\!n\!-\!2\!+\!b$.
Below we assume that the pair $(p,b)$ satisfies the last condition.\\

\begin{figure}
\begin{pspicture}(0,-1.5)(10,1.5)
\psset{unit=.3cm}
\pscircle*(10,0){.3}\rput(9.2,-.2){\smsize{$1$}}
\psline[linewidth=.04](10,0)(12,2)\rput(12.5,2){\smsize{$\bf 1$}}
\rput(18,1){\begin{tabular}{l} $|\Aut(\Ga_0)|\!=\!1$\\ 
$\Ver\!=\!\{v\}$\\ $S_v\!=\!\{1\}$\\ $\ov\Fl_v(\Ga_0)\!=\!\{1\}$\end{tabular}}
\pscircle(32,0){1}\pscircle*(33,0){.3}
\psline[linewidth=.04](33,0)(35,2)\rput(35.5,2){\smsize{$\bf 1$}}
\rput(42,1){\begin{tabular}{l} $|\Aut(\Ga_1)|\!=\!2$\\ 
$\Ver\!=\!\{v\}$\\ $S_v\!=\!\{1\}$\\ $\ov\Fl_v(\Ga_1)\!=\!\{1,+,-\}$\end{tabular}}
\rput(29,-4){$\wt\cA^{\star}(\Ga_0)\!=\!
\big\{(0,0,e_1,(0)),(0,0,\0,(1)),(0,1,\0,(0)),(1,0,\0,(0))\big\}, ~~
\wt\cA^{\star}(\Ga_1)\!=\!\big\{(\0_3,\0_3,\0,\cdot)\big\}$}
\end{pspicture}
\caption{The connected trivalent 1-marked genus~1 graphs}
\label{g1N1_fig}
\end{figure}
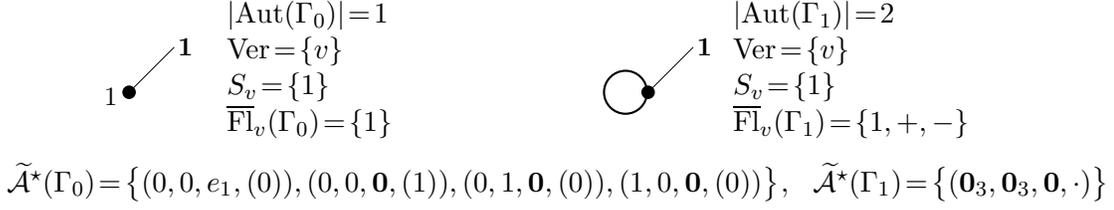

\noindent
By~\e_ref{ncCdfn2_e}, the contribution of an element $(d,p',b')$ 
of $\cS_{\Ga_1}(d,p,b)$ to $\nc_{1;p,b}^{(0)}$ is
\BE{g1N1_e7a}\begin{split}
&\frac{(-1)^{b+b'}}2 
\left\llbracket \!\Phi_{();\0}^{(0,\0_3)}(q)
\frac{\Phi_{\wh{p};-b}(q)}{\Phi_0(q)}
\frac{\Phi_{p';-b'}(q)}{\Phi_0(q)}
\frac{\Phi_{\wh{p}';1+b'}(q)}{\Phi_0(q)}\right\rrbracket_{q;0}\\
&\hspace{1.5in}
=\begin{cases}
\frac12\blrbr{\Phi_{();\0}^{(0,\0_3)}(q)
\frac{\Phi_{1;0}(q)}{\Phi_0(q)}
\frac{\Phi_{p';0}(q)}{\Phi_0(q)}
\frac{\Phi_{\wh{p}';1}(q)}{\Phi_0(q)}}_{q;0},&\hbox{if}~b,b'\!=\!0;\\
0,&\hbox{otherwise}.\end{cases}\end{split}\EE
In particular, $\Ga_1$  does not contribute to $\nc_{1;n-1,1}^{(0)}$.
By~\e_ref{g1N1_e3a}, \e_ref{g1N1_e3b}, and~\e_ref{F0exp_e2}, it contributes
$$\frac{1}{2}\Bgbr{nL(q)\frac{\Phi_1(q)}{\Phi_0(q)}-
\frac{n^2\!-\!1}{12}\big(1\!-\!L(q)^{-n}\big)}_{q;0}=0$$
to $\nc_{1;n-2,0}^{(0)}$.\\

\noindent 
The contribution of an element $(b'',\ep,\c,(i))$ of $\wt\cA^{\star}(\Ga_0)$
to $\nc_{1;p,b}^{(0)}$ is
\BE{g1N1_e7b}
(-1)^b  \left\llbracket \!\Phi_{(i);\c}^{(1,(\ep))}\!(q)
\frac{\Phi_{\wh{p};b''+\ep-b}(q)}{\Phi_0(q)}\right\rrbracket_{q;0}
=\begin{cases}
\blrbr{\Phi_{(i);\c}^{(1,(\ep))}\!(q)
\frac{\Phi_{1;b''+\ep}(q)}{\Phi_0(q)}}_{q;0},&\hbox{if}~b\!=\!0;\\
-\blrbr{\Phi_{0;\0}^{(1,(1))}\!(q)}_{q;0},
&\hbox{if}~b\!=\!1,\,\ep\!=\!1;\\
0,&\hbox{otherwise}.\end{cases}\EE
Along with the conclusion of the previous paragraph, \e_ref{g1N1_e3a}, 
and~\e_ref{F0exp_e2}, this implies~that 
\BE{g1N1_e9}\nc_{1;n-1,1}^{(0)}=\frac{n}{24}, \qquad 
\nc_{1;n-2,0}^{(0)}=-\frac{(n\!-\!1)n}{48}\,.\EE
Along with~\e_ref{g1N1_e}, this establishes~\e_ref{g1N1_e0}.
These two coefficients can also be obtained directly 
from~\e_ref{coeffnon0_e2} and~\e_ref{CgnI_e}.\\

\noindent
We now obtain~\e_ref{g1N1_e9} and the vanishing of the remaining coefficients
$\nc_{1;p,b}^{(d)}$ in~\e_ref{g1N1_e} from the recursion~\e_ref{coeffdfn_e}.
The collections~$\cP_{1,1}^{(1)}$ and $\cP_{1,1}^{(2)}$ contain one element each,
$$\Ga_0\equiv\big(1,(0,\{1\},1)\big) \qquad\hbox{and}\qquad
\Ga_1\equiv\big(0,\big((0,\{1\},1),(0,\eset,2)\!\big)\!\big),$$
respectively ($\cP_{1,1}^{(2)}$ also contains a copy of $\Ga_1$ with the two triplets
swapped); the collections $\cP_{1,1}^{(m)}$ with $m\!\ge\!3$ are empty.
The elements of the collections~$\cP_{1,1}^{(m)}$ precisely correspond
 to the two 1-marked genus~1 graphs depicted in Figure~\ref{g1N1_fig}
(this is because $N\!=\!1$; see Remark~\ref{graphbr_rmk}).
The collections~\e_ref{Phimcdfn_e0} corresponding to~$\Ga_0$ and~$\Ga_1$ are also as 
in the first computation, i.e.
$$\wt\cA^{\star}_{1,(1)}=\wt\cA^{\star}(\Ga_0) \qquad\hbox{and}\qquad
\wt\cA^{\star}_{0,(1,2)}=\wt\cA^{\star}(\Ga_1)\,.$$
The associated collections~\e_ref{Phimcdfn_e} are given by
\begin{gather}
\label{g1N1_e11a}
\cS_{1,(1)}(d,t)=\big\{(p',b')\!\in\!\nset\!\times\!\Z\!:\,
p'\!-\!b'=2\!+\!n(d\!+\!t)\big\}\,,\\
\label{g1N1_e11b}
\cS_{0,(1,2)}(d,t)=\big\{(\p',\b')\!\in\!\big(\nset\!\times\!\nset^2\big)
\!\times\!\big(\Z\!\times\!\Z^2\big)\!:|\p'|\!-\!|\b'|=n\!+\!2\!+\!n(d\!+\!t)\big\}\,,
\end{gather}
with $|\b'|$ above denoting the sum of the three components of~$\b'$
(not of their absolute values).\\

\noindent
We first determine the contribution of $\cP_{1,1}^{(1)}$ in~\e_ref{coeffdfn_e}
with $(g,N)\!=\!(1,1)$.
The first product in~\e_ref{coeffdfn_e} vanishes unless
\BE{g1N1_e12a}d'\!=\!d,\qquad t'\!=\!t,
\qquad b\ge0, \qquad b'=-1\!-\!b, \qquad p'\!=\!n\!-\!1\!-\!p;\EE
if these conditions hold, then this product equals $(-1)^b$. 
Along with the condition on~$(p',b')$ in~\e_ref{g1N1_e11a}, 
this implies that $\cP_{1,1}^{(1)}$ does not contribute to 
$\nc_{1;p,b}^{(d)}\!\equiv\!\nc_{1;p,b}^{(d,0)}$
unless $d\!=\!0$ and \hbox{$p\!=\!n\!-\!2\!-\!b$}.
If these two conditions hold, then the contribution of each element of 
$\wt\cA^{\star}_{1,(1)}\!=\!\wt\cA^{\star}(\Ga_0)$ is given by~\e_ref{g1N1_e7b}.
The sums of these contributions over the four elements of $\wt\cA^{\star}_{1,(1)}$
in the two possibly non-trivial cases are given
by~\e_ref{g1N1_e9}.\\

\noindent
We next determine the contribution of $\cP_{1,1}^{(2)}$ in~\e_ref{coeffdfn_e}
with $(g,N)\!=\!(1,1)$.
We write triples~$\p'$ and~$\b'$ appearing in~\e_ref{g1N1_e11b}
as $(p',p_+,p_-)$ and $(b',b_+,b_-)$, respectively.
The first product in~\e_ref{coeffdfn_e} vanishes unless \e_ref{g1N1_e12a} is satisfied and
$$b_+\ge0, \qquad b_-=-1\!-\!b_+, \qquad p_-\!=\!n\!-\!1\!-\!p_+;$$
if these conditions hold, then this product equals $(-1)^{b+b_+}$. 
Along with the condition on~$(\p',\b')$ in~\e_ref{g1N1_e11b}, 
this implies that $\cP_{1,1}^{(2)}$ does not contribute to 
$\nc_{1;p,b}^{(d)}\!\equiv\!\nc_{1;p,b}^{(d,0)}$
unless $d\!=\!0$ and \hbox{$p\!=\!n\!-\!2\!-\!b$}.
If these two conditions hold, then the contribution of  
$\wt\cA^{\star}_{1,(1,2)}\!=\!\wt\cA^{\star}(\Ga_1)$ is half of~\e_ref{g1N1_e7a}
and thus still vanishes.
This re-establishes~\e_ref{g1N1_e9} and the vanishing of the remaining coefficients
$\nc_{1;p,b}^{(d)}$ in~\e_ref{g1N1_e}.

\subsection{Proofs of Theorems~\ref{GWbound_thm} and~\ref{GWeq0_thm}}
\label{applpf_subs}

\noindent
These two theorems for $g\!=\!0$ are established in~\cite{g0ci}.
Theorem~\ref{GWeq0_thm} is meaningless if $n\!=\!1$,
while Theorem~\ref{GWeq0_thm} in this case is justified immediately after its statement.
Thus, we can assume that $n\!\ge\!2$ and $2g\!+\!N\!\ge\!3$.
We also assume that the numbers in~\e_ref{GWXadfn_e} satisfy
$$c_s\!\in\!\nset~\forall\,s\!\in\![N] \qquad\hbox{and}\qquad
\sum_{s=1}^N(b_s\!+\!c_s)=nd+(n\!-\!4)(1\!-\!g)+N.$$
By~\e_ref{Zdfn_e}, the GW-invariant~\e_ref{GWXadfn_e} is the coefficient of
$$q^d\prod_{s=1}^N \!H^{\wh{c}_s}\hb_s^{-b_s-1}\,, 
\qquad\hbox{where}\quad \wh{c}_s=n\!-\!1\!-\!c_s,$$
on the right-hand side of~\e_ref{mainthm_e}.\\

\noindent
For tuples $\bfd\!\equiv\!(d_s)_{s\in[N]}$ and $\b'\!\equiv\!(b_s')_{s\in[N]}$
in $(\Z^{\ge0})^N$, define
$$\p(\bfd,\b')\in\Z^N \qquad\hbox{by}\qquad
p_s(\bfd,\b')=n d_s+\wh{c}_s-b_s+b_s'\,.$$ 
By~\e_ref{mainthm_e} and~\e_ref{Depdfn_e},
\BE{GWcoeff_e}\begin{split}
&\blr{\tau_{b_1}H^{c_1},\ldots,\tau_{b_N}H^{c_N}}_{g,d}^{\P^{n-1}}\\
&\hspace{.5in}=\sum_{d'=0}^{d'=d}
\sum_{\begin{subarray}{c}\bfd\in(\Z^{\ge0})^N\\ |\bfd|=d-d'\end{subarray}}
\sum_{\begin{subarray}{c}\b'\in(\Z^{\ge0})^N\\
|\b'|\le 3(g-1)+N\end{subarray}}\hspace{-.3in}
\nc_{g;\p(\bfd,\b'),\b'}^{(d')}\!\!
\prod_{s=1}^{s=N}\!\!
\LRbr{\LRbr{F_{p_s(\bfd,\b')}(w,q)}_{q;d_s}}_{w;b_s-b_s'},
\end{split}\EE
with $\lrbr{\ldots}_{q;d}$ denoting the coefficient of~$q^d$ and
$\nc_{g;\p(\bfd,\b'),\b'}^{(d')}\!\equiv\!0$ if $p_s(\bfd,\b')\!\not\in\!\nset$ 
for some $s\!\in\![N]$.\\

\noindent
For any non-vanishing summand on the right-hand side of \e_ref{GWcoeff_e}, 
$$p_s(\bfd,\b)\le n\!-\!1,\quad b_s\!+\!c_s\ge n d_s \qquad\forall~s\!\in\![N].$$
Thus, $d_s\!=\!0$ if $b_s\!+\!c_s\!<\!n$. 
Since the coefficient of~$q^0$ in $\wh{F}_p(w,q)$ is~$1$, it follows that 
$b_s'\!=\!b_s$ and $p_s(\bfd,\b')\!=\!\wh{c}_s$ in such a case.
Since $|\b'|\!\le\!3(g\!-\!1)\!+\!N$, this implies Theorem~\ref{GWeq0_thm}.
By \cite[Corollary~5.3]{g0ci},
\BE{GWcoeff_e4}\bigg|\LRbr{\LRbr{F_{p_s(\bfd,\b')}(w,q)}_{q;d_s}}_{w;b_s-b_s'}
\bigg|\le \frac{C_n^{d_s}}{(nd_s)!}\EE
for some $C_n\!\in\!\R^+$ dependent only on~$n$.
Along with the next statement below, which is proved in the reminder of this
section, \e_ref{GWcoeff_e} and~\e_ref{GWcoeff_e4} imply Theorem~\ref{GWbound_thm}.\\

\noindent
For $\b\!\in\!(\Z^{\ge0})^N$ as above, let
$$\b!=\prod_{s=1}^{s=N}\!b_s, \qquad 
\binom{|\b|}{\b}=\binom{|\b|}{b_1,\ldots,b_N}\,.$$

\begin{prp}\label{Coeffbnd_prp}
Let $g,n\!\in\!\Z^{\ge0}$ with $n\!\ge\!2$.
There exists $C_{n,g}\!\in\!\R$ such that
$$\big|\nc_{g;\p,\b}^{(d)}\big|\le \frac{N!}{\b!}C_{n,g}^{N+d}
\qquad\forall\,d\!\in\!\Z^{\ge0},\,\p\!\in\!\llfloor{n}\rrfloor^N,
\,\b\!\in\!(\Z^{\ge0})^N.$$
\end{prp}

\begin{lmm}\label{Phibnd_lmm}
Let $g\!\in\!\Z^{\ge0}$.
There exists $C_g\!\in\!\R$ such that 
$$\bigg|\LRbr{\Phi_{I;\c}^{(g,\un\ep)}(q)}_{q;d} \bigg|
\le \frac{(3(g\!-\!1)\!+\!m\!+\!|\c|)!}{|\c|!\,\un\ep!}
C_g^{\|\c\|}\!\!\LRbr{(1\!-\!C_gq)^{-\|\c\|-1}}_{q;d}
\binom{|\c|}{\c}\!\!\prod_{r=1}^{\i}\bigg(\frac{1}{r\!+\!1}\bigg)^{\!\!c_r}$$ 
for all $I\!\in\!(\Z^{\ge0})^g$, $\c\!\in\!(\Z^{\ge0})^{\i}$,  
$\un\ep\!\in\!(\Z^{\ge0})^m$, and $m\!\in\!\Z^{\ge0}$
with $2g\!+\!m\!\ge\!3$.
\end{lmm}

\begin{proof}
By \e_ref{hatAdfn_e} and~\e_ref{DMstr_e2}, 
there exists $C_g\!\in\!\R$ such that 
$$\big| \wh{A}_{I;\c}^{(g,\un\ep)} \big|\le 
\frac{(3(g\!-\!1)\!+\!m\!+\!|\c|)!}{\un\ep!}\,C_g2^{|\c|+\|\c\|}$$
for all $I$, $\c$,  $\un\ep$, and $m$ as above.
By~\e_ref{F0exp_e2}, there exists $C_{n,g}\!\in\!\R$ such~that 
$$\Big|\blrbr{\Phi_0(q)^{2-2g}}_{q;d}\Big|
=\Big|\blrbr{(1\!+\!q)^{\frac{(n-1)(g-1)}{n}}}_{q;d}\Big|
\le C_{n,g}^d =\blrbr{(1\!-\!C_{n,g}q)^{-1}}_{q;d}\,.$$
By \cite[Lemma~5.6]{g0ci}, there exists $C_n\!\in\!\R$ such~that 
$$\bigg|\LRbr{\frac{\Phi_r(q)}{\Phi_0(q)}}_{q;d} \bigg|
\le C_n^r \LRbr{(1\!-\!C_nq)^{-r}}_{q;d} r!
\qquad\forall\,r,d\!\in\!\Z^{\ge0}\,. $$
Combining~\e_ref{PsimiCdfn_e} with the above three estimates, 
we obtain the claim. 
\end{proof}

\noindent
For $g,N\!\in\!\Z^{\ge0}$ with $2g\!+\!N\ge\!2$, let 
$$a_{g,N}=\sum_{\Ga\in\cA_{g,N+1}}\frac{1}{|\Aut(\Ga)|}
\prod_{v\in\Ver}\!\!\!\big(3(\g(v)\!-\!1)\!+\!|\ov\Fl_v(\Ga)|\big)!\,,$$
with each $\Ga$ as in the $S\!=\![N\!+\!1]$ case of~\e_ref{GaNdfn_e}.

\begin{lmm}\label{graphcnt_lmm}
There exist $C\!\in\!\R^+$ and  $C_g\!\in\!\R$ for each $g\!\in\!\Z^{\ge0}$ such~that 
\BE{graphcnt_e}a_{g,N}\le C_gC^N N!
\quad\forall\,g,N\!\in\!\Z^{\ge0}~\hbox{with}~2g\!+\!N\ge2.\EE
\end{lmm}

\begin{proof}
We define $a_{0,0}\!=\!0$, $a_{0,1}\!=\!1$, and 
$$f_g(q)=\sum_{N=0}^{\i}\frac{a_{g,N}}{N!}q^N\in \Q[[q]] \qquad\forall\,g\!\in\!\Z^{\ge0}.$$
The claim of the lemma is equivalent to the existence of some $C\!\in\!\R^+$
such that all power series~$f_g(q)$ converge whenever $|q|\!<\!1/C$.
The $g\!=\!0$ case of this claim is \cite[Lemma~5.10]{g0ci}.
It implies that the power series 
\BE{graphcnt_e3}f_0(q)\in q\Q[[q]], \qquad 
\big(1\!+\!\ln(1\!-\!f_0(q)\big)^{-1},\big(1\!-\!f_0(q)\big)^{-m}\in\Q[[q]]
~\hbox{with}~m\!\in\!\Z\EE
converge for $|q|\!<\!1/C$  for some $C\!\in\!\R^+$.\\

\noindent
We thus assume that $g\!\in\!\Z^+$.
Suppose $N\!\in\!\Z^{\ge0}$, 
$\Ga$ is a connected trivalent $(N\!+\!1)$-marked genus~$g$ graph 
as in~\e_ref{GaNdfn_e} with
$$S =\flr{N\!+\!1}\!\equiv\!\big\{0,1,\ldots,N\big\},$$ 
and $v\!=\!\eta(0)$ is the vertex to which the  marked point labeled by~0
is attached.
Let $S_v\!\subset\!\flr{N\!+\!1}$ be the subset of marked points attached to~$v$, 
as in Section~\ref{Mainform_subs}.
Breaking~$\Ga$ at~$v$ and replacing the flags at~$v$ with marked points, we~obtain
connected strands similarly to Figure~\ref{strands_fig} on page~\pageref{strands_fig}.
The set of these strands consists~of
\begin{enumerate}[label=$\bullet$,leftmargin=*]

\item a genus~0 two-vertex graph, with the marked point~0 at one of the vertices
and another marked point at the other vertex;

\item $m$ additional strands~$\Ga_i$, each of which is of genus~$g_i$ and carries $N_i\!\in\!\Z^{\ge0}$
of the original marked points of the set~$[N]$ and $(s_i\!+\!1)\!\in\!\Z^+$ additional marked points
so that \hbox{$g_i\!+\!N_i\!+\!s_i\!>\!0$} and $\Ga_i$ is a trivalent graph if 
\hbox{$2g_i\!+\!N_i\!+\!s_i\!\ge\!2$}.

\end{enumerate}
Since $\Ga$ is a  trivalent $N$-marked genus~$g$ graph,
$$2\g(v)\!+\!m\!+\!\sum_{i=1}^{i=m}\!s_i \ge2,  \quad 
\g(v)\!+\!\sum_{i=1}^{i=m}\!g_i\!+\!\sum_{i=1}^{i=m}\!s_i =g, \quad \sum_{i=1}^{i=m}\!N_i=N.$$
Summing over all possibilities for~$\Ga$, we thus obtain
$$a_{g,N}=\sum_{g'=0}^{g'=g}\sum_{m=1}^{\i}
\sum_{\begin{subarray}{c}\bfg,\s\in(\Z^{\ge0})^m\\ 2g'+m+|\s|\ge2\\
g'+|\bfg|+|\s|=g\end{subarray}}
\!\sum_{\begin{subarray}{c}\bfN\in(\Z^{\ge0})^m\\ |\bfN|=N\end{subarray}}\hspace{-.2in}
\frac{(3g'\!+\!m\!+\!|\s|\!-\!2)!}{m!}\binom{N}{\bfN}
\!\!\prod_{i=1}^{i=m}\frac{a_{g_i,N_i+s_i}}{(s_i\!+\!1)!}\,.$$
This is equivalent to
$$f_g(q)=\sum_{g'=0}^{g'=g}\sum_{m=1}^{\i}
\sum_{\begin{subarray}{c}\bfg,\s\in(\Z^{\ge0})^m\\ 2g'+m+|\s|\ge2\\
g'+|\bfg|+|\s|=g\end{subarray}}\hspace{-.25in}
\frac{(3g'\!+\!m\!+\!|\s|\!-\!2)!}{m!}
\prod_{i=1}^{i=m}\!\!\bigg(\frac{1}{(s_i\!+\!1)!}
\frac{\tnd^{s_i}f_{g_i}}{\tnd q^{s_i}}(q)\!\!\bigg).$$
We note that
$$3g'\!+\!\big|\big\{i\!\in\![m]\!:(g_i,s_i)\!\neq\!(0,0)\big\}\big|\!+\!|\s|\ge 2$$
above unless $(g_i,s_i)\!=\!(g,0)$ for some $i\!\in\![m]$ and 
$(g_j,s_j)\!=\!(0,0)$ for all $j\!\neq\!i$.
Thus,
\begin{equation*}\begin{split}
f_g(q)&=f_g(q)\!\!\sum_{m'=1}^{\i}\!\!\frac{f_0(q)^{m'}}{m'}+
\sum_{g'=0}^{g'=g}\sum_{m=0}^{\i}
\sum_{\begin{subarray}{c}\bfg,\s\in(\Z^{\ge0})^m\\  2g'+m+|\s|\ge2\\
g'+|\bfg|+|\s|=g\\ (g_i,s_i)\neq\0\,\forall i\end{subarray}}
\!\!\!\!\Bigg(\!\!
\frac{(3g'\!+\!m\!+\!|\s|\!-\!2)!}{m!}
\prod_{i=1}^{i=m}\!\!\bigg(\frac{1}{(s_i\!+\!1)!}
\frac{\tnd^{s_i}f_{g_i}}{\tnd q^{s_i}}(q)\!\!\bigg)\\
&\hspace{3in}\times\!
\sum_{m'=0}^{\i}\!\binom{3g'\!+\!m\!+\!m'\!+\!|\s|\!-\!2}{m'}f_0(q)^{m'}\!\Bigg)\,.
\end{split}\end{equation*}
Combining the above with the last identity in Lemma~\ref{comb_l0}, we obtain 
\BE{graphcnt_e9}\begin{split}
&\big(1\!+\!\ln(1\!-\!f_0(q)\big)f_g(q)\\
&\hspace{.5in}=\sum_{g'=0}^{g'=g}\sum_{m=0}^{\i}
\sum_{\begin{subarray}{c}\bfg,\s\in(\Z^{\ge0})^m\\  2g'+m+|\s|\ge2\\
g'+|\bfg|+|\s|=g\\ (g_i,s_i)\neq\0\,\forall i\end{subarray}}
\!\!\!
\frac{(3g'\!+\!m\!+\!|\s|\!-\!2)!}{m!(1\!-\!f_0(q))^{3g'+m+|\s|-1}}
\prod_{i=1}^{i=m}\!\!\bigg(\frac{1}{(s_i\!+\!1)!}
\frac{\tnd^{s_i}f_{g_i}}{\tnd q^{s_i}}(q)\!\!\bigg)\,.
\end{split}\EE
The right-hand side above is a finite sum (the terms with $m\!>\!g\!-\!g'$ vanish)
and $g_i\!<\!g$ for all $i\!\in\![m]$ and for all summands in~\e_ref{graphcnt_e9}.
Since the power series~\e_ref{graphcnt_e3} converge for $|q|\!<\!1/C$,
it follows by induction that so does the power series~$f_g(q)$.
This establishes the claim.
\end{proof}

\begin{proof}[{\bf{\emph{Proof of Proposition~\ref{Coeffbnd_prp}}}}]
Let $a!\!\equiv\!0$ if $a\!<\!0$.
We first note~that 
\BE{Coeffbnd_e0a}
\frac{(b''\!+\!\ep\!-\!b)!}{b''!}\le \frac{(b''\!+\!\ep)!}{b''!b!}
\le   \binom{b''\!+\!\ep}{\ep}\frac{\ep!}{b!}
\le 2^{b''+\ep}\frac{\ep!}{b!}  \qquad\forall\,b,b'',\ep\!\in\!\Z^{\ge0}.\EE
Furthermore,
\BE{Coeffbnd_e0b}\begin{split}
\frac{(b_-''\!+\!\ep_-\!-\!b')!}{b_-''!}
\frac{(b_+''\!+\!\ep_+\!+\!1\!+\!b')!}{b_+''!}
&\le \binom{b_-''\!+\!b_+''\!+\!\ep_-\!+\!\ep_+\!+\!1}{b_-'',b_+'',\ep_-,\ep_+,1}\ep_-!\ep_+!\\
&\le 5^{b_-''+b_+''+\ep_-+\ep_++1}\ep_-!\ep_+!
\end{split}\EE
for all $b',b_-'',b_+'',\ep_-,\ep_+\!\in\!\Z^{\ge0}$.\\

\noindent
Let $\Ga\!\in\!\cA_{g,N}$ be as in~\e_ref{GaNdfn_e}. 
Since $\Ga$ is a trivalent graph, 
$$2g_v\!+\!\big|\ov\Fl_v(\Ga)\big|\ge3 ~~\forall\,v\!\in\!\Ver, \quad
\big|\Ver\big|\le 
2\bigg(\sum_{v\in\Ver}\!\!\!\big(g_v\!-\!1)+\big|\Edg\big|\bigg)\!+\!N
=2\big(g\!-\!1\big)\!+\!N\,;$$
the second statement above follows from the first one,
the first statement in~\e_ref{mvsum_e}, and~\e_ref{gGAcond_e}.
Combining~\e_ref{gGAcond_e}, the above bound, and
the first statement in~\e_ref{mvsum_e}, we~obtain
\BE{Coeffbnd_e2a}
\big|\Edg\big|\le g\!-\!1+\big|\Ver\big|\le 3(g\!-\!1)\!+\!N, \quad
\sum_{v\in\Ver}\!\!\big((g_v\!-\!1)+|\ov\Fl_v(\Ga)|\big)\le 4(g\!-\!1)+2N\,.\EE
If in addition $\b'\!\in\!(\Z^{\ge0})^{\Edg}$ and $v\!\in\!\Ver$, let
$$ \De_v(\b')=3(g_v\!-\!1)\!+\!3\big|\ov\Fl_v(\Ga)\big|\!+\!1
-\sum_{f\in \Fl_v^-(\Ga)}\!\!\!\!\!\!b_{e_f}'
+\sum_{f\in \Fl_v^+(\Ga)}\!\!\!\!\!\!b_{e_f}'\,.$$
By the above bounds,
\BE{Coeffbnd_e2}
\sum_{v\in\Ver}\!\!\De_v(\b')\le 14(g\!-\!1)+7N\,.\EE

\vspace{.2in}

\noindent
Let $(\b'',\un\ep,\c,\bfI)\!\in\!\wt\cA^{\star}(\Ga)$ and $\b'\!\in\!(\Z^{\ge0})^{\Edg}$.
By~\e_ref{wtcAstaredfn_e}, 
\BE{Coeffbnd_e3}\begin{split}
\|\c_v\|\!+\!1+\!\!\sum_{s\in S_v}\!\!\big(b_s''\!+\!\ep_s\!-\!b_s\!+\!1\big)
&+\!\!\!\sum_{f\in\Fl_v^-(\Ga)}\!\!\!\!\!\!\big(b_f''\!+\!\ep_f\!-\!b_{e_f}'\!+\!1\big)\\
&+\!\!\!\sum_{f\in\Fl_v^+(\Ga)}\!\!\!\!\!\!\big(b_f''\!+\!\ep_f\!+\!1\!+\!b_{e_f}'\!+\!1\big)
\le\De_v(\b')
\end{split}\EE
for every $v\!\in\!\Ver$.
Along with~\e_ref{Coeffbnd_e0a} and~\e_ref{Coeffbnd_e0b}, this implies~that 
\BE{Coeffbnd_e4}\frac{1}{\un\ep_v!}
\prod_{s\in S_v}\!\frac{(b_s''\!+\!\ep_s\!-\!b_s)!}{b_s''!}
\!\!\!\!\prod_{f\in\Fl_v^-(\Ga)}\!\!\!\!\!\!\!
\frac{(b_f''\!+\!\ep_f\!-\!b_{e_f}')!}{b_f''!}
\!\!\!\!\prod_{f\in\Fl_v^+(\Ga)}\!\!\!\!\!\!\!
\frac{(b_f''\!+\!\ep_f\!+\!1\!+\!b_{e_f}')!}{b_f''!}
\le  5^{\De_v(\b')}\!\!\prod_{s\in S_v}\!\!\frac{1}{b_s!}\,.\EE

\vspace{.2in}

\noindent
By the first statement in \cite[Corollary~5.8]{g0ci},
there exists $C_n\!\in\!\R$ such~that
\BE{Coeffbnd_e5}\bigg|\LRbr{\frac{\Phi_{p;b}(q)}{\Phi_0(q)}}_{q;d} \bigg|
\le  C_n^b \LRbr{\big(1\!-\!C_nq\big)^{-b-1}}_{q;d}b! 
\quad\forall\,b,d\!\in\!\Z^{\ge0},\,p\!\in\!\nset.\EE
By Lemma~\ref{Phibnd_lmm}, \e_ref{Coeffbnd_e5}, \e_ref{Coeffbnd_e3}, and~\e_ref{Coeffbnd_e4},
there exists $C_{n,g}\!\in\!\R$ such~that
the absolute value of each nonzero factor~$\lrbr{\cdot}$ in~\e_ref{ncCdfn2_e} 
is bounded above~by
\begin{gather*}
\frac{(3(g_v\!-\!1)\!+\!|\ov\Fl_v(\Ga)|\!+\!|\c_v|)!}{|\c_v|!}
\binom{|\c_v|}{\c_v}\!\!\prod_{r=1}^{\i}\!\!\bigg(\frac{1}{r\!+\!1}\bigg)^{\!c_{v;r}}
C_{n,g}^{\De_v(\b')}\!\!\LRbr{(1\!-\!C_{n,g}q)^{-\De_v(\b')}}_{q;d_v}
\prod_{s\in S_v}\!\!\frac{1}{b_s!}\,
\end{gather*}
Along with~\e_ref{Coeffbnd_e2},
this implies that 
the absolute value of each summand in~\e_ref{ncCdfn2_e} 
is bounded above~by
\begin{equation*}\begin{split}
&\frac{C_{n,g}^{14(g-1)+7N}\LRbr{(1\!-\!C_{n,g}q)^{-(14(g-1)+7N)}}_{q;d}}{\b!}\\
&\hspace{1in}\times\!\!\!
\prod_{v\in\Ver}\!\!\Bigg(\!\!
\frac{(3(g_v\!-\!1)\!+\!|\ov\Fl_v(\Ga)|\!+\!|\c_v|)!}{|\c_v|!}
\binom{|\c_v|}{\c_v}\!\!\prod_{r=1}^{\i}\!\!\bigg(\frac{1}{r\!+\!1}\bigg)^{\!c_{v;r}}\!\Bigg).
\end{split}\end{equation*}

\vspace{.2in}

\noindent
By~\e_ref{wtcAstaredfn_e}, the number of possibly nonzero factors in~\e_ref{ncCdfn2_e} 
with $\c_v$ fixed (and $\b'',\un\ep$, and $I$ varying)  is bounded above~by
$$\binom{3(g_v\!-\!1)\!+\!|\ov\Fl_v(\Ga)|\!-\!\|\c_v\|
+g_v\!+\!2|\ov\Fl_v(\Ga)|\!-\!1}{g_v\!+\!2|\ov\Fl_v(\Ga)|\!-\!1} 
\le 2^{4(g_v-1)+3|\ov\Fl_v(\Ga)|-\|\c_v\|}\,.$$
Along with the conclusion of the previous paragraph and~\e_ref{Coeffbnd_e2a}, this implies that
the absolute value of the sum in~\e_ref{ncCdfn2_e} with $\Ga$, $(\bfd,\p',\b')$,
and $\c\!\equiv\!(\c_v)_{v\in\Ver}$ fixed is bounded above~by
\begin{equation*}\begin{split}
&\frac{C_{n,g}^{30(g-1)+15N}\LRbr{(1\!-\!C_{n,g}q)^{-(14(g-1)+7N)}}_{q;d}}{\b!}
\prod_{v\in\Ver}\!\!\!\!\big(3(g_v\!-\!1)\!+\!|\ov\Fl_v(\Ga)|\big)!\\
&\hspace{1in}\times2^{-\|\c\|}\!\!\!\prod_{v\in\Ver}\!\!\Bigg(\!
\frac{(3(g_v\!-\!1)\!+\!|\ov\Fl_v(\Ga)|\!+\!|\c_v|)!}{(3(g_v\!-\!1)\!+\!|\ov\Fl_v(\Ga)|)!|\c_v|!}
\binom{|\c_v|}{\c_v}\!\!\prod_{r=1}^{\i}\!\!\bigg(\frac{1}{r\!+\!1}\bigg)^{\!c_{v;r}}\!\Bigg).
\end{split}\end{equation*}
By Lemma~\ref{grcomb_lmm}, the sum of the terms on the second line above
over all possibilities for~$\c$ is bounded by~$2^{3g+N}$.\\

\noindent
By~\e_ref{wtcAstaredfn_e} and the second equality in~\e_ref{mvsum_e},
$$|\b'|\le |\b''|\!+\!|\un\ep|\le 3(g\!-\!1)\!+\!N-\big|\Edg\big|$$
for every nonzero summand~$\lrbr{\cdot}$ in~\e_ref{ncCdfn2_e}. 
Thus, the number of tuples~$\b'$ in~\e_ref{ncCdfn2_e} which contribute 
to~\e_ref{ncCdfn2_e} with $\Ga$ and $(\bfd,\p')$ fixed is bounded above~by
$$\binom{3(g\!-\!1)\!+\!N\!-\!|\Edg|+|\Edg|\!-\!1}{|\Edg|\!-\!1}
\le 2^{3(g-1)+N}\,.$$
Along with~\e_ref{sumcond_e2} and the first statement in~\e_ref{Coeffbnd_e2a}, 
this implies that  the number of elements of $\cS_{\Ga}(d,\p,\b)$ 
with nonzero summands in~\e_ref{ncCdfn2_e} is bounded above~by
$$2^{3(g-1)+N}\cdot n^{|\Edg|}\le (2n)^{3(g-1)+N}\,.$$

\vspace{.2in}

\noindent
By the conclusions of the last two paragraphs,
there exists $C_{n,g}\!\in\!\R$ such~that
the absolute value of the contribution of each $\Ga\!\in\!\cA_{g,n}$ in~\e_ref{ncCdfn2_e} 
times~$|\Aut(\Ga)|$ is bounded above~by
\begin{equation*}\begin{split}
&\frac{C_{n,g}^N}{\b!}\bigg|\binom{-(14(g\!-\!1)\!+\!7N)}{d}\bigg|C_{n,g}^d\cdot
\prod_{v\in\Ver}\!\!\!\!\big(3(g_v\!-\!1)\!+\!|\ov\Fl_v(\Ga)|\big)!\\
&\hspace{1in}
=\frac{C_{n,g}^N}{\b!}\binom{14(g\!-\!1)\!+\!7N+d\!-\!1}{d}C_{n,g}^d\cdot
\prod_{v\in\Ver}\!\!\!\!\big(3(g_v\!-\!1)\!+\!|\ov\Fl_v(\Ga)|\big)!\\
&\hspace{1in}
\le \frac{C_{n,g}^N}{\b!}2^{14(g-1)+7N+d}C_{n,g}^d\cdot
\prod_{v\in\Ver}\!\!\!\!\big(3(g_v\!-\!1)\!+\!|\ov\Fl_v(\Ga)|\big)!\,;
\end{split}\end{equation*}
on the first line above
$$\binom{a}{b}\equiv \frac{a(a\!-\!1)\ldots(a\!-\!d\!+\!1)}{d!}$$
is as in the Binomial Theorem.
The  claimed bound for~$\nc_{g;\p,\b}^{(d)}$ now follows from  Lemma~\ref{graphcnt_lmm}.
\end{proof}

\section{Torus equivariant setting}
\label{equiv_sec}

\noindent
In Section~\ref{equivGW_subs}, we first review the relevant aspects of equivariant cohomology;
a more detailed discussion can be found in \cite[Section~1.1]{bcov1}.
We then state an equivariant version of Theorem~\ref{main_thm}; 
see Theorem~\ref{equiv_thm}.
Theorem~\ref{main_thm} is obtained from Theorem~\ref{equiv_thm} 
by setting $\al\!=\!0$ and using the $l\!=\!0$ case of the second statement 
of \cite[Theorem~5]{PoZ}.
In Section~\ref{outline_subs}, we apply the Virtual Equivariant Localization Theorem
of~\cite{GP} to reduce the generating series~\e_ref{cZdfn_e} for equivariant GW-invariant
to a sum over the fixed loci
of the  actions of the $n$-torus $\T$ on the moduli spaces $\ov\M_{g,N}(\P^{n-1},d)$.
The two proofs of Theorem~\ref{equiv_thm} carried out in Sections~\ref{ClFormComp_subs} 
and~\ref{RecFormComp_subs} are outlined in Section~\ref{pfoutline_subs};
they involve breaking the fixed loci into  pieces of finitely many types
for each fixed pair~$(g,N)$.
The technical observations and background data needed for these proofs are
gathered in Section~\ref{prelim_subs}.

\subsection{Equivariant GW-invariants}
\label{equivGW_subs}

\noindent
Denote by $\ga\!\lra\!\P^{\i}$ the tautological line bundle and 
by $\T$ the complex $n$-torus $(\C^*)^n$.
Its group cohomology is the polynomial algebra on $n$~generators:
\BE{HTdfn_e}H_{\T}^*\equiv H^*(B\T;\Q)=\Q[\un\al]\equiv\Q[\al_1,\ldots,\al_n],\EE
where $\al_i\!=\!\pi_i^*c_1(\ga^*)$  and
$$\pi_i\!: B\T\!\equiv\!\big(\P^{\i}\big)^{\!N}  \lra B\C^*\!\equiv\!\P^{\i}$$
is the projection onto the $i$-th component.
For $r\!\in\!\Z^{\ge0}$, let $\si_r\!\in\!H_{\T}^*$ be 
the $r$-th elementary symmetric polynomial in $\al_1,\al_2,\ldots,\al_n$ and
$$\wh\si_r=(-1)^{r-1}\si_r\in H_{\T}^*\,.$$
Denote by 
$$\cI\subset\Q[\al_1,\ldots,\al_n]^{S_n} \subset H_{\T}^*$$
the ideal generated by $\si_1,\si_2,\ldots,\si_{n-1}$ inside of 
the ring of symmetric polynomials.\\

\noindent
If $\T$ is acting on a topological space $M$, let
$$H_{\T}^*(M)\equiv H^*(BM;\Q), \qquad\hbox{where}\qquad BM=E\T\!\times_{\T}\!M,$$
be the \sf{equivariant cohomology} of $M$.
If the $\T$-action on $M$ lifts to an action on a (complex) 
vector bundle $V\!\lra\!M$, let
$$\E(V)\equiv e(BV),{\bf c}(V)\equiv c(BV)\in H_{\T}^*(M)$$
denote the \sf{equivariant Euler} and \sf{Chern classes~of} $V$.\\

\noindent
The projection map $BM\!\lra\!B\T$ induces an action of $H_{\T}^*$ on $H_{\T}^*(M)$.
If in addition $M$ is a compact oriented manifold, 
this projection induces a well-defined integration-along-the-fiber homomorphism
\BE{intMTdfn_e}\int_M\!: H_{\T}^*(M)\lra H_{\T}^*\EE
for the fiber bundle $BM\!\lra\!B\T$.
It commutes with the actions of~$H_{\T}^*$.
If $M'$ is another compact oriented manifold with a $\T$-action, 
a $\T$-equivariant continuous map $f\!:M\!\lra\!M'$ determines 
an \sf{equivariant cohomology push-forward homomorphism}
\BE{HTpushdfn_e}f_*\!: H_{\T}^*(M)\lra H_{\T}^*(M')\,.\EE
It is characterized by the property that 
\BE{HTpushdfn_e2}\int_M\!\psi\big(f^*\psi'\big)=\int_{M'}\!\big(f_*\psi\big)\psi'\in H_{\T}^*
\qquad\forall~\psi\!\in\!H_{\T}^*(M),~\psi'\!\in\!H_{\T}^*(M').\EE
The homomorphism~\e_ref{HTpushdfn_e} commutes with the actions of~$H_{\T}^*$.\\

\noindent
Throughout the paper we work with the standard action of $\T$ on $\P^{n-1}$:
$$\big(e^{\I\th_1},\ldots,e^{\I\th_n}\big)\cdot [z_1,\ldots,z_n] 
=\big[e^{\I\th_1}z_1,\ldots,e^{\I\th_n}z_n\big].$$
It naturally lifts to the tautological line bundle $\ga$
and to the tangent bundle~$T\P^{n-1}$.
Let 
$$\x\equiv \E(\ga^*) \in H_{\T}^2(\P^{n-1})$$
be the \sf{equivariant hyperplane class}.
For $N\!\in\!\Z^{\ge0}$,
the $\T$-equivariant cohomology of $\P^{n-1}_N$ with respect 
to the induced diagonal $\T$-action on~$\P^{n-1}_N$ is given~by 
\BE{pncoh_e}
H_{\T}^*(\P^{n-1}_N) = \Q\big[\un\al,\x_1,\ldots,\x_n\big]\Big/
\big\{(\x_s\!-\!\al_1)\ldots(\x_s\!-\!\al_n)\!:s\!=\!1,\ldots,N\big\},\EE
where $\x_s\!=\!\pi_s^*\x$ and $\pi_s\!:\P^{n-1}_N\!\lra\!\P^{n-1}$ is
the projection onto the $s$-th component.
For each $\p\!\in\!\nset^N$, let
$$\un\x^{\p}=\prod_{s=1}^{s=N}\!\x_s^{p_s}\in H_{\T}^*(\P^{n-1}_N)\,;$$
these elements form a basis for $H_{\T}^*(\P^{n-1}_N)$ as a module 
over $H_{\T}^*\!=\!\Q[\un\al]$.\\

\noindent
For $g,N\!\in\!\Z^{\ge0}$, the action of $\T$ on $\P^{n-1}$
induces an action on $\ov\M_{g,N}(\P^{n-1},d)$ so that the evaluation~map
$$\ev^d\!\equiv\!\ev_1\!\times\!\ldots\!\times\!\ev_N\!: 
\ov\M_{g,N}(\P^{n-1},d)\lra \P^{n-1}_N$$
is $\T$-equivariant.
By~\cite{BF,LiT0}, the moduli space $\ov\M_{g,N}(\P^{n-1},d)$ carries 
an equivariant \sf{virtual fundamental class}.
It defines a homomorphism~\e_ref{HTpushdfn_e} 
with $M\!=\!\ov\M_{g,N}(\P^{n-1},d)$ which satisfies~\e_ref{HTpushdfn_e2}
with $\int_M$ replaced by the integration against this class.
In particular, there is a well-defined equivariant cohomology push-forward homomorphism
\BE{evPNdfn_e}\ev_*^d\equiv \big\{\ev_1\!\times\!\ldots\!\times\!\ev_N\big\}_*\!:
H_{\T}^*\big(\ov\M_{g,N}(\P^{n-1},d)\big)\lra H_{\T}^*\big(\P^{n-1}_N\big).\EE
It is characterized by the property that 
\BE{evPNdfn_e2}
\int_{[\ov\M_{g,N}(\P^{n-1},d)]^{\vir}}\!\!\!\!\psi\big(\ev^{d*}\psi'\big)
=\int_{\!M'}\!\!\big(\ev_{d*}\psi\!\big)\psi'\in H_{\T}^*\EE
for all $\T$-equivariant cohomology classes $\psi$ on $\ov\M_{g,N}(\P^{n-1},d)$
and $\psi'$ on~$\P^{n-1}_N$.
The homomorphism~\e_ref{evPNdfn_e} commutes with the actions of~$H_{\T}^*$.\\

\noindent
With $\ev_*^d$ as in~\e_ref{evPNdfn_e}, 
$\un\hb$ and $\un\hb^{-1}$ as in~\e_ref{Zdfn_e}, and $\un\x\!=\!(\x_1,\ldots,\x_n)$, 
let
\BE{cZdfn_e} \cZ^{(g)}\big(\un\hb,\un\x,q\big)= \sum_{d=0}^{\i}q^d
\ev_*^d\bigg\{\prod_{s=1}^{s=N}\!\!\frac{1}{\hb_s\!-\!\psi_s}\bigg\}
\in H_{\T}^*(\P^{n-1}_N)\big[\big[\un\hb^{-1},q\big]\big]\,.\EE
For $g\!=\!0$ and $N\!=\!1,2$, we define the coefficient of $q^0$ in~\e_ref{cZdfn_e} to~be
$$ 1 \qquad\hbox{and}\qquad
-\frac{1}{\hb_1\!+\!\hb_2}
\sum_{\begin{subarray}{c}p_1,p_2,r\in\Z^{\ge0}\\ p_1+p_2+r=n-1 \end{subarray}}
\!\!\!\!\!\!\!\!\!\!\!\wh\si_r\x_1^{p_1}\x_2^{p_2}\,,$$
respectively.
For each $p\!\in\!\nset$, let
\BE{cZpdfn_e}\cZ_p(\hb,\x,q)=\x^p
+\sum_{d=1}^{\i}q^d\ev_{1*}^d\bigg\{\!\frac{\ev_2^{d*}\x^p}{\hb\!-\!\psi_1}\!\bigg\}
\in H_{\T}^*(\P^{n-1})\big[\big[\hb^{-1},q\big]\big],\EE
where $\ev_1^d,\ev_2^d\!:\ov\M_{0,2}(\P^{n-1},d)\!\lra\!\P^{n-1}$.
For  $\p\!=\!(p_1,p_2,\ldots,p_N)\!\in\!\nset^N$, define
\BE{cZbpdfn_e}
\cZ_{\p}(\un\hb,\un\x,q)
=\prod_{s=1}^{s=N}\hb_s^{-1}\cZ_{p_s}(\hb,\x_s,q)\,.\EE

\begin{thmlet}\label{equiv_thm}
Suppose $n,N\!\in\!\Z^+$ and $g\!\in\!\Z^{\ge0}$ with $n\!\ge\!2$ and $2g\!+\!N\!\ge\!3$.
The generating function~\e_ref{cZdfn_e} for the equivariant $N$-pointed genus~$g$ GW-invariants 
of~$\P^{n-1}$ 
is given~by 
\BE{equivthm_e}\cZ^{(g)}\big(\un\hb,\un\x,q\big)= 
\sum_{\p\in\nset^N}\sum_{\b\in(\Z^{\ge0})^N}
\sum_{d=0}^{\i}\cC_{g;\p,\b}^{(d)}q^d\un\hb^{-\b}\cZ_{\p}(\un\hb,\un\x,q)\EE
for some $\cC_{g;\p,\b}^{(d)}\!\in\!\Q[\al]$ such that 
\BE{equivthm_e2}\cC_{g;\p,\b}^{(d)}-\sum_{t=0}^{\i} \nc_{g;\p,\b}^{(d,t)}
\wh\si_n^t\in\cI,\EE
where $\nc_{g;\p,\b}^{(d,t)}\!\in\!\Q$ are the numbers defined in Section~\ref{Mainform_subs}.
\end{thmlet}

\noindent
The closed formula~\e_ref{cCformula_e} and separately the recursion~\e_ref{cCformula_e4}
compute the coefficients~$\cC_{g;\p,\b}^{(d)}$ in~\e_ref{equivthm_e} 
and thus provide a straightforward (though laborious) 
algorithm for computing the  generating function~\e_ref{cZdfn_e}
for the equivariant $N$-pointed genus~$g$ GW-invariants of~$\P^{n-1}$.
Let
$$\cY(\hb,\x,q)=
\sum_{d=0}^{\i}\frac{q^d}
{\prod\limits_{r=1}^{r=d}\left(\prod\limits_{k=1}^{k=n}\!\!(\x\!-\!\al_k\!+\!r\hb)
-\prod\limits_{k=1}^{k=n}\!\!(\x\!-\!\al_k)\!\right)}
\in \big(\Q_{\al;\hb,\x}'\!\cap\!\Q_{\al}[\x][[\hb^{-1}]]\big)\big[\big[Q\big]\big].$$
By \cite[Section~29.1]{MirSym} and \cite[Lemma~A.1]{g0ci}, 
the power series~$\ze(\x,q)$ and $\Psi_b(\x,q)$ in~\e_ref{cZexp_e} are described~by
$$\cY(\hb,\x,q)=e^{\ze(\x,q)/\hb}\sum_{b=0}^{\i}\Psi_b(\x,q)\hb^b\,.$$
By the proofs of Lemma~A.1 and Proposition~2.1 in~\cite{g0ci}, 
this relation determines~$\ze(\x,q)$ and $\Psi_b(\x,q)$ through an explicit recursion
involving differential operators.
By~\e_ref{Psimcdfn_e} and the proof of \cite[Proposition~4.2]{g0ci},
the power series~$\Psi_b(\x,q)$ determine the power series
$\Psi_{I;\c}^{(g,\un\ep)}$ in~\e_ref{cZmB_e} and 
$\Psi_{p;b}(\x,q)$  in~\e_ref{cZpexp_e}.\\

\noindent
For concreteness, we now describe the power series~$\ze(\x,q)$ and $\Psi_0(\x,q)$ explicitly.
For any $r\!\in\!\Z^{\ge0}$ and any power series~$f$, 
denote by~$\si_r(f)$ the power series
obtained from~$f$ by taking the $r$-th elementary symmetric polynomial in
$\{f\!-\!\al_i\}_{i\in[n]}$. 
Define
$$ L(\x,q)\in \x\!+\!q\cdot\Q[\al,\x,\si_{n-1}(\x)^{-1}]\big[\big[\x^{-1}q\big]\big] 
\quad\hbox{by}\quad  \si_n\big(L(\x,q)\big)-q=\si_n(\x).$$
The power series $\ze(\x,q)$ and $\Psi_0(\x,q)$ are described~by
\begin{gather*}
\ze\in \x q\cdot\Q[\al,\x,\si_{n-1}(\x)^{-1}]\big[\big[q\big]\big], \qquad
\x+D\ze(\x,q)=L(\x,q)\,,\\
\Psi_0(\x,q)=\Bigg(\frac{\x\si_{n-1}(\x)}{L(\x,q)\,\si_{n-1}(L(\x,q))}\Bigg)^{1/2}
\bigg(\frac{L(\x,q)}{\x}\bigg)^{1/2}\,;
\end{gather*}
setting $\al\!=\!0$ and $\x\!=\!1$ above gives the first formulas in~\e_ref{PhiODE_e}
and in~\e_ref{F0exp_e2}.\\

\noindent
As demonstrated in \cite{GWvsSQ,PoZ,bcov1}, equivariant localization computations
in GW-theory can sometimes be carried out by working with the residues of 
the equivariant mirror B-side functions and by extracting the non-equivariant terms at the~end.
In such situations, precise knowledge of the equivariant coefficients 
$\cC_{g;\p,\b}^{(d)}$ in~\e_ref{equivthm_e} is not avoidable.

\subsection{Equivariant localization setup}
\label{outline_subs}

\noindent
Denote by
$$\H_{\T}^*=\Q_{\al}\equiv \Q(\al_1,\ldots,\al_n)$$
the field of fractions of $H_{\T}^*$.
If $M$ is a topological space with a $\T$-action, let 
$$\H_{\T}^*(M)=H_{\T}^*(M)\otimes_{H_{\T}^*}\H_{\T}^*.$$
In the case $M$ is a compact oriented manifold,
the classical equivariant localization theorem of~\cite{AB}
relates the homomorphism~\e_ref{intMTdfn_e} to the fixed locus of the $\T$-action.
The latter is a union of compact orientable manifolds~$F$
and $\T$ acts on the normal bundle $\N F$ of each~$F$.
Once an orientation of $F$ is chosen, there is a well-defined 
integration-along-the-fiber homomorphism
$$\int_F\!: H_{\T}^*(F)\lra H_{\T}^*.$$
The localization theorem of~\cite{AB} states that 
\BE{ABothm_e}
\int_M\psi = \sum_F\int_F\frac{\psi|_F}{\E(\N F)} \in H_{\T}^* \subset\H^*_{\T}
\qquad\forall~\psi\in  H_{\T}^*(M),\EE
where the sum is taken over all components $F$ of the fixed locus of $\T$. 
Part of the statement of~\e_ref{ABothm_e} is that $\E(\N F)$
is invertible in~$\H_{\T}^*(F)$.\\

\noindent
The standard action of~$\T$ on~$\P^{n-1}$ has $n$~fixed points:
$$P_1\equiv[1,0,\ldots,0], \qquad P_2\equiv[0,1,0,\ldots,0], 
\quad\ldots\quad P_n\equiv [0,\ldots,0,1].$$
By the choice of the lift of $\T$-action to the tautological line bundle~$\ga$
over~$\P^{n-1}$,
\BE{xrestr_e}\x|_{P_i}=\al_i\in H_{\T}^*\!=\!H_{\T}^*(P_i)
 \qquad\forall\,i\!\in\![n].\EE
Along with the $\T$-equivariance of Euler's sequence for $\P^{n-1}$,
this implies~that 
\BE{ETPn_e}
\E\big(T\P^{n-1}\big)\big|_{P_i}\equiv
\E\big(T_{P_i}\P^{n-1}\big)
=\prod_{k\in[n]-i}\!\!\!\!(\al_i\!-\!\al_k)\in 
H_{\T}^*(P_i)=H_{\T}^*=\Q[\un\al] \qquad\forall\,i\!\in\![n].\EE
For each $i\!\in\![n]$, define
\BE{phidfn_e}
\phi_i= \prod_{k\neq i}(\x\!-\!\al_k) \in H_{\T}^*(\P^{n-1}).\EE
This is the equivariant Poincare dual of $P_i$ in $\P^{n-1}$ 
in the sense of the $N\!=\!1$ case of~\e_ref{phidfn_e2} below.\\

\noindent
The standard diagonal $\T$-action on $\P^{n-1}_N$ has $n^N$~fixed points:
$$P_{i_1\ldots i_N}\equiv 
P_{i_1}\!\times\!\ldots\!\times\!P_{i_N}, \qquad
i_1,\ldots,i_N\!\in\![n].$$
By~\e_ref{ABothm_e} and~\e_ref{ETPn_e}, this implies that  
\BE{phidfn_e2}\eta|_{P_{i_1\ldots i_N}}\equiv
\int_{P_{i_1\ldots i_N}}\!\!\!\!\eta=
\int_{\P_N^{n-1}}\!
\eta\!\prod_{s=1}^N\!\pi_{i_s}^*\phi_{i_s}\in H_{\T}^*
\qquad\forall~\eta\!\in\!H_{\T}^*\big(\P^{n-1}_N\big),~
i_1,\ldots,i_N\!\in\![n].\EE
Under the identifications~\e_ref{HTdfn_e} and~\e_ref{pncoh_e},
the restriction maps on the equivariant cohomology induced by the inclusions
of $P_{i_1\ldots i_N}$ into~$\P^{n-1}_N$ are the homomorphisms
\BE{restrmap_e}
H_{\T}^*(\P_N^{n-1})\lra H_{\T}^*, \qquad 
\x_s\lra\al_{i_s},~s\!\in\![N].\EE
By~\e_ref{pncoh_e} and \e_ref{restrmap_e},
an element of $H_{\T}^*(\P_N^{n-1})$ is determined by its restrictions
$$\eta|_{P_{i_1\ldots i_N}}\equiv 
\int_{P_{i_1}\times\ldots\times P_{i_N}}\!\!\!\eta \in H_{\T}^*$$
to the $n^N$ fixed points of the $\T$-action on~$\P_N^{n-1}$.
Along with~\e_ref{phidfn_e2} and~\e_ref{evPNdfn_e2},  
this implies that 
the power series $\cZ^{(g)}(\un\hb,\un\x,q)$ in~\e_ref{cZdfn_e} is  
determined by the $n^N$ power series
\BE{cZval_e}\cZ^{(g)}\big(\un\hb,\al_{i_1,\ldots,i_N},q\big)
=\sum_{d=0}^{\i}q^d\!\!
\int_{\ov\M_{g,N}(\P^{n-1},d)}\prod_{s=1}^{s=N}\!\!
\frac{\ev_s^*\phi_{i_s}}{\hb_s\!-\!\psi_s},\EE
where $\al_{i_1\ldots i_N}\!\equiv\!(\al_{i_1},\ldots,\al_{i_N})$.\\

\noindent
The virtual localization theorem of~\cite{GP} extends~\e_ref{ABothm_e}
to the integration against equivariant virtual fundamental classes.
It in particular determines an induced $\T$-action 
on the virtual normal bundle $\N\cZ_{\Ga}^{\vir}$ of each
topological components~$\cZ_{\Ga}$ of the fixed locus of the $\T$-action
and reduces~\e_ref{cZval_e} to integrals over~$\cZ_{\Ga}$.
By \cite[Section~4]{GP},
\BE{GPthm_e}
\int_{\ov\M_{g,N}(\P^{n-1},d)}\!\!\prod_{s=1}^{s=N}\!\!
\frac{\ev_s^*\phi_{i_s}}{\hb_s\!-\!\psi_s}
= \sum_{\Ga\in\cA_{g,N}(n,d)}\int_{\cZ_{\Ga}}\frac{1}{\E(\N\cZ_{\Ga}^{\vir})}
\prod_{s=1}^{s=N}\!\!
\frac{\ev_s^*\phi_{i_s}}{\hb_s\!-\!\psi_s},\EE
where $\cA_{g,N}(n,d)$ is the set of equivalence classes of
connected $[n]$-valued $N$-marked genus~$g$ degree~$d$ graphs;
these are defined below.
Part of the statement of~\e_ref{GPthm_e} is that $\E(\N\cZ_{\Ga}^{\vir})$
is invertible in~$\H_{\T}^*(\cZ_{\Ga})$.
In Section~\ref{ClFormComp_subs}, we use~\e_ref{GPthm_e} to first reduce~\e_ref{cZval_e}
to a sum over the set $\cA_{g,N}(n)$ of equivalence classes of 
connected $[n]$-valued $N$-marked genus~$g$ graphs defined below.
We then sum up over all possibilities for the $[n]$-values of the vertices
to reduce the resulting sum to a sum over the collection~$\cA_{g,N}$ 
of connected trivalent $N$-marked genus~$g$  graphs 
to obtain Theorem~\ref{equiv_thm} with the first definition of 
the structure coefficients~$\nc_{g;\p,\b}^{(d)}$ in Section~\ref{Mainform_subs}.\\

\noindent
An \sf{$[n]$-valued $S$-marked weighted graph} is a tuple
\BE{decortgraphdfn_e}
\Ga\equiv 
\big((\g,\mu)\!:\!\Ver\!\lra\!\Z^{\ge0}\!\times\![n],
\eta\!:S\!\sqcup\!\Fl\!\lra\!\Ver,\d\!:\Edg\!\lra\!\Z^+\big)\EE
such that the tuple
\BE{decortgraphdfn_e0}\Ga_0\equiv \big(\g\!:\!\Ver\!\lra\!\Z^{\ge0},
\eta\!:S\!\sqcup\!\Fl\!\lra\!\Ver,\Edg\big)\EE
is an $S$-marked graph and
\BE{decorgraphcond_e}
\mu(f_e^-)\neq\mu(f_e^+\big)  \qquad\forall~e\!\equiv\!\big\{f_e^-,f_e^+\big\}\in\Edg.\EE
The first diagram in Figure~\ref{graphcore_fig} represents 
an $[n]$-valued 2-marked weighted graph~$\Ga$
with $\g(v)\!=\!0$ for all $v\!\in\!\Ver$.
The values of~$\mu$ on the vertices and of~$\d$ on the edges are indicated by 
the numbers next to the vertices and the edges.
By~\e_ref{decorgraphcond_e}, no two consecutive vertex labels are the same.\\

\noindent
An \sf{equivalence} between an $S$-marked weighted graph as in~\e_ref{decortgraphdfn_e}
and another $S$-marked weighted graph
$$ \Ga'\equiv 
\big((\g',\mu')\!:\!\Ver'\!\lra\!\Z^{\ge0}\!\times\![n],
\eta'\!:S\!\sqcup\!\Fl'\!\lra\!\Ver',\d'\!:\Edg'\!\lra\!\Z^+\big)$$
is an equivalence $(h_{\Ver},h_{\Fl})$ between the associated $S$-marked graphs~$\Ga_0$ 
and~$\Ga_0'$ such~that  
$$\mu=\mu'\!\circ\!h_{\Ver}, \qquad 
\d(e)=\d'\big(h_{\Fl}(e)\big)~~\forall\,e\!\in\!\Edg.$$
We denote by $\Aut(\Ga)$ the group of \sf{automorphisms} of~$\Ga$.\\

\noindent
For $\Ga$ as in~\e_ref{decortgraphdfn_e}, we denote by
$$|\Ga|\equiv\sum_{e\in\Edg}\!\!\!\d(e)$$
its \sf{degree}.
We call a vertex~$v$ of~$\Ga$ \sf{trivalent} if $v$ is a trivalent 
vertex of the associated $N$-marked graph~$\Ga_0$.
We call $\Ga$ \sf{connected} if $\Ga_0$ is connected.
If $\Ga$ is connected, we define its arithmetic genus~$\fa(\Ga)$ 
to be~$\fa(\Ga_0)$.
Let $\cA_{g,N}(n,d)$ be
the set of (equivalence classes~of) connected 
$[n]$-valued $N$-marked  genus~$g$ degree~$d$ graphs
and $\cA_{g,N}(n,*)$ be the union of the sets $\cA_{g,N}(n,d)$ over $d\!\in\!\Z$.\\

\begin{figure}
\begin{pspicture}(6,-1.1)(10,2)
\psset{unit=.4cm}
\pscircle*(30,0){.3}\rput(30,-.7){\smsize{$2$}}
\psline[linewidth=.04](30,0)(26,0)\pscircle*(26,0){.2}
\rput(28,.5){\smsize{$3$}}\rput(26,.6){\smsize{$4$}}
\psline[linewidth=.04](26,0)(23.5,-2.5)\pscircle*(23.5,-2.5){.2}
\rput(24.1,-1.3){\smsize{$3$}}\rput(23,-2.5){\smsize{$2$}}
\psline[linewidth=.04](26,0)(28.5,-2.5)\pscircle*(28.5,-2.5){.2}
\rput(28.1,-1.5){\smsize{$2$}}\rput(29.1,-2.5){\smsize{$3$}}
\psline[linewidth=.04](26,0)(22,0)\pscircle*(22,0){.3}
\rput(24,.5){\smsize{$2$}}\rput(22.5,-.6){\smsize{$1$}}
\psline[linewidth=.04](22,0)(19.5,2.5)\pscircle*(19.5,2.5){.2}
\rput(21,1.7){\smsize{$2$}}\rput(18.9,2.5){\smsize{$3$}}
\psline[linewidth=.04](22,0)(19.5,-2.5)\pscircle*(19.5,-2.5){.2}
\rput(20.2,-1){\smsize{$2$}}\rput(19,-2.6){\smsize{$3$}}
\psline[linewidth=.04](19.5,-2.5)(17.5,-.5)\rput(17.1,-.4){\smsize{$\bf 2$}}
\rput(26,2.6){\smsize{$2$}}
\pnode(30,0){A}\pnode(22,0){B}
\nccurve[ncurv=1,nodesep=.1,angleA=135,angleB=45]{-}{A}{B}
\rput(23,3){$\Ga$}
\psline[linewidth=.04](30,0)(33,2)\pscircle*(33,2){.2}
\rput(31.4,1.5){\smsize{$2$}}\rput(32.7,2.5){\smsize{$1$}}
\psline[linewidth=.04](30,0)(33,-2)\pscircle*(33,-2){.2}
\rput(31.6,-.5){\smsize{$1$}}\rput(32.7,-2.5){\smsize{$4$}}
\psline[linewidth=.04](35.5,3.5)(33,2)\rput(35.8,3.8){\smsize{$\bf 1$}}
\psline[linewidth=.04](36,0)(33,2)\pscircle*(36,0){.2}
\rput(34.6,1.5){\smsize{$1$}}\rput(36.3,.6){\smsize{$3$}}
\psline[linewidth=.04](36.5,-2)(33,-2)\pscircle*(36.5,-2){.2}
\rput(34.7,-1.5){\smsize{$1$}}\rput(36.8,-2.5){\smsize{$3$}}
\pscircle*(50,0){.3}\pscircle*(46,0){.3}
\psline[linewidth=.04](46,0)(44,-2)\rput(43.5,-2){\smsize{$\bf 2$}}
\psline[linewidth=.04](50,0)(52,2)\rput(52.5,2){\smsize{$\bf 1$}}
\pnode(50,0){A}\pnode(46,0){B}
\nccurve[ncurv=.7,nodesep=.1,angleA=210,angleB=-30]{-}{A}{B}
\nccurve[ncurv=.7,nodesep=.1,angleA=150,angleB=30]{-}{A}{B}
\rput(46,.8){\smsize{$1$}}\rput(50,.8){\smsize{$2$}}
\rput(45,3.1){$\ov\Ga$}
\end{pspicture}
\caption{An $[n]$-valued 2-marked genus~1 degree~21 graph~$\Ga$, 
with special vertices indicated by larger dots, 
and its core~$\ov\Ga$.}
\label{graphcore_fig}
\end{figure}
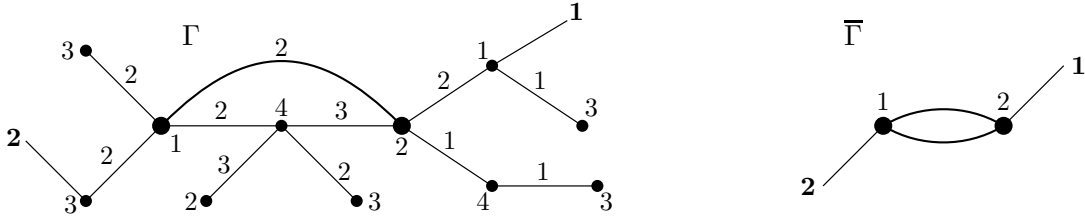

\noindent
An \sf{$[n]$-valued $N$-marked graph} is a tuple
$$\Ga\equiv 
\big((\g,\mu)\!:\!\Ver\!\lra\!\Z^{\ge0}\!\times\![n],
\eta\!:[N]\!\sqcup\!\Fl\!\lra\!\Ver,\Edg\big)$$
such that the tuple~\e_ref{decortgraphdfn_e0} is an $N$-marked graph and
$\mu$ satisfies~\e_ref{decorgraphcond_e}.
We define the notions of equivalence, trivalence, connectedness, and genus for such graphs
via the associated graph~\e_ref{decortgraphdfn_e0} as above.
We denote by $\Aut(\Ga)$ the group of automorphisms of a graph~$\Ga$ as above 
and by $\cA_{g,N}(n)$ the set of equivalence classes of 
connected trivalent $[n]$-valued $N$-marked genus~$g$ graphs.

\subsection{Outline of proofs of Theorem~\ref{equiv_thm}}
\label{pfoutline_subs}

\noindent
Let $\Ga$ be a connected $N$-marked genus~$g$ graph as in~\e_ref{GaNdfn_e}
with $2\fa(\Ga)\!+\!N\!\ge\!3$.
A vertex $v\!\in\!\Ver$ such that $\val_v(\Ga)\!\le\!0$
can then be \sf{contracted} to obtain another connected $N$-marked genus~$g$ graph
\begin{gather*}
\Ga'\equiv 
\big(\g'\!:\!\Ver'\!\lra\!\Z^{\ge0},
\eta'\!:[N]\!\sqcup\!\Fl'\!\lra\!\Ver',\Edg'\big) \qquad\hbox{s.t.}\\
\Ver'=\Ver\!-\!\{v\}, \quad
\g'=\g|_{\Ver'}, \quad 
\Fl'\subset\Fl\!\cap\!\eta^{-1}\big(\Ver'\big),\quad
\eta'=\eta~~\hbox{on}~\big([N]\!\cap\!\eta^{-1}(\Ver')\!\big)\!\sqcup\!\Fl',
\end{gather*}
as follows. 
If $|\Fl\!\cap\!\eta^{-1}(v)|\!=\!2$, we take 
$$\Fl'=\Fl\!\cap\!\eta^{-1}\big(\Ver'\big), \quad
\Edg'=\big\{e\!\in\!\Edg\!:e\!\cap\!\eta^{-1}(v)\!=\!\eset\}
\sqcup\big\{\{f\!\in\!\Fl'\!:e_f\!\!\cap\!\eta^{-1}(v)\!\neq\!\eset\}\big\}.$$
If $|\Fl\!\cap\!\eta^{-1}(v)|\!=\!1$, we take 
$$\Fl'=\big\{f\!\in\!\Fl\!:e_f\!\cap\!\eta^{-1}(v)\!=\!\eset\big\},\quad
\Edg'=\big\{e\!\in\!\Edg\!:e\!\cap\!\eta^{-1}(v)\!=\!\eset\}.$$
In this case, there is a unique $f^c\!\in\!\Fl\!\cap\!\eta^{-1}(\Ver')$
with $e_{f^c}\!\cap\!\eta^{-1}(v)\!\neq\!\eset$.
We complete the definition of~$\eta'$ by requiring that
$$\eta'(S_v)=\big\{f^c\big\};$$
the set $S_v$ consists of at most one element in this case
(it is empty in the previous case).\\

\noindent
Let $\Ga$ be a connected $[n]$-valued $N$-marked genus~$g$ weighted graph
with $2\fa(\Ga)\!+\!N\!\ge\!3$. 
We call the connected trivalent $[n]$-valued $N$-marked genus~$g$ graph
\BE{ga0dfn_e}\ov\Ga \equiv 
\big((\ov\g,\ov\mu)\!:\!\ov\Ver\!\lra\!\Z^{\ge0}\!\times\![n],
\ov\eta\!:[N]\!\sqcup\!\ov\Fl\!\lra\!\ov\Ver,\ov\Edg\big)\EE
obtained by forgetting the map~$\d$ and 
repeatedly contracting the non-trivalent vertices of~$\Ga_0$ 
until all vertices become trivalent the \sf{core} of~$\Ga$.
It is independent of the choice of the order in which 
the non-trivalent vertices are contracted and satisfies
\begin{gather*}
\big\{v\!\in\!\Ver\!:g(v)\!\ge\!1\big\}\cup
\big\{v\!\in\!\Ver\!:\{v\}\!=\!\eta(e)~\hbox{for some}~e\!\in\!\Edg\big\}
\subset\ov\Ver\subset\Ver, \\
\ov\Edg\!\cap\!\Edg=\big\{e\!\in\!\Edg\!:\eta(e_f)\!\subset\!\ov\Ver\big\}, \quad
\{f\!\in\!\Fl\!:e_f\!\in\!\ov\Edg\!\cap\!\Edg\big\}\subset\ov\Fl\subset\Fl,\\
(\ov\g,\ov\mu)=(\g,\mu)\big|_{\ov\Ver}, \quad
\ov\eta\big|_{([N]\cap\eta^{-1}(\ov\Ver))\sqcup\ov\Fl}
=\eta\big|_{([N]\cap\eta^{-1}(\ov\Ver))\sqcup\ov\Fl}\,.
\end{gather*}
We call the vertices $\ov\Ver$ of the core $\ov\Ga$ the \sf{special vertices} of~$\Ga$.
The graph~$\ov\Ga$ on the right-hand side of Figure~\ref{graphcore_fig}
is the core of the graph~$\Ga$ on its left-hand side.\\

\noindent
We compute~\e_ref{cZval_e} by breaking each graph $\Ga\!\in\!\cA_{g,N}(n,*)$
at its special vertices into~\sf{strands}
\BE{stranddfn_e}
\Ga_{\fp}\equiv 
\big((\g_{\fp},\mu_{\fp})\!:\!\Ver_{\fp}\!\lra\!\Z^{\ge0}\!\times\![n],
\eta_{\fp}\!:S_{\fp}\!\sqcup\!\Fl_{\fp}\!\lra\!\Ver_{\fp},
\d_{\fp}\!:\Edg_{\fp}\!\lra\!\Z^+\big);\EE
see Figure~\ref{strands_fig}.
The sets~$\Edg_{\fp}$ of the edges of these strands partition the set~$\Edg$
of the edges of~$\Ga$. 
Each edge $e\!\equiv\!\{f,f'\}$ of~$\Ga$ so that $v\!\equiv\!\eta(f)$ is a special vertex
of~$\Ga$
keeps a copy~$v_f$ of~$v$ with
\BE{stranddfn_e2} \g_{\fp}(v_f)=0, \qquad \mu_{\fp}(v_f)=\mu(v).\EE
We also add a marked point labeled by~$\wh{f}$ to this vertex.
Thus, the collection~$S_{\fp}$ of the marked points of the strands~$\Ga_{\fp}$
consists of the original $[N]$-marked points of~$\Ga$ and of a copy~$\wh{f}$
of each flag $f\!\in\!\Fl$ of~$\Ga$ so that $\mu(f)\!\in\!\ov\Ver$.\\

\begin{figure}
\begin{pspicture}(6,-1.1)(10,2)
\psset{unit=.4cm}
\pscircle*(23,2){.2}\rput(23.2,2.6){\smsize{$1$}}
\psline[linewidth=.04](23,2)(19,2)\pscircle*(19,2){.2}
\rput(21,2.5){\smsize{$2$}}\rput(18.8,2.6){\smsize{$3$}}
\psline[linewidth=.04](23,2)(25,.5)\rput(25.6,.6){\smsize{$\bf\wh{f}_3$}}
\pscircle*(23,-2){.2}\rput(23.2,-2.6){\smsize{$1$}}
\psline[linewidth=.04](23,-2)(19,-2)\pscircle*(19,-2){.2}
\rput(21,-2.5){\smsize{$2$}}\rput(18.8,-2.6){\smsize{$3$}}
\psline[linewidth=.04](19,-2)(17,0)\rput(16.5,0){\smsize{$\bf 2$}}
\psline[linewidth=.04](23,-2)(25,-.5)\rput(25.5,-.6){\smsize{$\bf\wh{f}_2$}}
\psline[linewidth=.04](38,2)(30,2)
\pscircle*(30,2){.2}\pscircle*(38,2){.2}
\rput(30,2.7){\smsize{$1$}}\rput(38,2.7){\smsize{$2$}}
\psline[linewidth=.04](38,2)(40,3.5)\rput(40.7,3.8){\smsize{$\bf\wh{f}_6$}}
\psline[linewidth=.04](30,2)(28,3.5)\rput(27.5,3.6){\smsize{$\bf\wh{f}_4$}}
\pscircle*(38,0){.2}\rput(38,-.7){\smsize{$2$}}
\psline[linewidth=.04](38,0)(34,0)\pscircle*(34,0){.2}
\rput(36,.5){\smsize{$3$}}\rput(34,.6){\smsize{$4$}}
\psline[linewidth=.04](34,0)(31.5,-2.5)\pscircle*(31.5,-2.5){.2}
\rput(32.1,-1.3){\smsize{$3$}}\rput(31,-2.5){\smsize{$2$}}
\psline[linewidth=.04](34,0)(36.5,-2.5)\pscircle*(36.5,-2.5){.2}
\rput(36.1,-1.5){\smsize{$2$}}\rput(37.1,-2.5){\smsize{$3$}}
\psline[linewidth=.04](34,0)(30,0)\pscircle*(30,0){.2}
\rput(32,.5){\smsize{$2$}}\rput(30.5,-.6){\smsize{$1$}}
\psline[linewidth=.04](38,0)(40,1.5)\rput(40.7,1.8){\smsize{$\bf\wh{f}_7$}}
\psline[linewidth=.04](30,0)(28,1.5)\rput(27.5,1.6){\smsize{$\bf\wh{f}_5$}}
\pscircle*(44,2){.2}\rput(44,2.7){\smsize{$2$}}
\psline[linewidth=.04](44,2)(48,2)\pscircle*(48,2){.2}
\rput(46,2.5){\smsize{$2$}}\rput(47.7,2.5){\smsize{$1$}}
\psline[linewidth=.04](50.5,3.5)(48,2)\rput(50.8,3.8){\smsize{$\bf 1$}}
\psline[linewidth=.04](51,0)(48,2)\pscircle*(51,0){.2}
\rput(49.6,1.5){\smsize{$1$}}\rput(51.3,.6){\smsize{$3$}}
\psline[linewidth=.04](44,2)(42,.5)\rput(41.6,.5){\smsize{$\bf\wh{f}_1$}}
\pscircle*(44,-2){.2}\rput(44,-2.7){\smsize{$2$}}
\psline[linewidth=.04](44,-2)(48,-2)\pscircle*(48,-2){.2}
\rput(46,-1.5){\smsize{$1$}}\rput(47.7,-2.5){\smsize{$4$}}
\psline[linewidth=.04](51.5,-2)(48,-2)\pscircle*(51.5,-2){.2}
\rput(49.7,-1.5){\smsize{$1$}}\rput(51.8,-2.5){\smsize{$3$}}
\psline[linewidth=.04](44,-2)(42,-.5)\rput(41.6,-.5){\smsize{$\bf\wh{f}_8$}}
\end{pspicture}
\caption{The strands of the graph in the first diagram in Figure~\ref{graphcore_fig}.}
\label{strands_fig}
\end{figure}
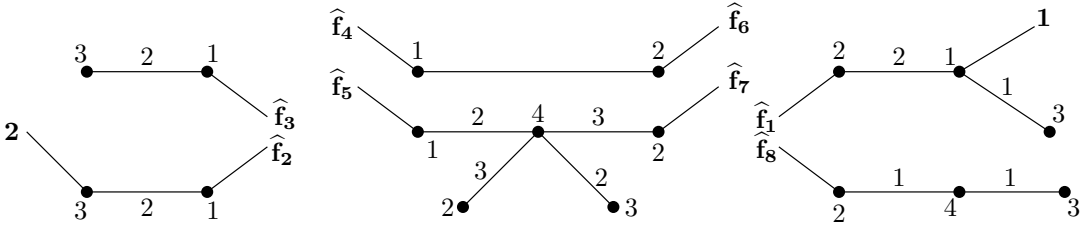

\noindent
There are three types of strands~$\Ga_{\fp}$:
\begin{enumerate}[label=(S\arabic*),leftmargin=*]
\item genus~0 strands with one new marked point;
\item\label{Z2conn_item} genus~0 strands with two new marked points;
\item\label{Z2main_item} genus~0 strands 
with one new marked point and one of the original $N$ marked points.
\end{enumerate}
By~\e_ref{ABothm_e}, each one-pointed strand at a special vertex 
$v\!\in\!\ov\Ver\!\subset\!\Ver$ contributes to
\BE{cZprdfn_e}\cZ^*\big(\hb,\al_j,q\big)
\equiv\sum_{d=1}^{\i}q^d\!\!
\int_{\ov\M_{0,1}(\P^{n-1},d)}
\frac{\ev_1^*\phi_j}{\hb\!-\!\psi_1},\EE
where $j\!=\!\mu(v)\!\in\![n]$ is the label of the vertex $v$ of $\Ga$.
By the dilaton relation \cite[p527]{MirSym},
$$\wt\cZ^*\big(\hb,\al_j,q\big)\equiv
\sum_{d=1}^{\i}q^d\!\!
\int_{\ov\M_{0,2}(\P^{n-1},d)}
\bigg(\frac{\ev_1^*\phi_j}{\hb\!-\!\psi_1}\bigg)
=\hb^{-1}\cZ^*\big(\hb,\al_j,q\big).$$
Each of the two-pointed strands contributes to 
\BE{cZ2stardfn_e}\cZ^*\big(\hb_1,\hb_2,\al_{j_1},\al_{j_2},q\big)
\equiv\sum_{d=1}^{\i}q^d\!\!
\int_{\ov\M_{0,2}(\P^{n-1},d)}
\frac{\ev_1^*\phi_{j_1}}{\hb_1\!-\!\psi_1}\frac{\ev_2^*\phi_{j_2}}{\hb_2\!-\!\psi_2},\EE
where $j_1,j_2\!\in\![n]$ are the labels of the vertices to which the marked points
are attached.
This implies that the power series $\cZ^{(g)}(\un\hb,\un\x,q)$ in~\e_ref{cZdfn_e} 
is determined by the previously computed power series for
one- and two-pointed GW-invariants
and by Hodge integrals over the Deligne-Mumford moduli spaces 
of stable curves.\\

\noindent
While the number of one-marked strands at each node can be arbitrarily large,
as indicated in~\cite[Sections~2.1,2.2]{bcov1} 
it is possible to sum over all possibilities 
for these strands at each special vertex; see Proposition~\ref{pt1sum_prp} below.
On the other hand, the numbers of special vertices, 
of two-pointed strands of type~\ref{Z2conn_item},
and of two-pointed strands of type~\ref{Z2main_item},
are bounded (by $3g\!+\!N$).
Using the Residue Theorem for~$\P^1$, one can then sum up over all 
possibilities of the markings for each of the special nodes.
Thus, the approach of breaking trees at special vertices reduces~\e_ref{cZval_e} 
to a finite sum, with one summand for 
each connected trivalent  $N$-marked genus~$g$ graph.\\

\noindent
The first description of the structure constants $\nc_{g;\p,\b}^{(d)}$ 
after Theorem~\ref{main_thm}
is obtained by breaking every graph $\Ga\!\in\!\cA_{g,N}(n,d)$ at all special vertices
of~$\Ga$.
On the other hand, the second description is obtained 
by breaking each such graph at the special vertex~$\ov\eta(N)$ only.
In addition to the strands~(S1), 
we then obtain strands~$\Ga_{\fp}$ of various genera~$g_{\fp}$ 
with $N_{\fp}\!\in\!\Z^+$ new marked points~$\wh{f}$ corresponding to 
the flags $f\!\in\!\Fl_v(\Ga)$ with $e_f\!\in\!\Edg_{\fp}$
and subsets $S_{\fp}^*\!\subset\![N]$
partitioning the original marked points~$s$ of~$\Ga$ not lying on~$v$
(i.e.~$\eta(s)\!\neq\!v$).
If $N\!\in\!S_{\fp}^*$ for some strand~$\Ga_{\fp}$, then 
$g_{\fp}\!=\!0$, $N_{\fp}\!=\!1$, and $S_{\fp}^*\!=\!\{N\}$.
Setting $g_s\!=\!0$, $N_s\!=\!1$, and $S_s^*\!=\!\{s\}$ for 
$s\!\in\!S_v$, we thus obtain an element
$$\big(g_v,(g_{\fp},S_{\fp}^*,N_{\fp})_{\fp\in[m]}\big)\in \cP_{g,N}^{(m)},$$
where $m\!\in\!\Z^+$ is the number the non-(S1) strands 
and of the marked points in~$S_v$.
With either approach, the main step is summing over all possibilities for 
the strands~(S1), as is done in Proposition~\ref{pt1sum_prp}.

\subsection{Key equivariant inputs}
\label{prelim_subs}

\noindent
With the notation as at the beginning of Section~\ref{equivGW_subs}, let
$$D_{\al}=\prod_{j\neq k}(\al_j\!-\!\al_k) \in\Q[\al_1,\ldots,\al_n]^{S_n}
\subset H_{\T}^*\,.$$ 
If $f\!=\!f(\hb)$ is a rational function in $\hb$ and $\hb_0\!\in\!\P^1$,
let
$$\Res{\hb=\hb_0}\big\{f(\hb)\big\} = \frac{1}{2\pi\I}\oint f(\hb)\tnd\hb\,,$$
where the integral is taken over a positively oriented loop around $\hb\!=\!\hb_0$
containing no other singular points of~$f$, 
denote the residue of $f(\hb)\tnd\hb$ at $\hb\!=\!\hb_0$.
With this definition,
$$\Res{\hb=\i}\big\{f(\hb)\big\}=-\Res{w=0}\big\{w^{-2}f(w^{-1})\big\}.$$
If $f$ involves variables other than $\hb$, 
$\Res{\hb=\hb_0}\big\{f(\hb)\big\}$ will be a function  of such variables.
If $f$ is a power series in $q$ with coefficients that are rational
functions in~$\hb$ and possibly other variables, denote by 
$\Res{\hb=\hb_0}\big\{f(\hb)\big\}$ the  power series in~$q$ obtained by 
replacing each of the coefficients by its residue at $\hb\!=\!\hb_0$.
If $\hb_1,\ldots,\hb_k$ is a collection of distinct points in~$\P^1$, let
$$\Res{\hb=\hb_1,\ldots,\hb_k}\big\{f(\hb)\big\} =\sum_{i=1}^{i=k} \Res{\hb=\hb_i}\big\{f(\hb)\big\}$$
be the sum of the residues at the specified values of $\hb$.\\

\noindent
We denote by
$$\Q_{\al}'\equiv \Q\big[\al,\si_n^{-1},D_{\al}^{-1}\big]^{S_n}
\subset\Q_{\al}$$
the subring of symmetric rational functions in $\al_1,\ldots,\al_n$
with denominators that are products of~$\si_n$ and~$D_{\al}$.
Let 
$$\Q_{\al;\hb,\x}'\equiv 
\Q_{\al}'[\hb,\x^{\pm1}]_{\left\langle(\x+r\hb)^n-\x^n,
\prod\limits_{k=1}^{k=n}(\x-\al_k+r\hb)-\prod\limits_{k=1}^{k=n}(\x-\al_k)
\big|r\in\Z^+\right\rangle}\subset\Q_{\al}(\hb,\x)$$
be the subring of rational functions in $\al_1,\ldots,\al_n$, $\hb$, and~$\x$,
symmetric in $\al_1,\ldots,\al_n$, 
with numerators that are polynomials in $\al_1,\ldots,\al_n$, $\hb$, and~$\x$,
and with denominators that are products~of 
$$\si_n\,,\quad D_{\al}\,,\quad \x\,, \quad (\x\!+\!r\hb)^n\!-\!\x^n\,, \quad 
\prod\limits_{k=1}^{k=n}(\x\!-\!\al_k\!+\!r\hb)-\prod\limits_{k=1}^{k=n}(\x\!-\!\al_k)\,,
\qquad\hbox{with}~~r\in\Z^+.$$
If $R$ is one of the rings $\Q_{\al}'$, $\Q_{\al}'[\x^{\pm1}]$, or $\Q_{\al;\hb,\x}'$
and $f_1$ and $f_2$ are elements of~$R$ or~$R[[q]]$, we will write   $f_1\!\sim\!f_2$
if $f_1\!-\!f_2$ lies in $\cI\cdot R$ or  $\cI\cdot R[[q]]$, respectively.
By the next lemma, certain operations on these rings respect these equivalence relations.

\begin{lmm}\label{ressum_lmm}
\begin{enumerate}[label=(\arabic*),leftmargin=*]
\item If $f\!\in\!\Q_{\al;\hb,\x}'$, there exists 
$g\!\in\!\Q_{\al}'[\x^{\pm1}]$ such that
$$\Res{\hb=0}\big\{f(\hb,\x\!=\!\al_j)\big\}=g(\x\!=\!\al_j)
\qquad\forall\,j\!\in\![n].$$
\item If $g\!\in\!\Q_{\al}'[\x^{\pm1}]$,
$$\Res{\x=0,\i}\left\{\frac{g(\x)}{\prod\limits_{k=1}^{k=n}(\x-\al_k)}\right\}
\in\Q_{\al}'\,.$$
\item For every $p\!\in\!\Z$, 
$$-\Res{\x=0,\i}\left\{\frac{\x^p}{\prod\limits_{k=1}^{k=n}(\x-\al_k)}\right\}
\sim\begin{cases} \wh\si_n^t,&\hbox{if}~p\!=\!n\!-\!1+nt~\hbox{with}~t\in\Z;\\
0,&\hbox{if}~p\!+\!1\not\in n\Z.
\end{cases}$$
\end{enumerate}
\end{lmm}

\begin{proof} This is a modification of \cite[Lemma~4.1]{g0ci},
with the elements of~$\cI$, $\Q_{\al}'$, and $\Q_{\al;\hb,\x}'$ required to 
be symmetric in $\al_1,\ldots,\al_n$.
The proof of \cite[Lemma~4.1]{g0ci} applies verbatim in this setting.
\end{proof}

\noindent
We will also use  the Residue Theorem on~$\P^1$: 
$$\sum_{\x_0\in\P^1}\Res{\x=\x_0}\big\{f(\x)\big\}=0$$
for every  rational function $f\!=\!f(\x)$ on $\P^1\!\supset\!\C$.\\

\noindent
The most fundamental generating function for GW-invariants in the mirror symmetry computations
following~\cite{Gi} is
\BE{wtcZdfn_e}\begin{split}\wt\cZ(\hb,\x,q) &\equiv1+\wt\cZ^*(\hb,\x,q)\\
&\equiv 1+\sum_{d=1}^{\i}q^d\ev_{1*}^d\bigg\{\frac{1}{\hb-\psi_1}\bigg\}
\in H_{\T}^*(\P^{n-1})\big[\big[\hb^{-1},q\big]\big],
\end{split}\EE
where $\ev_1^d\!:\ov\M_{0,2}(\P^{n-1},d)\!\lra\!\P^{n-1}$.
By~\cite{Gi}, $\wt\cZ(\hb,\al_i,q)\!\in\!\Q_{\al}(\hb)$ for all $i\!\in\![n]$.
Thus, we can define
$$\ze(\al_i,q)=\Res{\hb=0}\big\{\ln\big(1+\wt\cZ^*(\hb,\al_i,q)\big)\big\}
\in q\cdot\Q_{\al}\big[\big[q\big]\big]$$
for each $i\!\in\![n]$.\\

\noindent
The proof of \cite[Proposition~4.2]{g0ci} provides power series
$\Psi_0,\Psi_1,\ldots\!\in\!\Q_{\al}'[\x^{\pm1}][[q]]$ such~that 
\begin{gather}\label{cZexp_e0}
\Psi_0(0)=1,\qquad \Psi_b(0)=0~~\forall\,b\!\in\!\Z^+,\qquad
\Psi_b(\x,q)\sim\Phi_b(q/\x^n)\x^{-b}~~\forall\,b\!\in\!\Z^{\ge0},\\
\label{cZexp_e}
\wt\cZ(\hb,\al_i,q)=e^{\ze(\al_i,q)/\hb}\sum_{b=0}^{\i}\Psi_b(\al_i,q)\hb^b
\quad\forall\,i\!\in\![n].
\end{gather}
Furthermore, for every $p\!\in\!\nset$ there exist
$\Psi_{b;0},\Psi_{b;1},\ldots\!\in\!\Q_{\al}'[\x^{\pm1}][[q]]$ such~that
\begin{alignat}{2}
\label{cZpexp_e0}
\Psi_{p;b}(\x,q)&\sim \Phi_{p;b}(q/\x^n)\x^{p-b}
&\qquad&\forall\,b\!\in\!\Z^{\ge0},\\
\label{cZpexp_e}
\cZ_p(\hb,\al_i,q)
&=e^{\ze(\al_i,q)/\hb}\sum_{b=0}^{\i}\Psi_{p;b}(\al_i,q)\hb^b
&\qquad&\forall\,i\!\in\![n]\,.
\end{alignat}

\begin{lmm}\label{pt2_lmm}
There exists a collection  
$\{\cC_{p_-p_+}\}_{p_{\pm}\in\nset}\!\subset\!\Q[\un\al]^{S_n}[[q]]$
such~that 
\begin{equation*}\begin{split}
&\Res{\hb_-=0}
\left\{\frac{e^{-\frac{\ze(\al_{i_-},q)}{\hb_-}}}
{\hb_-^{1+b_-}}
\cZ(\hb_-,\hb_+,\al_{i_-},\al_{i_+},q)\right\}\\
&\hspace{.5in}
=\sum_{b=0}^{b=b_-}\!\!\left( \frac{(-1)^b}{\hb_+^{b+1}}
\!\!\!\!\sum_{p_-,p_+\in\nset}\!\!\!\!\!\!\!
\cC_{p_-p_+}(q)\Psi_{p_-;b_--b}(\al_{i_-},q)\cZ_{p_+}(\hb_+,\al_{i_+},q)\!\!\right)
\end{split}\end{equation*}
for all $b_-\!\in\!\Z^{\ge0}$ and $i_-,i_+\!\in\![n]$ and 
\BE{p2cC_e} \cC_{p_-p_+}(q)\sim
\begin{cases}
1,&\textnormal{if}~p_-\!+\!p_+\!=\!n\!-\!1;\\
0,&\textnormal{otherwise}.\end{cases}\EE
\end{lmm}

\begin{proof}
The proof of \cite[Lemma~4.4]{g0ci} establishes the present lemma
(the equivalence~$\sim$ in the statement of \cite[Lemma~4.4]{g0ci}
is taken with respect to the ideal~$\cI$ inside of the entire ring~$\Q[\al]$). 
\end{proof}

\begin{crl}\label{pt2_crl} For all $b_-,b_+\!\in\!\Z^{\ge0}$ and $i_-,i_+\!\in\![n]$,
\begin{equation*}\begin{split}
&\Res{\hb_+=0}\left\{\Res{\hb_-=0}
\left\{\frac{e^{-\frac{\ze(\al_{i_-},q)}{\hb_-}-\frac{\ze(\al_{i_+},q)}{\hb_+}}}
{\hb_-^{1+b_-}\hb_+^{1+b_+}}
\cZ(\hb_-,\hb_+,\al_{i_-},\al_{i_+},q)\right\}\right\}\\
&\hspace{.5in}
=\sum_{b'=0}^{b'=b_-}\sum_{p_-,p_+\in\nset}\!\!\!\!\!\!\!
(-1)^{b'}\cC_{p_-p_+}(q)\Psi_{p_-;b_--b'}(\al_{i_-},q)\Psi_{p_+;b_++1+b'}(\al_{i_+},q).
\end{split}\end{equation*}
\end{crl}

\begin{proof}
This follows immediately from Lemma~\ref{pt2_lmm} and~\e_ref{cZpexp_e}.
\end{proof}

\noindent
Let $g\!\in\!(\Z)^{\ge0}$.
By the $\T$-equivariance of Euler's sequence for $\P^{n-1}$ and~\e_ref{CgIdfn_e0},
there exist \hbox{$h_{g,n;I}\!\in\!\cI[\x]$} with $I\!\in\!(\Z^{\ge0})^g$ such~that 
\BE{CgIdfn_e}\begin{split}
\E\big(\bE_g^*\!\otimes\!T_{P_i}\P^{n-1}\big)
&=\bigg(\sum_{I\in(\Z^{\ge0})^g}\!\!\!\!\!\!
\big(C_{g,n;I}\x^{(n-1)g-\|I\|}\!+\!h_{g,n;I}(\x)\big)\la_{g;I}\bigg)
\bigg|_{\x=\al_i}\\
&\qquad\in H_{\T}^*\big(\ov\cM_{g;m}\!\times\!P_i\big)=
 H^*\big(\ov\cM_{g;m}\big)\big[\un\al\big]
\end{split}\EE
for all $m\!\in\!\Z^{\ge0}$ with $2g\!+\!m\!\ge\!0$ and $i\!\in\![n]$.
For $m\!\in\!\Z^{\ge0}$ and \hbox{$k\!\in\![m]$},
we denote the $k$-th component of $\b\!\in\!(\Z^{\ge0})^m$ by $b_k\!\in\!\Z^{\ge0}$.
If $m\!\in\!\Z^{\ge0}$ with $2g\!+\!m\!\ge\!0$, $I\!\in\!(\Z^{\ge0})^g$,
$\b\!\in\!(\Z^{\ge0})^m$,  and $i\!\in\![n]$, let
\BE{wtcZgbIdfn_e}\begin{split}
&\wt\cZ_{\b,I}^{(g)}(\al_i,q)=
\big(C_{g,n;I}\al_i^{(n-1)g-\|I\|}\!+\!h_{g;n;I}(\al_i)\!\big)\\
&\qquad\times\!\!\sum_{m'=0}^{\i}\!\!\!\!\!\!
\sum_{\begin{subarray}{c}\b'\in(\Z^{\ge0})^{m'} \\
|\b|+|\b'|=\mu_g(I)+m+m'\end{subarray}}
\hspace{-.37in}
\Bigg(\!\frac{\llrr{\la_{g;I};\wt\tau_{\b\b'}}}{m'!}
\!\!\!\!\prod_{k\in[m']}\!
\Res{\hb=0}\Big\{\frac{(-\hb)^{-b_k'}}{b_k'!}\wt\cZ^*(\hb,\al_i,q)\!\Big\}\!\!\Bigg).
\end{split}\EE
Each residue above is an element of $\Q_{\al}[[q]]$.
Since the power series $\wt\cZ^*(\hb,\x,q)$ has no $q$-constant term,
the above sum is finite in each $q$-degree.
By Section~\ref{RecFormComp_subs}, the power series $\wt\cZ_{\b,I}^{(g')}(\al_i,q)$ 
describe the contributions to~\e_ref{cZval_e} of the strands~(S1) of the graphs~$\Ga$
with a fixed core~$\ov\Ga$ at a vertex~$v$ of~$\ov\Ga$ 
 with $|\ov\Fl_v(\ov\Ga)|\!=\!m$ and $\ov\mu(v)\!=\!i$.

\begin{prp}\label{pt1sum_prp}
Let $g,m\!\in\!\Z^{\ge0}$ with $2g\!+\!m\!\ge\!3$ and $I\!\in\!(\Z^{\ge0})^g$. 
There exist \hbox{$\Psi_{I;\c}^{(g,\un\ep)}\!\in\!\Q_{\al}'[\x^{\pm1}][[q]]$}
with $\c\!\in\!(\Z^{\ge0})^{\i}$ and $\un\ep\!\in\!(\Z^{\ge0})^m$ such~that 
\BE{cZmB_e}\begin{split}
\wt\cZ_{\b,I}^{(g)}(\al_i,q)&=
\!\!\!\!\!\sum_{\c\in(\Z^{\ge0})^{\i}}
\!\!\!
\sum_{\begin{subarray}{c} \un\ep\in (\Z^{\ge0})^m\\
|\un\ep|\le\mu_g(I)+m\\ \ep_k\le b_k~\forall k\in[m]\end{subarray}} \hspace{-.15in}
\Bigg(\!\!(-1)^{\mu_g(I)+m-\|\c\|}\Psi_{I;\c}^{(g,\un\ep)}(\al_i,q)  \\
&\hspace{.8in} \times\ze(\al_i,q)^{|\b|-(\mu_g(I)+m-\|\c\|)}
\binom{|\b|\!-\!|\un\ep|}{\mu_g(I)\!+\!m\!-\|\c\|\!-\!|\un\ep|}
\!\!\prod_{k=1}^m\!\!\frac{b_k!}{(b_k\!-\!\ep_k)!}\!\!\Bigg)
\end{split}\EE
for all $\b\!\in\!(\Z^{\ge0})^m$ and $i\!\in\![n]$ and
\BE{PsimcPhi_e}
\Psi_{I;\c}^{(g,\un\ep)}(\x,q)\sim 
\frac{\Phi_{I;\c}^{(g,\un\ep)}(q/\x^n)}{\Phi_0(q/\x^n)^m}\x^{(n-1)g-\|I\|-\|\c\|}\,.\EE
\end{prp}  

\begin{proof}
By the first two statements in~\e_ref{cZexp_e0}, \e_ref{cZexp_e}, 
and Proposition~\ref{HodgeIntGS_prp}, \e_ref{cZmB_e} holds with 
\BE{Psimcdfn_e}\begin{split}
\Psi_{I;\c}^{(g,\un\ep)}(\x,q)&=
\frac{(-1)^{\mu_g(I)+m+|\c|}
(C_{g,n;I}\x^{(n-1)g-\|I\|}\!+\!h_{g;n;I}(\x))
\wh{A}_{I;\c}^{(g,\un\ep)}}{\Psi_0(\x,q)^{2g-2+m}}\\
&\hspace{2in}
\times\prod_{r=1}^{\i}
\frac{1}{c_r!}\bigg(\frac{\Psi_r(\x,q)}{(r\!+\!1)!\,\Psi_0(\x,q)}\!\!\bigg)^{\!\!c_r}\,\,.
\end{split}\EE
By the last statement in~\e_ref{cZexp_e0} and \e_ref{PsimiCdfn_e}, 
$\Psi_{I;\c}^{(g,\un\ep)}$ satisfies~\e_ref{PsimcPhi_e}.
\end{proof}

\section{Proof of Theorem~\ref{equiv_thm}}
\label{maincomp_sec}

\noindent
We prove Theorem~\ref{equiv_thm}, with each of the two definitions 
of the structure constants~$\nc_{g;\p,\b}^{(d)}$, 
by summing up the contributions of the $\T$-fixed loci~$\cZ_{\Ga}$ of 
$\ov\M_{g,N}(\P^{n-1},d)$ as in~\e_ref{GPthm_e} over all possibilities for the graph~$\Ga$
as in~\e_ref{decortgraphdfn_e}. 
As outlined in Section~\ref{pfoutline_subs}, this will be done by breaking each~$\Ga$
(and correspondingly each fixed locus~$\cZ_{\Ga}$) either at every special vertex of~$\Ga$
or only at the special vertex $v\!=\!\ov\mu(N)$.

\subsection{Some preparation and notation}
\label{prelimcomp_subs}

\noindent
Sections~\ref{ClFormComp_subs} and~\ref{RecFormComp_subs} describe coefficients
$\cC_{g;\p,\b}^{(d)}\!\in\!\Q_{\al}$ by a closed formula and via a recursion,
respectively, so that~\e_ref{equivthm_e} is satisfied and 
\e_ref{equivthm_e2} holds with $t\!\in\!\Z$ instead of~$\Z^{\ge0}$.
These coefficients are symmetric in~$\al_1,\ldots,\al_n$
(this is also implied by the proof of the first claim below).
The full statement of Theorem~\ref{equiv_thm} then follows from the next observation.

\begin{lmm}\label{cCpolyn_lmm}
Let $g,n,N\!\in\!\Z^{\ge0}$ be as in Theorems~\ref{main_thm} and~\ref{equiv_thm}.
If $\cC_{g;\p,\b}^{(d)}\!\in\!\Q_{\al}$ are such that~\e_ref{equivthm_e} is satisfied,
then $\cC_{g;\p,\b}^{(d)}\!\in\!\Q[\al]$.
If in addition~\e_ref{equivthm_e2} holds with $t\!\in\!\Z$ instead of~$\Z^{\ge0}$
for some $\nc_{g;\p,\b}^{(d,t)}\!\in\!\Q$,
then it holds as stated with same coefficients~$\nc_{g;\p,\b}^{(d,t)}$.
\end{lmm}

\begin{proof}
Let $\b\!\in\!(\Z^{\ge0})^N$ and $d\!\in\Z^{\ge0}$.
By \e_ref{cZpdfn_e}-\e_ref{equivthm_e}, the coefficient~of 
$$q^d \prod_{s=1}^{s=N}\!\big((\hb_s^{-1})^{b_s+1}\big)$$
in the power series $\cZ^{(g)}(\un\hb,\un\x,q)$ is
\BE{cZcoeff_e}\begin{split}
&\LRbr{\cZ^{(g)}(\un\hb,\un\x,q)}_{\un\hb^{-1},q;\b+\mathbf1,d}
=\sum_{\p\in\nset^N}\!\!\!\cC_{g;\p,\b}^{(d)}\un\x^{\p}\\
&\hspace{.2in}
+\!\!\!\sum_{d'\in\dset}
\sum_{\begin{subarray}{c}\bfd\in(\Z^{\ge0})^N\\ |\bfd|=d-d'\end{subarray}}
\sum_{\p\in\nset^N}
\sum_{\begin{subarray}{c} \b'\in(\Z^{\ge0})^N\\ b_s'\le b_s\,\forall s\end{subarray}} 
\!\!\!\!\!\cC_{g;\p,\b'}^{(d')}
\!\!\prod_{s=1}^{s=N}\!\! \LRbr{\cZ_{p_s}(\hb_s,\x_s,q) }_{\hb_s^{-1},q;b_s-b_s',d_s}\,,
\end{split}\EE
where $\LRbr{\cZ_p(\hb,\x,q) }_{\hb^{-1},q;b,d'}$
is the coefficient of $q^{d'}(\hb^{-1})^b$ in
$$\cZ_p(\hb,\x,q)\in H_{\T}^*\big(\P^{n-1}\big)\big[\big[\hb^{-1},q\big]\big].$$
Since $H_{\T}^*(\P^{n-1})$ and $H_{\T}^*(\P_N^{n-1})$ are free modules over $\Q[\al]$
with bases $\{\x^p\}_{p\in\nset}$ and $\{\un\x^{\p}\}_{\p\in\nset^N}$, respectively,
and
$$\LRbr{\cZ_p(\hb,\x,q) }_{\hb^{-1},q;b,d'}\in H_{\T}^*(\P^{n-1}),\qquad
\LRbr{\cZ^{(g)}(\un\hb,\un\x,q)}_{\un\hb^{-1},q;\b+\mathbf1,d}\in H_{\T}^*(\P_N^{n-1})$$
by \e_ref{cZpdfn_e} and \e_ref{cZdfn_e}, respectively, 
\e_ref{cZcoeff_e} and induction on~$d$ imply that 
$\cC_{g;\p,\b}^{(d)}\!\in\!\Q[\al]$.\\

\noindent
By \cite[Lemma~3.3]{bcov1}, the ideal $\cI\!\subset\!\Q[\al]^{S_n}$ does not contain any
power of~$\si_nD_{\al}$.
Along with the algebraic independence of elementary symmetric polynomials 
$\si_1,\ldots,\si_n$, this implies the second claim of the lemma.
\end{proof}

\noindent
For $g\!\in\!\Z^{\ge0}$ and a finite set $S$ with $2g\!+\!|S|\!\ge\!3$, 
we denote by $\ov\cM_{g,S}$ the Deligne-Mumford moduli space of
stable $S$-marked genus~$g$ curves and~by
$$\bE_g\lra \ov\cM_{g,S}$$
the Hodge vector bundle of holomorphic differentials.
For each $f\!\in\!S$, let
$$L_f\lra \ov\cM_{g,S}$$
be the universal tangent line bundle for the marked point labeled by~$f$.\\

\noindent
For $g\!\in\!\Z^{\ge0}$, a finite set $S$, and $d\!\in\!\Z^+$, 
we denote by $\ov\M_{g,S}(\P^{n-1},d)$
the moduli space of stable $S$-marked genus~$g$ degree~$d$ maps to~$\P^{n-1}$. 
For each $f\!\in\!S$, let
$$\psi_f\equiv c_1(L_f^*)\in H^2\big(\ov\M_{g,S}(\P^{n-1},d)\big)$$
be the first Chern class of 
the universal cotangent line bundle~$L_f$ over  $\ov\M_{g,S}(\P^{n-1},d)$
for the marked point labeled by~$f$.\\

\noindent
For an $[n]$-valued $N$-marked graph $\ov\Ga$ as in~\e_ref{ga0dfn_e}, let
$$\cA(\ov\Ga)=(\Z^{\ge0})^{\ov\Fl(\ov\Ga)}\!\!\times\!\!\!
\prod_{v\in\ov\Ver}\!\!\!(\Z^{\ge0})^{\ov\g(v)}\,.$$
Let $\Ga$ be a connected $[n]$-valued $N$-marked genus~$g$ weighted graph
as in~\e_ref{decortgraphdfn_e} and $\ov\Ga$ be its core as in~\e_ref{ga0dfn_e}.
Define
\begin{gather*}
\Fl(\Ga)=\bigsqcup_{v\in\ov\Ver}\!\!\!\Fl_v(\Ga),\qquad
\ov\Fl(\Ga)=\bigsqcup_{v\in\ov\Ver}\!\!\!\ov\Fl_v(\Ga)
=\big([N]\!\cap\!\eta^{-1}(\ov\Ver)\!\big)\!\sqcup\!\Fl(\Ga),\\
\Fl^*(\Ga)= \big([N]\!-\!\eta^{-1}(\ov\Ver)\!\big)\!\sqcup\!\ov\Edg,\qquad
\ov\Fl^*(\Ga)= [N]\!\sqcup\!\ov\Edg
=\big([N]\!\cap\!\eta^{-1}(\ov\Ver)\!\big)\!\sqcup\!\Fl^*(\Ga),\\
\cA^{\star}(\Ga)=\big\{(\b,(I_v)_{v\in\ov\Ver})\!\in\!
(\Z^{\ge0})^{\ov\Fl(\Ga)}\!\!\times\!\!\!\prod_{v\in\ov\Ver}\!\!\!(\Z^{\ge0})^{g_v}\!:
|\b|_{\ov\Fl_v(\Ga)}\!=\!\mu_{g_v}(I_v)\!+\!|\ov\Fl_v(\Ga)|~\forall\,v\!\in\!\ov\Ver\big\}.
\end{gather*}
For $v\!\in\!\ov\Ver$ and $I\!\in\!(\Z^{\ge0})^{g_v}$, let 
$$\cA_{v;I}^{\star}(\Ga)=\big\{\b\!\in\!(\Z^{\ge0})^{\ov\Fl_v(\Ga)}\!:
|\b|\!=\!\mu_{g_v}(I)\!+\!|\ov\Fl_v(\Ga)|\big\}.$$
We denote the components of an element $(\b,\bfI)$ of~$\cA^{\star}(\Ga)$ by 
$b_f\!\in\!\Z^{\ge0}$ for $f\!\in\!\ov\Fl(\Ga)$ and
\hbox{$I_v\!\in\!(\Z^{\ge0})^{g_v}$}  for $v\!\in\!\ov\Ver$.
For $f\!\in\!\Fl$, let
$$\mu(f)=\mu\big(\eta(f)\big)\in[n] \quad\hbox{and}\quad
\mu_c(f)=\mu\big(\eta(f')\big)\in[n]  ~~\hbox{if}~ e_f=\{f,f'\}$$
be the $\mu$-values at the two flags contained in the edge~$e_f$.

\subsection{The closed formula approach}
\label{ClFormComp_subs}

\noindent
We first break $\Ga$ and $\cZ_{\Ga}$ at all vertices~$v$ of $\ov\Ver\!\subset\!\Ver$ 
as described in Section~\ref{pfoutline_subs}.
The set of strands of type~(S2) is naturally indexed by 
the edges~$\ov\Edg$ of~$\ov\Ga$. 
By the constructions of the core and of the strands~$\Ga_{\fp}$ in Section~\ref{pfoutline_subs},
$$f_e^{\pm}\in \Fl_{\eta(f_e^{\pm})}(\Ga) \qquad\forall~
e\!\equiv\!\big\{f_e^+,f_e^-\big\}\in\ov\Edg$$
and the set~$S_e$ of the marked points on the strand~$\Ga_e$ corresponding to 
an edge~$e$ as above is $\{\wh{f}_e^+,\wh{f}_e^-\}$.
The set of strands of type~(S3) is naturally indexed by the subset
\hbox{$[N]\!-\!\eta^{-1}(\ov\Ver)$} of the original marked points. 
By the construction of the core in Section~\ref{pfoutline_subs},
for each such $s\!\in\![N]$ there exists a unique flag $f_s\!\in\!\Fl_{\ov\eta(s)}(\Ga)$ 
so that the edge $e_{f_s}\!\in\!\Edg$ splits the graph~$\Ga$ into~two,
one containing the vertex~$\ov\eta(s)$ and the other containing the vertex~$\eta(s)$.
The set~$S_s$ of the marked points on the corresponding strand~$\Ga_s$ 
is $\{s,\wh{f}_s\}$.
All of the flags $f_s$ with $s\!\in\![N]\!-\!\eta^{-1}(\ov\Ver)$ and
$f_e^+,f_e^-$ with $e\!\in\!\ov\Edg$ are distinct.
The set of strands of type~(S1) is naturally indexed by the complement
\BE{FlprGadfn_e}\Fl'(\Ga)\equiv\Fl(\Ga)-
\big\{f_s\!:s\!\in\![N]\!-\!\eta^{-1}(\ov\Ver)\big\}-
\big\{f_e^{\pm}\!:e\!\in\!\ov\Edg\big\}\subset\Fl\EE
of such flags inside of all flags of~$\Ga$ at the vertices $\ov\Ver\!\subset\!\Ver$.
The set~$S_f$ of the marked points on the strand~$\Ga_f$ corresponding to $f\!\in\!\Fl'(\Ga)$
is~$\{\wh{f}\}$.\\

\noindent
The~set
$$\Fl^{\dag}(\Ga)\equiv \Fl^*(\Ga)\sqcup \Fl'(\Ga)$$ 
of all strands~$\Ga_{\fp}$ of~$\Ga$ is thus a quotient of~$\Fl(\Ga)$
so that two flags $f_1,f_2\!\in\!\Fl(\Ga)$ determine the same element 
$f_1^{\dag}\!=\!f_2^{\dag}$ of $\Fl^{\dag}(\Ga)$
if and only if the marked points labeled by~$\wh{f}_1$ and~$\wh{f}_2$ lie on the same strand.
The~set
\BE{ovFldagdfn_e}
\ov\Fl^{\dag}(\Ga)\equiv \big([N]\!\cap\!\eta^{-1}(\ov\Ver)\!\big)\!\sqcup\!\Fl^{\dag}(\Ga)
=[N]\!\sqcup\!\ov\Edg\!\sqcup\!\Fl'(\Ga) \EE
is similarly a quotient of the set $\ov\Fl(\Ga)$.
For $\fp\!\in\!\ov\Fl^{\dag}(\Ga)$, we write $f\!\in\!\fp$ if
$f\!\in\!\ov\Fl(\Ga)$ and $f^{\dag}\!=\!\fp$.
Denote by $S_{\fp}^*\!\subset\![N]$ the empty set if $\fp\!\not\in\![N]$
and $\{\fp\}$ if $\fp\!\in\![N]$.
For each $v\!\in\!\ov\Ver$, let 
$$\Fl_v'(\Ga)\equiv \Fl_v(\Ga)\!\cap\!\Fl'(\Ga)$$
be the subset of strands of type~(S1) arising from~$v$.\\

\noindent
The fixed locus $\cZ_{\Ga}$ corresponding to~$\Ga$ and 
the Euler class of the virtual normal bundle of~$\cZ_{\Ga}$ are given~by
\begin{gather}
\label{Zreg_e5a}
\cZ_{\Ga}=\prod_{v\in\ov\Ver}\!\!\ov\cM_{g_v,\ov\Fl_v(\Ga)}\times
\prod_{\fp\in\Fl^{\dag}(\Ga)}\!\!\!\!\!\!\!\cZ_{\Ga_{\fp}},\\
\label{Zreg_e5b}
\frac{\prod\limits_{v\in\ov\Ver}\!\!\!\E(T_{P_{\mu(v)}}\P^{n-1})}{\E(\N\cZ_{\Ga}^{\vir})}
=\prod_{v\in\ov\Ver}\!\!\!\!\E(\bE_{g_v}^*\!\otimes\!T_{P_{\mu(v)}}\P^{n-1})
\!\!\!\!\!
\prod_{\fp\in\Fl^{\dag}(\Ga)}\!\frac{1}{\E(\N\cZ_{\Ga_{\fp}})}
\prod_{f\in\Fl(\Ga)}\!\!\!\!
\frac{\E(T_{P_{\mu(f)}}\P^{n-1})}{\hb_f'\!-\!\psi_{\wh{f}}}\,,
\end{gather}
where
$$\hb_f' \equiv c_1(L_f) \in H^*\big(\ov\cM_{g_v,\ov\Fl_v(\Ga)}\big)
\qquad\forall\,f\!\in\!\Fl_v(\Ga),\,v\!\in\!\ov\Ver\,.$$ 
By \cite[Section~27.2]{MirSym},
$$\psi_{\wh{f}}|_{\cZ_{\Ga_{\fp}}}
=\frac{\al_{\mu_c(f)}\!-\!\al_{\mu(f)}}{\d(e_f)}
\qquad\forall\,f\!\in\!\Fl(\Ga)\,.$$
Thus,
\BE{cMint_e}\begin{split}
&\int_{\ov\cM_{g_v,\ov\Fl_v(\Ga)}}\!\!\!\!\la_{g_v;I}
\Bigg\{\!\!\Bigg(\prod_{f\in\Fl_v(\Ga)}\!\frac{1}{\hb_f'\!-\!\psi_{\wh{f}}}\Bigg)
\!\!\Bigg(\prod_{s\in S_v}\!\!\frac{1}{\hb_s\!-\!\psi_s}\Bigg)\!\!\Bigg\}\\
&\hspace{.2in}
=(-1)^{|\Fl_v(\Ga)|}\!\!\!\sum_{\b\in(\Z^{\ge0})^{\ov\Fl_v(\Ga)}}\! 
\int_{\ov\cM_{g_v,\ov\Fl_v(\Ga)}}\!\!\la_{g_v;I}\Bigg\{\!\!
\Bigg(\prod_{f\in\Fl_v(\Ga)}\!\!\!\!\!\psi_{\wh{f}}^{-b_f-1}\hb_f'^{\,b_f}\!\!\Bigg)\!\!
\Bigg(\prod_{s\in S_v}\!\!\!\hb_s^{-b_s-1}\psi_s^{b_s}\!\!\Bigg)
\!\!\Bigg\}\\
&\hspace{.2in}
=\!\!\!\sum_{\b\in\cA_{v;I}^{\star}(\Ga)}\!\!
\Bigg\{\!\!\bllrr{\la_{g_v;I};\tau_{\b}}
\!\Bigg(\prod_{f\in\Fl_v(\Ga)}\!\!\! 
\bigg(\frac{\al_{\mu(f)}\!-\!\al_{\mu_c(f)}}{\d(e_f)}\bigg)^{\!\!-b_f-1}\!\Bigg)\!\!
\Bigg(\prod_{s\in S_v}\!\!\hb_s^{-b_s-1}\!\!\Bigg)
\!\Bigg\}
\end{split}\EE
for all $I\!\in\!(\Z^{\ge0})^{g_v}$ and $v\!\in\!\ov\Ver$.\\

\noindent  
Combining \e_ref{Zreg_e5a}-\e_ref{cMint_e}
with \e_ref{ETPn_e}, \e_ref{CgIdfn_e}, and \e_ref{phidfn_e}, we obtain
\BE{decomp2_e2}\begin{split}
&\bigg(\prod_{v\in\ov\Ver}
\prod_{k\neq\mu(v)}\!\!\!\!(\al_{\mu(v)}\!-\!\al_k)\!\!\bigg)\!\!
\int_{\cZ_{\Ga}}\!
\frac{1}{\E(\N\cZ_{\Ga}^{\vir})}\!\!
\prod_{s=1}^{s=N}\!\!\frac{\ev_s^*\phi_{i_s}}{\hb_s\!-\!\psi_s}\\
&\hspace{.1in}
=\hspace{-.1in}\sum_{(\b,\bfI)\in\cA^{\star}(\Ga)}
\!\!\left\{\prod_{v\in\ov\Ver}\!\!\!
\big(C_{g_v,n;I_v}\al_{\mu(v)}^{(n-1)g_v-\|I_v\|}\!+\!h_{g_v;n;I_v}(\al_{\mu(v)})\!\big)
\llrr{\la_{g_v;I_v};\wt\tau_{\b}}
\right.\\
&\hspace{.9in} \left.\times
\!\!\prod_{\fp\in\ov\Fl^{\dag}(\Ga)}\!\!\!
\Bigg(\!
\prod_{f\in\fp}\!\frac{1}{b_f!}\!
\bigg(\frac{\al_{\mu(f)}\!-\!\al_{\mu_c(f)}}{\d(e_f)}\bigg)^{\!\!-b_f-1}
\!\!\!\int_{\cZ_{\Ga_{\fp}}}\!\!\!
\frac{\prod\limits_{f\in\fp}\!\!\ev_f^*\phi_{\mu(f)}}
{\E(\N\cZ_{\Ga_{\fp}})}
\prod_{s\in S_{\fp}^*}\!\!\frac{\ev_s^*\phi_{i_s}}{\hb_s\!-\!\psi_s}
\!\!\Bigg)\!\!\right\},
\end{split}\EE
where
$$\prod_{f\in\fp}\!\frac{1}{b_f!}\!
\bigg(\frac{\al_{\mu(f)}\!-\!\al_{\mu_c(f)}}{\d(e_f)}\bigg)^{\!\!-b_f-1}
\!\!\!\!\int_{\cZ_{\Ga_{\fp}}}\!\!\!
\frac{\prod\limits_{f\in\fp}\!\!\ev_f^*\phi_{\mu(f)}}
{\E(\N\cZ_{\Ga_{\fp}})} \!\prod_{s\in S_{\fp}^*}\!\!
\frac{\ev_s^*\phi_{i_s}}{\hb_s\!-\!\psi_s}
\equiv\frac{1}{b_{\fp}!}
\bigg(\!\!\hb_{\fp}^{-b_{\fp}-1}\!\!
\prod_{k\neq i_{\fp}}\!\!\big(\al_{\mu(\eta(\fp))}\!-\!\al_k\big)\!\!\bigg)$$
if $\fp\!\in\![N]\!\cap\!\eta^{-1}(\ov\Ver)\!\subset\!\ov\Fl^{\dag}(\Ga)\!\cap\!\ov\Fl(\Ga)$.
The equality in~\e_ref{decomp2_e2} holds after taking into account the automorphism groups;
this is done below after summing over all possibilities for the strands~$\Ga_{\fp}$.

\begin{lmm}\label{S1sum_lmm}
If $\fp\!\in\!\Fl'(\Ga)$, then 
\BE{Zreg_e8}\begin{split}
&\sum_{\Ga_{\fp}}q^{|\Ga_{\fp}|}
\prod_{f\in\fp}\!\!
\bigg(\frac{\al_{\mu(f)}\!-\!\al_{\mu_c(f)}}{\d(e_f)}\bigg)^{\!\!-b_f-1}
\!\!\!\int_{\cZ_{\Ga_{\fp}}}\!\!\!
\frac{\prod\limits_{f\in\fp}\!\!\ev_f^*\phi_{\mu(f)}}
{\E(\N\cZ_{\Ga_{\fp}})}
\prod_{s\in S_{\fp}^*}\!\frac{\ev_s^*\phi_{i_s}}{\hb_s\!-\!\psi_s}\\
&\hspace{2.5in}
=\Res{\hb_{\fp}=0}\Big\{\!\big(\!-\!\hb_{\fp}\big)^{-b_{\fp}}
\wt\cZ^*(\hb_{\fp},\al_v',q)\Big\}\,,
\end{split}\EE  
where the sum is taken over all possibilities for the strand $\Ga_{\fp}$, 
leaving the vertex $v\!\equiv\!\eta(\fp)$, with $\al_v'\!\equiv\!\al_{\mu(v)}$
fixed.
\end{lmm}

\begin{proof}
In this case, $S_{\fp}^*\!=\!\eset$ and
each $\Ga_{\fp}$ is a connected 1-marked genus~0 graph.
Thus, the claim of this lemma is \cite[(2.14)]{bcov1} with $\E(\V_0')\!=\!1$
and slightly different notation 
(the left-hand side of \cite[(2.14)]{bcov1} is also missing $q^{|\Ga_e|}$).
\end{proof}

\begin{lmm}\label{S3sum_lmm}
If $\fp\!\in\![N]$, then 
\BE{Zreg_e8b}\begin{split}
&\sum_{\Ga_{\fp}}q^{|\Ga_{\fp}|}
\prod_{f\in\fp}\!\
\bigg(\frac{\al_{\mu(f)}\!-\!\al_{\mu_c(f)}}{\d(e_f)}\bigg)^{\!\!-b_f-1}
\!\!\!\int_{\cZ_{\Ga_{\fp}}}\!\!\!
\frac{\prod\limits_{f\in\fp}\!\!\ev_f^*\phi_{\mu(f)}}
{\E(\N\cZ_{\Ga_{\fp}})}
\prod_{s\in S_{\fp}^*}\!\!\frac{\ev_s^*\phi_{i_s}}{\hb_s\!-\!\psi_s}\\
&\hspace{1.8in}
=-\Res{\hb_f=0}\Big\{\!\big(-\hb_f\big)^{-b_f-1}\!
\cZ\big(\hb_f,\hb_{\fp},\al_v',\al_{i_{\fp}},q\big)\!\Big\},
\end{split}\EE
where the sum is taken over 
\begin{enumerate}[label=$\bullet$,leftmargin=*]

\item all possibilities for the strand $\Ga_{\fp}$,  
leaving the vertex $v\!\equiv\!\ov\eta(\fp)$, with $\al_v'\!\equiv\!\al_{\mu(v)}$ 
fixed and $|\Ga_{\fp}|\!\in\!\Z^+$, 
and carrying the marked point indexed by~$\fp$, 

\item and  the special case~$\Ga_{\fp}$ corresponding to the case $\eta(\fp)\!=\!v$.  

\end{enumerate}
\end{lmm}

\begin{proof}
In this case, $S_{\fp}^*\!=\!\{\fp\}$,
$\eta(f)\!=\!v$ for the unique element $f\!\in\!\fp$,
and each $\Ga_{\fp}$ is a connected 2-marked genus~0 graph.
There is a unique edge $e\!\equiv\!e_{f_{\fp}}$ from the vertex $v\!\equiv\!\ov\eta(\fp)$
which is contained in~$\Ga_{\fp}$.
By~(3.24) and Lemma~1.2 in~\cite{bcov0} with $m\!=\!2$ and $\E(\V_0'')\!=\!1$,
the left-hand side of~\e_ref{Zreg_e8b} summed over~$\Ga_{\fp}$
with $\d(e)\!=\!d$ for $d\!\in\!\Z^+$
and $\mu_c(f_{\fp})\!=\!j$ for $j\!\in\![n]\!-\!\mu(v)$ is the residue~of 
$$\big(\!-\!\hb_f\big)^{-b_f-1}
\cZ^*\big(\hb_f,\hb_{\fp},\al_v',\al_{i_{\fp}},q\big)\tnd\hb_f
\in \Q_{\al}\big(\hb_f,\hb_{\fp}\big)\big[\big[q\big]\big]\tnd\hb_f$$
at $\hb_f\!=\!(\al_j\!-\!\al_v')/d$.
Furthermore, this meromorphic 1-form on~$\P^1$ has no poles outside of
$\hb_f\!=\!(\al_j\!-\!\al_v')/d$ with $d\!\in\!\Z^+$ and $j\!\in\![n]\!-\!\mu(v)$
and $\hb\!=\!0,\i$.\\

\noindent
By \cite[(3.10)]{PoZ} with $l\!=\!0$ and $\lr\a\!=\!1$,
\BE{cZ2form_e}
\cZ\big(\hb_f,\hb_{\fp},\al_v',\al_{i_{\fp}},q\big)
=-\frac{1}{\hb_f\!+\!\hb_{\fp}} 
\!\!\!\!\sum_{\begin{subarray}{c} p_-,p_+,r\in\Z^{\ge0}\\
p_-+p_++r=n-1\end{subarray}} \hspace{-.27in}
\wh\si_r\cZ_{p_-}\big(\hb_f,\al_v',q\big)\cZ_{p_+}\big(\hb_{\fp},\al_{i_{\fp}},q\big).\EE
By \cite[(3.11)]{PoZ}, $\hb_f^{-2}\cZ_{p_-}(\hb_f,\al_v',q)\tnd\hb_f$
has no pole at $\hb_f\!=\!\i$.
Combining these statements with the Residue Theorem on~$\P^1$, we find~that 
\begin{equation*}\begin{split}
&\sum_{\Ga_{\fp}}q^{|\Ga_{\fp}|}
\prod_{f\in\fp}\!
\bigg(\frac{\al_{\mu(f)}\!-\!\al_{\mu_c(f)}}{\d(e_f)}\bigg)^{\!\!-b_f-1}
\!\!\!\int_{\cZ_{\Ga_{\fp}}}\!\!\!
\frac{\prod\limits_{f\in\fp}\!\!\ev_f^*\phi_{\mu(f)}}
{\E(\N\cZ_{\Ga_{\fp}})}
\prod_{s\in S_{\fp}^*}\!\frac{\ev_s^*\phi_{i_s}}{\hb_s\!-\!\psi_s}\\
&\hspace{1.8in}
=-\Res{\hb_f=0}\Big\{\!\big(\!-\!\hb_f\big)^{-b_f-1}\!
\cZ^*\big(\hb_f,\hb_{\fp},\al_v',\al_{i_{\fp}},q\big)\!\Big\}
\end{split}\end{equation*}  
if  the sum is taken over all possibilities for the strand $\Ga_{\fp}$,  
leaving the vertex $v\!\equiv\!\ov\eta(\fp)$, 
with \hbox{$\al_v'\!\equiv\!\mu(v)$} fixed and $|\Ga_{\fp}|\!\in\!\Z^+$, 
and carrying the marked point indexed by~$\fp$.\\

\noindent
By the definition of the degree~0 term in the $(g,m)\!=\!(0,2)$ case of~\e_ref{cZdfn_e}, 
\BE{cZ2deg0_e}\begin{split}
\hb_{\fp}^{-b_{\fp}-1}\!\!\prod_{k\neq i_{\fp}}\!(\al_v'\!-\!\al_k)
=(-1)^{b_{\fp}}\Res{\hb_f=0}\bigg\{\hb_f^{-b_{\fp}-1}
\!\LRbr{\cZ\big(\hb_f,\hb_{\fp},\al_v',\al_{i_{\fp}},q\big)}_{q;0}\bigg\}\,.
\end{split}\EE
This is the summand for the special case~$\Ga_{\fp}$ corresponding to 
the case $\eta(\fp)\!=\!v$.
\end{proof}

\begin{lmm}\label{S2sum_lmm}
If $\fp\!\in\!\ov\Edg$, then 
\BE{Zreg_e8d}\begin{split}
&\sum_{\Ga_{\fp}}q^{|\Ga_{\fp}|}
\prod_{f\in\fp}\!
\bigg(\frac{\al_{\mu(f)}\!-\!\al_{\mu_c(f)}}{\d(e_f)}\bigg)^{\!\!-b_f-1}
\!\!\!\int_{\cZ_{\Ga_{\fp}}}\!\!\!
\frac{\prod\limits_{f\in\fp}\!\!\ev_f^*\phi_{\mu(f)}}
{\E(\N\cZ_{\Ga_{\fp}})}
\prod_{s\in S_{\fp}^*}\!\frac{\ev_s^*\phi_{i_s}}{\hb_s\!-\!\psi_s}\\
&\hspace{.5in}
=\Res{\hb_+=0}\Big\{\Res{\hb_-=0}\Big\{
\!\big(\!-\!\hb_-\big)^{-b_{\fp}^--1}\!\big(\!-\!\hb_+\big)^{-b_{\fp}^+-1}
\cZ\big(\hb_-,\hb_+,\al_{\fp}^-,\al_{\fp}^+,q\big)\!\!\Big\}\!\!\Big\},
\end{split}\EE
where $b_{\fp}^{\pm}\!\equiv\!b_{f^{\pm}}$ if $\fp\!=\!\{f^+,f^-\}$ and
 the sum is taken over all possibilities for the strand $\Ga_{\fp}$ between the vertices
$v_-\!\equiv\!\ov\eta(f_{\fp}^-)$ and $v_+\!\equiv\!\ov\eta(f_{\fp}^+)$,
with $\al_{\fp}^-\!\equiv\!\mu(v_-)$ and $\al_{\fp}^+\!\equiv\!\mu(v_+)$ fixed. 
\end{lmm}

\begin{proof} By the reasoning as in the proof of Lemma~\ref{S3sum_lmm} applied twice, 
\begin{equation*}\begin{split}
&\sum_{\Ga_{\fp}}q^{|\Ga_{\fp}|}
\prod_{f\in\fp}\!
\bigg(\frac{\al_{\mu(f)}\!-\!\al_{\mu_c(f)}}{\d(e_f)}\bigg)^{\!\!-b_f-1}
\!\!\!\int_{\cZ_{\Ga_{\fp}}}\!\!\!
\frac{\prod\limits_{f\in\fp}\!\!\ev_f^*\phi_{\mu(f)}}
{\E(\N\cZ_{\Ga_{\fp}})}
\prod_{s\in S_{\fp}^*}\!\frac{\ev_s^*\phi_{i_s}}{\hb_s\!-\!\psi_s}\\
&\hspace{.5in}
=\Res{\hb_+=0}\Big\{\Res{\hb_-=0}\Big\{
\!\big(\!-\!\hb_-\big)^{-b_{\fp}^--1}\!\big(\!-\!\hb_+\big)^{-b_{\fp}^+-1}
\cZ^*\big(\hb_-,\hb_+,\al_{\fp}^-,\al_{\fp}^+,q\big)\!\!\Big\}\!\!\Big\}.
\end{split}\end{equation*}
Since 
$$\Res{\hb_+=0}\Big\{\Res{\hb_-=0}\Big\{\hb_-^{-b_{\fp}^--1}\hb_+^{-b_{\fp}^+-1}
\!\LRbr{\cZ\big(\hb_-,\hb_+,\al_{\fp}^-,\al_{\fp}^+,q\big)}_{q;0}\!\!\Big\}\!\!\Big\}=0
\quad\forall\,b_{\fp}^-,b_{\fp}^+\!\in\!\Z^{\ge0},$$
we can replace $\cZ^*$ in the previous expression by $\cZ$.
\end{proof} 

\noindent
We now combine \e_ref{decomp2_e2} with \e_ref{ovFldagdfn_e} and 
Lemmas~\ref{S1sum_lmm}-\ref{S2sum_lmm} and sum over all possibilities for~$\Fl'(\Ga)$.
Taking into account the automorphism groups, we~obtain
\BE{decomp2_e2b}\begin{split}
&\big|\Aut(|\ov\Ga|)\big|\bigg(\prod_{v\in\ov\Ver}
\prod_{k\neq\mu(v)}\!\!\!\!(\al_v'\!-\!\al_k)\!\!\bigg)\!\!
\sum_{\Ga}q^{|\Ga|}\!\!\!\int_{\cZ_{\Ga}}\!
\frac{1}{\E(\N\cZ_{\Ga}^{\vir})}\!\!
\prod_{s=1}^{s=N}\!\!\frac{\ev_s^*\phi_{i_s}}{\hb_s\!-\!\psi_s}\\
&\hspace{.3in}
=\sum_{(\b,\bfI)\in\cA(\ov\Ga)}\!\!
\Bigg\{\prod_{v\in\ov\Ver}\!\!\!
\wt\cZ_{\b_v,I_v}^{(g_v)}(\al_v',q)
\prod_{s\in[N]}\Res{\hb=0}\Big\{\frac{(-\hb\big)^{-b_s}}{b_s!\,\hb}
\cZ\big(\hb,\hb_s,\al_{\ov\eta(s)}',\al_{i_s},q\big)\!\Big\}\\
&\hspace{.8in}\times\!\!\!\!\prod_{e\in\ov\Edg}\!
\Res{\hb_+=0}\!\bigg\{\Res{\hb_-=0}\!\bigg\{
\frac{(-\hb_-)^{-b_e^-}(-\hb_+)^{-b_e^+}}
{b_e^-!\,b_e^+!\,\hb_-\hb_+}
\cZ\big(\hb_-,\hb_+,\al_e^-,\al_e^+,q\big)\!\!\bigg\}\!\!\bigg\}
\!\!\Bigg\},
\end{split}\EE
with $\wt\cZ_{\b_v,I_v}^{(g_v)}$ as in~\e_ref{wtcZgbIdfn_e}.
The sum on the left-hand side above is taken over all equivalence classes of
connected $[n]$-valued $N$-marked genus~$g$ weighted graphs~$\Ga$
as in~\e_ref{decortgraphdfn_e} with a fixed core~$\ov\Ga$ as in~\e_ref{ga0dfn_e}.\\

\noindent
By the first statement of Lemma~\ref{comb_l0}, 
\begin{equation*}\begin{split}
&(-1)^{\mu_{g_v}(I_v)+m_v-c}
\ze(\al_v',q)^{|\b_v|-(\mu_{g_v}(I_v)+m_v-c)}
\binom{|\b_v|\!-\!|\un\ep_v|}{\mu_{g_v}(I_v)\!+\!m_v\!-\!c\!-\!|\un\ep_v|}
\!\!\!\prod_{f\in\ov\Fl_v(\ov\Ga)}\!\!\!\binom{b_f}{\ep_f}\\
&\hspace{1.5in}=\sum_{\begin{subarray}{c}\b'_v\in(\Z^{\ge0})^{\ov\Fl_v(\ov\Ga)}\\
|\b_v'|+|\un\ep_v|+c=\mu_{g_v}(I_v)+m_v\\
b_f'+\ep_f\le b_f~\forall\,f\in \ov\Fl_v(\ov\Ga)\end{subarray}}
\!\!\!\!\!\!\prod_{f\in \ov\Fl_v(\ov\Ga)}\!\!\!\bigg(\frac{(-1)^{b_f}b_f!}{\ep_f!b_f'!} 
\,\frac{(-\ze(\al_v',q))^{b_f-b_f'-\ep_f}}{(b_f\!-\!b_f'\!-\!\ep_f)!}\bigg)
\end{split}\end{equation*}
for all $c\!\in\!\Z$ and 
$\b,\un\ep\!\in\!(\Z^{\ge0})^{\ov\Fl(\ov\Ga)}$,
where $m_v\!=\!|\ov\Fl_v(\ov\Ga)|$.
Combining~\e_ref{decomp2_e2b} with~\e_ref{cZmB_e}, we thus obtain
\BE{decomp2_e15}\begin{split}   
&\big|\Aut(|\ov\Ga|)\big|\bigg(\prod_{v\in\ov\Ver}
\prod_{k\neq\mu(v)}\!\!\!\!(\al_v'\!-\!\al_k)\!\!\bigg)\!\!
\sum_{\Ga}q^{|\Ga|}\!\!\!\int_{\cZ_{\Ga}}\!
\frac{1}{\E(\N\cZ_{\Ga}^{\vir})}\!\!
\prod_{s=1}^{s=N}\!\!\frac{\ev_s^*\phi_{i_s}}{\hb_s\!-\!\psi_s}\\
&\hspace{.2in}=\sum_{(\b',\un\ep,\c,\bfI)\in\wt\cA^{\star}(\ov\Ga)}\!\!
\left\{\! \prod_{v\in\ov\Ver}\!\!\!\!\Psi_{I_v;\c_v}^{(g_v,\un\ep_v)}(\al_v',q)\times
\prod_{s=1}^{s=N}\!\!\bigg(\!   
\frac{1}{b_s'!}
\Res{\hb=0}\bigg\{\frac{e^{-\frac{\ze(\al_{\ov\eta(s)}',q)}{\hb}}}{\hb^{\,b_s'+\ep_s+1}}
\!\cZ\big(\hb,\hb_s,\al_{\ov\eta(s)}',\al_{i_s},q\big)\!\bigg\}
\!\!\bigg)\!\!\right.\\
&\hspace{.68in}
\times\!\!\!\left.\prod_{e\in\ov\Edg}\!\!\Bigg(
\frac{1}{b_e'^-!b_e'^+!}
\Res{\hb_+=0}\bigg\{\Res{\hb_-=0}\bigg\{
\frac{e^{-\frac{\ze(\al_e^-,q)}{\hb_-}-\frac{\ze(\al_e^+,q)}{\hb_+}}}
{\hb_-^{b_e'^-+\ep_e^-+1}\hb_+^{b_e'^++\ep_e^++1}}
\!\cZ\big(\hb_-,\hb_+,\al_e^-,\al_e^+,q\big)\!\bigg\}\!\!\bigg\}
\!\!\Bigg)\!\!\right\},
\end{split}\EE
where $b_e'^{\pm}\!=\!b_{f^{\pm}}'$ and $\ep_e'^{\pm}\!=\!\ep_{f^{\pm}}'$
if $e\!=\!\{f^+,f^-\}$ and $\wt\cA^{\star}(\ov\Ga)$ is as in~\e_ref{wtcAstaredfn_e}.\\

\noindent
The first statement of
Lemma~\ref{pt2_lmm} and Corollary~\ref{pt2_crl} reduce the last two products above~to
\begin{equation*}\begin{split}
\sum_{\begin{subarray}{c}\p\in\nset^N\\ \b\in(\Z^{\ge0})^N\end{subarray}}
\!\!\Bigg\{\un\hb^{-\b}\!\cZ_{\p}(\un\hb,\al_{i_1\ldots\al_N},q) 
\!\!\!\sum_{\begin{subarray}{c}\p'\in\nset^{\ov\Fl(\ov\Ga)}\\  
\b''\in(\Z^{\ge0})^{\ov\Edg}\end{subarray}}\hspace{-.22in}
(-1)^{|\b|+|\b''|}\!\!
\prod_{s=1}^{s=N}\frac{\cC_{p_sp_s'}(q)\Psi_{p_s';b_s'+\ep_s-b_s}(\al_{\ov\eta(s)}',q)}{b_s'!}&\\
\times\!\prod_{e\in\ov\Edg}
\!\!\!\frac{\cC_{p_e'^-p_e'^+}(q)\Psi_{p_e'^-;b_e'^-+\ep_e^--b_e''}(\al_e^-,q)
\Psi_{p_e'^+;b_e'^++\ep_e^++1+b_e''}(\al_e^+,q)}{b_e'^-!b_e'^+!}&\Bigg\}.
\end{split}\end{equation*}
By the Residue Theorem on~$\P^1$,
\begin{equation*}\begin{split}
&\sum_{j=1}^n
\frac{\Psi_{I_v;\c_v}^{(g_v,\un\ep_v)}(\al_j,q)}{\prod\limits_{k\neq j}\!\!(\al_j\!-\!\al_k)}
\!\!\prod_{s\in\ov\eta^{-1}(v)}\!\!\!\!\!\!\!\!
\frac{\Psi_{p_s';b_s'+\ep_s-b_s}(\al_j,q)}{b_s'!}
\!\!\!\prod_{f\in\Fl_v^-(\ov\Ga)}\!\!\!\!\!\!\!\!
\frac{\Psi_{p_f';b_f'+\ep_f-b_{e_f}''}(\al_j,q)}{b_f'!}
\!\!\!\prod_{f\in\Fl_v^+(\ov\Ga)}\!\!\!\!\!\!\!\!
\frac{\Psi_{p_f';b_f'+\ep_f+1+b_{e_f}''}(\al_j,q)}{b_f'!}\\
&\hspace{3.8in}=-\Res{\x=0,\i}\Bigg\{
\frac{\F^{(g_v,\un\ep_v)}_{I_v;\c_v;\p_v',\b_v',\b_v''}(\x,q)}
{\prod\limits_{k=1}^{k=n}\!\!(\x\!-\!\al_k)}\Bigg\},
\end{split}\end{equation*}
where 
\begin{equation*}\begin{split}
\F^{(g_v,\un\ep_v)}_{I_v;\c_v;\p_v',\b_v',\b_v''}(\x,q)
&=\Psi_{I_v;\c_v}^{(g_v,\un\ep_v)}(\x,q)
\!\!\!\!\!
\prod_{s\in\ov\eta^{-1}(v)}\!\!\!\!\!\!\!\!\frac{\Psi_{p_s';b_s'+\ep_s-b_s}(\x,q)}{b_s'!}
\!\!\!\prod_{f\in\Fl_v^-(\ov\Ga)}\!\!\!\!\!\!\!\!
\frac{\Psi_{p_f';b_f'+\ep_f-b_{e_f}''}(\x,q)}{b_f'!}\\
&\hspace{2.25in}
\times
\!\!\!\prod_{f\in\Fl_v^+(\ov\Ga)}\!\!\!\!\!\!\!\!
\frac{\Psi_{p_f';b_f'+\ep_f+1+b_{e_f}''}(\x,q)}{b_f'!}.
\end{split}\end{equation*}

\vspace{.2in}

\noindent
We now divide both sides of~\e_ref{decomp2_e15} by the first two factors on
the left-hand side and sum over all possibilities for $\mu(v)\!\in\![n]$
and $\ov\Ga\!\in\!\cA_{g,N}$;
we replace~$\ov\Ga$ with~$\Ga$ below.
Using the equations after~\e_ref{decomp2_e15}, we obtain an explicit formula
for the coefficients~$\cC_{g;\p,\b}^{(d)}$ in Theorem~\ref{equiv_thm}:
\BE{cCformula_e}\begin{split}
&\cC_{g;\p,\b}^{(d)} =
\sum_{\Ga\in\cA_{g,N}}\!\frac{1}{|\Aut(\Ga)|}
\sum_{d'=0}^{d'=d}
\sum_{\begin{subarray}{c}\bfd\in(\Z^{\ge0})^{\Ver}\\  
|\bfd|=d-d'\end{subarray}}\!\!
\sum_{\begin{subarray}{c}\p'\in\nset^{\ov\Fl(\Ga)}\\  
\b''\in(\Z^{\ge0})^{\Edg}\end{subarray}}\hspace{-.22in}
(-1)^{|\b|+|\b''|}\hspace{-.3in}
\sum_{(\b',\un\ep,\c,\bfI)\in\wt\cA^{\star}(\Ga)}\\
&\hspace{.5in}
\Bigg\llbracket \prod_{s=1}^{s=N}\!\!\cC_{p_sp_s'}(q) \prod_{e\in\Edg}\!\!\!\!\cC_{p_e'^-p_e'^+}(q)
\Bigg\rrbracket_{q;d'}
\prod_{v\in\Ver}\!\!(-1)\!\!\!\Res{\x=0,\i}\Bigg\llbracket
\frac{\F^{(g_v,\un\ep_v)}_{I_v;\c_v;\p_v',\b_v',\b_v''}(\x,q)}
{\prod\limits_{k=1}^{k=n}\!\!(\x\!-\!\al_k)} \Bigg\rrbracket_{q;d_v}\,.
\end{split}\EE
This establishes~\e_ref{equivthm_e}.\\

\noindent
It remains to show that~\e_ref{equivthm_e2} with the summation over $t\!\in\!\Z$
instead of~$\Z^{\ge0}$ holds for some \hbox{$\nc_{g;\p,\b}^{(d,t)}\!\in\!\Q$} such that 
$\nc_{g;\p,\b}^{(d,0)}\!=\!\nc_{g;\p,\b}^{(d)}$ with $\nc_{g;\p,\b}^{(d)}$
given by~\e_ref{ncCdfn2_e}.
Let $(\b',\un\ep,\c,\bfI)\!\in\!\wt\cA^{\star}(\ov\Ga)$ and $\p,\p',\b,\b''$ be as above.
For $v\!\in\!\ov\Ver$, define
$$|\p_v'|=\sum_{f\in\ov\Fl_v(\ov\Ga)}\!\!\!\!\!p_f',\quad
|\b_v|=\sum_{s\in\ov\eta^{-1}(v)}\!\!\!\!\!b_s, \quad
|\b_v''^-|=\sum_{f\in\Fl_v^-(\ov\Ga)}\!\!\!\!\!b_{e_f}'',\quad
|\b_v''^+|=\sum_{f\in\Fl_v^+(\ov\Ga)}\!\!\!\!\!b_{e_f}''\,.$$
By~\e_ref{p2cC_e},
$$\prod_{s=1}^{s=N}\!\!\cC_{p_sp_s'}(q) \prod_{e\in\ov\Edg}\!\!\!\!\cC_{p_e'^-p_e'^+}(q)
\sim\begin{cases}1,&\hbox{if}~
p_s\!+\!p_s',p_e'^-\!+\!p_e'^+\!=\!n\!-\!1~\forall\,s\!\in\![N],\,e\!\in\!\ov\Edg;\\
0&\hbox{otherwise}.\end{cases}$$
By~\e_ref{PsimcPhi_e}, \e_ref{cZpexp_e0}, and~\e_ref{wtcAstaredfn_e}, 
\begin{equation*}\begin{split}
\Big\llbracket \F^{(g_v,\un\ep_v)}_{I_v;\c_v;\p_v',\b_v',\b_v''}(\x,q)
\Big\rrbracket_{q;d_v}
&\sim\! \x^{(n-1)-nd_v+(n-4)(g_v-1)-|\ov\Fl_v(\ov\Ga)|-|\Fl_v^+(\ov\Ga)|
+|\p_v'|+|\b_v|+|\b_v''^-|-|\b_v''^+|}\\
&\hspace{2.5in}\times\Bigg\llbracket F^{(g_v,\un\ep_v)}_{I_v;\c_v;\p_v',\b_v',\b_v''}(q)
\Bigg\rrbracket_{q;d_v},
\end{split}\end{equation*}
where 
$$F^{(g_v,\un\ep_v)}_{I_v;\c_v;\p_v',\b_v',\b_v''}(q)
=\Phi_{I_v;\c_v}^{(g_v,\un\ep_v)}(q)
\!\!\!\prod_{s\in\ov\eta^{-1}(v)}\!\!\!\!\!\!\!
\frac{\Phi_{p_s';b_s'+\ep_s-b_s}(q)}{b_s'!\Phi_0(q)}
\!\!\!\prod_{f\in\Fl_v^-(\ov\Ga)}\!\!\!\!\!\!\!
\frac{\Phi_{p_f';b_f'+\ep_f-b_{e_f}''}(q)}{b_f'!\Phi_0(q)}
\!\!\!\prod_{f\in\Fl_v^+(\ov\Ga)}\!\!\!\!\!\!\!
\frac{\Phi_{p_f';b_f'+\ep_f+1+b_{e_f}''}(q)}{b_f'!\Phi_0(q)}\,.$$
Along with the last two statements in Lemma~\ref{ressum_lmm}, this implies~that
\BE{cFres_e}-\Res{\x=0,\i}\Bigg\llbracket
\frac{\F^{(g_v,\un\ep_v)}_{I_v;\c_v;\p_v',\b_v',\b_v''}(\x,q)}
{\prod\limits_{k=1}^{k=n}\!\!(\x\!-\!\al_k)} \Bigg\rrbracket_{q;d_v}
\sim F^{(g_v,\un\ep_v)}_{I_v;\c_v;\p_v',\b_v',\b_v''}(q)\wh\si_n^{t_v}\EE
with $t_v\!\in\!\Z$ defined by 
$$|\p_v'|+|\b_v|+|\b_v''^-|-|\b_v''^+|-\big|\Fl_v^+(\ov\Ga)\big|
=(n\!-\!4)(1\!-\!g_v)+|\ov\Fl_v(\ov\Ga)|+n(d_v\!+\!t_v)\,;$$
if an integer $t_v$ satisfying the above condition does not exist, 
we define $\wh\si_n^{t_v}$  to be~0.\\

\noindent
By~\e_ref{cCformula_e}, the above paragraph, and
the middle statement in Lemma~\ref{ressum_lmm},   
\begin{equation*}\begin{split}
\cC_{g;\p,\b}^{(d)}&\sim
\sum_{\Ga\in\cA_{g,N}}\!\frac{1}{|\Aut(\Ga)|}
\!\!
\sum_{\begin{subarray}{c}\bfd\in(\Z^{\ge0})^{\Ver}\\  
|\bfd|=d\end{subarray}}\!\!
\sum_{\begin{subarray}{c}\p'\in\nset^{\Edg}\\  
\b'\in(\Z^{\ge0})^{\Edg}\end{subarray}}\hspace{-.22in}
(-1)^{|\b|+|\b'|}\wh\si_n^{|\bft|}\hspace{-.3in}
\sum_{(\b'',\un\ep,\c,\bfI)\in\wt\cA^{\star}(\Ga)}\\
&\prod_{v\in\Ver}\! 
\Bigg\llbracket \!\Phi_{I_v;\c_v}^{(g_v,\un\ep_v)}(q)\!\!
\prod_{s\in S_v}\!\frac{\Phi_{\wh{p}_s;b_s''+\ep_s-b_s}\!(q)}{b_s''!\,\Phi_0(q)}
\!\!\!\!\!\prod_{f\in\Fl_v^-(\Ga)}\!\!\!\!\!\!
\frac{\Phi_{p_{e_f}';b_f''+\ep_f-b_{e_f}'}\!(q)}{b_f''!\Phi_0(q)}
\!\!\!\!\!\prod_{f\in\Fl_v^+(\Ga)}\!\!\!\!\!\!
\frac{\Phi_{\wh{p}_{e_f}';b_f''+\ep_f+1+b_{e_f}'}\!(q)}{b_f''!\,\Phi_0(q)}\Bigg\rrbracket_{q;d_v}
\end{split}\end{equation*}
with $\bft\!\equiv\!(t_v)_{v\in\Ver}\!\in\!\Z^{\Ver}$ defined by
\begin{equation*}\begin{split}
&\sum_{s\in S_v}\!\!\big(\wh{p}_s\!+\!b_s\big)
+\!\!\sum_{f\in\Fl_v^-(\Ga)}\!\!\!\!\!\!\big(p_{e_f}'\!+\!b_{e_f}'\big)
+\!\!\sum_{f\in\Fl_v^+(\Ga)}\!\!\!\!\!\!\big(\wh{p}_{e_f}'\!-\!1\!-\!b_{e_f}'\big)\\
&\hspace{1in}
=(n\!-\!4)(1\!-\!g_v)+\big|\ov\Fl_v(\Ga)\big|+n(d_v\!+\!t_v) 
\qquad\forall~v\!\in\!\Ver;
\end{split}\end{equation*}
the corresponding summand above is taken to be~0 if an integer~$t_v$ 
satisfying the above condition does not exist for some $v\!\in\!\Ver$.
This confirms~\e_ref{equivthm_e2} with $\nc_{g;\p,\b}^{(d)}\!\equiv\!\nc_{g;\p,\b}^{(d,0)}$ 
as defined in~\e_ref{ncCdfn2_e}
(and describes $\nc_{g;\p,\b}^{(d,t)}$ with $t\!\in\!\Z^+$ as~well).

\subsection{The recursion approach}
\label{RecFormComp_subs}

\noindent 
We next show that \e_ref{equivthm_e2} holds with the coefficients $\nc_{g;\p,\b}^{(d,t)}$
as defined recursively at the end of Section~\ref{Mainform_subs}.
Let $\Ga$ be a connected $[n]$-valued $N$-marked genus~$g$ weighted graph
and $\ov\Ga$ be its core as before.
This time we break $\Ga$ only at the vertex
$$v\equiv \ov\eta(N) \in \ov\Ver\subset\Ver$$ 
into strands~$\Ga_{\fp}$ as in~\e_ref{stranddfn_e}.
Each edge $e\!\equiv\!\{f,f'\}$ of the original graph~$\Ga$ with $\eta(f)\!=\!v$
keeps a copy~$v_f$ of the vertex~$v$ satisfying~\e_ref{stranddfn_e2} 
and carrying an additional marked point labeled by~$\wh{f}$.
The set of strands (i.e.~of the connected components of the graph) obtained 
from~$\Ga$ in this way is indexed by a quotient~$\Fl_v^{\dag}(\Ga)$
of the set~$\Fl_v(\Ga)$ of flags of~$\Ga$ at~$v$ so that 
$f_1,f_2\!\in\!\Fl_v(\Ga)$ determine the same element $f_1^{\dag}\!=\!f_2^{\dag}$
of $\Fl_v^{\dag}(\Ga)$
if and only if the marked points labeled by~$\wh{f}_1$ and~$\wh{f}_2$ lie on the same strand.
Since the strands~$\Ga_{\fp}$ are obtained from the genus~$g$ graph~$\Ga$
by breaking all edges at~$v$,
$$g_v\!+\!\sum_{\fp\in\Fl_v^{\dag}(\Ga)}\!\!\!\!\!\fa(\Ga_{\fp})+\big|\Fl_v(\Ga)\big|
=g\!+\!|\Fl_v^{\dag}(\Ga)\big|\,.$$

\vspace{.2in}

\noindent
The set of strands of type~(S1) is now the subset
$$\Fl_v'(\Ga)\equiv \Fl'(\Ga)\!\cap\!\Fl_v(\Ga)$$
of the set of strands of type~(S1) in Section~\ref{ClFormComp_subs} that leave from the vertex~$v$.
With the notation as in~\e_ref{FlprGadfn_e},
the set~$\Fl_v^*(\Ga)$ of the remaining strands~$\fp$ is a quotient of the~set 
$$\Fl_v(\Ga)\!-\!\Fl_v'(\Ga)\equiv 
\big\{f_s\!:s\!\in\![N],~\ov\eta(s)\!=\!v\!\neq\!\eta(s)\big\}\sqcup \Fl_v(\ov\Ga).$$
With $S_v\!\equiv\![N]\!\cap\!\eta^{-1}(v)$  as before, let
$$\ov\Fl_v^*(\Ga) \equiv S_v\!\sqcup\!\Fl_v^*(\Ga)\,.$$
The~set
\BE{ovFldagdfn_e2}
\ov\Fl_v^{\dag}(\Ga)\equiv S_v\!\sqcup\!\Fl_v^{\dag}(\Ga)
=S_v\!\sqcup\!\Fl_v^*(\Ga)\!\sqcup\!\Fl'(\Ga) 
=\ov\Fl_v^*(\Ga)\!\!\sqcup\!\Fl'(\Ga)\EE
is similarly a quotient of the set $\ov\Fl_v(\Ga)$.
For $\fp\!\in\!\ov\Fl_v^{\dag}(\Ga)$, we write $f\!\in\!\fp$ if
$f\!\in\!\ov\Fl_v(\Ga)$ and $f^{\dag}\!=\!\fp$.\\

\noindent
For $s\!\in\!S_v$, define 
$$\fa(\Ga_s)\!=\!0,\qquad S_s^*=\{s\}, \qquad |s|=|S_s'|\!=\!1.$$
For $\fp\!\in\!\Fl_v^{\dag}(\Ga)$, let 
$$S_{\fp}^*\equiv S_{\fp}\!\cap\![N], \qquad
S_{\fp}'\equiv S_{\fp}\!-\![N]=\big\{\wh{f}\!:f^{\dag}\!=\!\fp\big\},
\qquad\hbox{and}\quad |\fp|\equiv\big|S_{\fp}'\big|$$
be the set of the original marked points carried by~$\fp$,
the set of the additional marked points, and the cardinality of the latter,
respectively.
Thus,
\begin{gather*}
|\fp|\in\Z^+~~\forall\,\fp\!\in\!\ov\Fl^{\dag}(\Ga), \quad
\Fl_v'(\Ga)=\big\{\fp\!\in\!\ov\Fl^{\dag}(\Ga)\!:
\fa(\Ga_{\fp})\!=\!0,\,S_{\fp}^*\!=\!\eset,\,|\fp|\!=\!1\big\},\\
[N]\!-\!S_v=\bigsqcup_{\fp\in\Fl^*(\Ga)}\!\!\!\!\!\!S_{\fp}^*, \quad
\Fl_v(\Ga)\!-\!\Fl_v'(\Ga)=\bigsqcup_{\fp\in\Fl^*(\Ga)}\!\!\!\!\!\!S_{\fp}'\,,
\quad
g_v\!+\!\sum_{\fp\in\ov\Fl_v^*(\Ga)}\!\!\!\!\!\!
\big(\fa(\Ga_{\fp})\!+\!|\fp|\big)=g\!+\!|\ov\Fl_v^*(\Ga)\big|.
\end{gather*}
By the choice of~$v$, either $N\!\in\!S_v$ or 
$S_{\fp}\!=\!\{N\}$ for some $\fp\!\in\!\Fl_v^*(\Ga)$ with $\fa(\Ga_{\fp})\!=\!0$ 
and $|\fp|\!=\!1$.
Thus,  
\BE{Comp2typ_e}\big(g_v,(\fa(\Ga_{\fp}),S_{\fp}^*,|\fp|)_{\fp\in\ov\Fl_v^*(\Ga)} \big)
\in \cP_{g,N}^{(|\ov\Fl_v^*(\Ga)|)}.\EE

\vspace{.2in}

\noindent
The analogues of the decompositions~\e_ref{Zreg_e5a} and~\e_ref{Zreg_e5b} in this case are 
\begin{gather*}
\cZ_{\Ga}=\ov\cM_{g_v,\ov\Fl_v(\Ga)}\times
\prod_{\fp\in\Fl_v^{\dag}(\Ga)}\!\!\!\!\!\!\!\cZ_{\Ga_{\fp}},\\
\frac{\E(T_{P_{\mu(v)}}\P^{n-1})}{\E(\N\cZ_{\Ga}^{\vir})}=
\E(\bE_{g_v}^*\!\otimes\!T_{P_{\mu(v)}}\P^{n-1})\!\!\!\!\!\!\!
\prod_{\fp\in\Fl_v^{\dag}(\Ga)}\!\frac{1}{\E(\N\cZ_{\Ga_{\fp}}^{\vir})}
\prod_{f\in\Fl_v(\Ga)}\!\!\!\!
\frac{\E(T_{P_{\mu(v)}}\P^{n-1})}{\hb_f'\!-\!\psi_{\wh{f}}}\,.
\end{gather*} 
For each $I\!\in\!(\Z^{\ge0})^{g_v}$, \e_ref{cMint_e} still applies.
The analogue of~\e_ref{decomp2_e2} is~now
\BE{Zreg_e6}\begin{split}
&\prod_{k\neq\mu(v)}\!\!\!\!(\al_{\mu(v)}\!-\!\al_k)
\int_{\cZ_{\Ga}}\!
\frac{1}{\E(\N\cZ_{\Ga}^{\vir})}
\prod_{s=1}^{s=N}\!\!\bigg(\frac{\ev_s^*\phi_{i_s}}{\hb_s\!-\!\psi_s}\bigg)\\
&\hspace{.1in}
=\sum_{I\in(\Z^{\ge0})^{g_v}}\sum_{\b\in\cA_{v;I}^{\star}(\Ga)}\!\!\Bigg\{\!\!
\big(C_{g_v,n;I}\al_{\mu(v)}^{(n-1)g_v-\|I\|}\!+\!h_{g_v;n;I}(\al_{\mu(v)})\!\big)
\bllrr{\la_{g_v;I};\wt\tau_{\b}}\\
&\hspace{1in} \left.\times
\!\!\prod_{\fp\in\ov\Fl_v^{\dag}(\Ga)}\!\!\!
\Bigg(\!
\prod_{f\in\fp}\!\frac{1}{b_f!}\!
\bigg(\frac{\al_{\mu(v)}\!-\!\al_{\mu_c(f)}}{\d(e_f)}\bigg)^{\!\!-b_f-1}
\!\!\!\int_{\cZ_{\Ga_{\fp}}}\!\!\!
\frac{\prod\limits_{f\in\fp}\!\!\ev_f^*\phi_{\mu(v)}}
{\E(\N\cZ_{\Ga_{\fp}}^{\vir})}
\prod_{s\in S_{\fp}^*}\!\!\frac{\ev_s^*\phi_{i_s}}{\hb_s\!-\!\psi_s}
\!\!\Bigg)\!\!\right\},
\end{split}\EE
where
$$\prod_{f\in\fp}\!\frac{1}{b_f!}\!
\bigg(\frac{\al_{\mu(v)}\!-\!\al_{\mu_c(f)}}{\d(e_f)}\bigg)^{\!\!-b_f-1}
\!\!\!\!\int_{\cZ_{\Ga_{\fp}}}\!\!\!
\frac{\prod\limits_{f\in\fp}\!\!\ev_f^*\phi_{\mu(v)}}
{\E(\N\cZ_{\Ga_{\fp}}^{\vir})} \!\prod_{s\in S_{\fp}^*}\!\!
\frac{\ev_s^*\phi_{i_s}}{\hb_s\!-\!\psi_s}
\equiv\frac{1}{b_{\fp}!}
\bigg(\!\!\hb_{\fp}^{-b_{\fp}-1}\!\!
\prod_{k\neq i_{\fp}}\!\!\big(\al_{\mu(v)}\!-\!\al_k\big)\!\!\bigg)$$
if $\fp\!\in\!S_v\!\subset\!\ov\Fl_v^{\dag}(\Ga)\!\cap\!\ov\Fl_v(\Ga)$.
The equality in~\e_ref{Zreg_e6} holds after taking into account the automorphism groups;
this is done below after summing over all possibilities for the strands~$\Ga_{\fp}$.\\

\noindent
The relevant sum  over all possibilities for the strands~$\Ga_{\fp}$ with 
$\fp\!\in\!\Fl_v'(\Ga)$ is described by~\e_ref{Zreg_e8} with $\mu(f)\!=\!\mu(v)$.
For $\fp\!\in\!\ov\Fl_v^*(\Ga)$ such that $\fa(\Ga_{\fp})\!=\!0$, $|\fp|\!=\!1$,
and $S_{\fp}\!=\!\{s_{\fp}\}$ for some \hbox{$s_{\fp}\!\in\!\ov\eta^{-1}(v)$},
the corresponding sum is described by~\e_ref{Zreg_e8b} with $\mu(f)\!=\!\mu(v)$,
$\hb_{\fp}\!=\!\hb_{s_{\fp}}$, and $i_{\fp}\!=\!i_{s_{\fp}}$.
Lemma~\ref{stSsum_lmm} below extends~\e_ref{Zreg_e8b} to the remaining cases.
For each $\fp\!\in\!\ov\Fl_v^*(\Ga)$, we order the elements of 
\hbox{$S_{\fp}'\!\subset\!\Fl_v(\Ga)$}
as $f_1,\ldots,f_{|\fp|}$ and define
$$\Res{(\hb_f=0)_{\!f\in\fp}}\!\bigg\{\ldots\bigg\}
=\Res{\hb_{f_{|\fp|}}=0}\!\bigg\{\ldots
\Res{\hb_{f_1}=0}\bigg\{\ldots\bigg\}\ldots\bigg\}.$$

\begin{lmm}\label{stSsum_lmm}
If $\fp\!\in\!\ov\Fl_v^*(\Ga)$ with $2(\fa(\Ga_{\fp})\!+\!|\fp|)\!+\!|S_{\fp}^*|\!>\!3$, 
then 
\BE{Zreg_e8e}\begin{split}
&\sum_{\Ga_{\fp}}q^{|\Ga_{\fp}|}
\prod_{f\in\fp}\!
\bigg(\frac{\al_{\mu(v)}\!-\!\al_{\mu_c(f)}}{\d(e_f)}\bigg)^{\!\!-b_f-1}
\!\!\!\int_{\cZ_{\Ga_{\fp}}}\!\!\!
\frac{\prod\limits_{f\in\fp}\!\!\ev_f^*\phi_{\mu(f)}}
{\E(\N\cZ_{\Ga_{\fp}}^{\vir})}
\prod_{s\in S_{\fp}^*}\!\frac{\ev_s^*\phi_{i_s}}{\hb_s\!-\!\psi_s}\\
&\hspace{.5in}
=(-1)^{|\fp|}\!\!\!
\Res{(\hb_f=0)_{\!f\in\fp}}\!\Bigg\{\prod_{f\in\fp}\!\!\big(\!-\!\hb_f\big)^{-b_f-1}
\cZ^{(\fa(\Ga_{\fp}))}\!\big((\hb_f)_{\!f\in S_{\fp}},
(\al_v')_{\!f\in\fp},(\al_{i_f})_{\!f\in S_{\fp}^*},q\big)\!\!\Bigg\},
\end{split}\EE
where the sum is taken over all possibilities for the strand $\Ga_{\fp}$ 
with $v\!=\!\ov\eta(f)$ for all $f\!\in\!S_{\fp}'$ and $\al_v'\!\equiv\!\al_{\mu(v)}$ fixed.
\end{lmm}

\begin{proof}
If $\fa(\Ga_{\fp})\!=\!0$, $|\fp|\!=\!2$, and $S_{\fp}^*\!=\!\eset$,
\e_ref{Zreg_e8e} is the case of~\e_ref{Zreg_e8d} with $\mu(f)\!=\!\mu(v)$
and $\al_{\fp}^{\pm}\!=\!\al_v'$.
Suppose \hbox{$2\fa(\Ga_{\fp})\!+\!|\fp|\!+\!|S_{\fp}^*|\!\ge\!3$}.
By the proofs of (3.24) and Lemma~1.2 in~\cite{bcov0} 
applied with arbitrary genus and $\E(\V_0'')\!=\!1$,
the left-hand side of~\e_ref{Zreg_e8e} summed over~$\Ga_{\fp}$
with $\d(e_f)\!=\!d_f$ for $d_f\!\in\!\Z^+$
and $\mu_c(f)\!=\!j_f$ for $j_f\!\in\![n]\!-\!\mu(v)$ 
fixed for each $f\!\in\!\fp$ is the $|\fp|$-fold residue~of 
$$\prod_{f\in\fp}\!\!\big(\!-\!\hb_f\big)^{-b_f-1}
\cZ^{(\fa(\Ga_{\fp}))}\!\big((\hb_f)_{f\in S_{\fp}},
(\al_v')_{\!f\in\fp},(\al_{i_f})_{\!f\in S_{\fp}^*},q\big)$$
at $\hb_f\!=\!(\al_{j_f}\!-\!\al_v')/d_f$ for each $f\!\in\!\fp$
(i.e.~first take the residue of the above power series in rational functions
 at $\hb_{f_1}\!=\!(\al_{j_{f_1}}\!-\!\al_v')/d_{f_1}$,
then the residue of the resulting power series at 
$\hb_{f_1}\!=\!(\al_{j_{f_1}}\!-\!\al_v')/d_{f_1}$, and so~on).
Furthermore, for each $f\!\in\!\fp$ the associated meromorphic 1-form on~$\P^1$ 
in~$\hb_f$ has no poles outside of
$\hb_f\!=\!(\al_j\!-\!\al_v')/d$ with $d\!\in\!\Z^+$ and $j\!\in\![n]\!-\!\mu(v)$
and $\hb_f\!=\!0$.
The claim now follows from the Residue Theorem on~$\P^1$,
as in the proof of Lemma~\ref{S3sum_lmm}, but applied $|\fp|$~times.
\end{proof}

\noindent
We now combine \e_ref{Zreg_e6} with 
Lemmas~\ref{S1sum_lmm}, \ref{S3sum_lmm}, and~\ref{stSsum_lmm}
and sum over all possibilities for~$\Fl_v'(\Ga)$.
Taking into account the automorphism groups, we~obtain an analogue~\e_ref{decomp2_e2b}:
\BE{decomp2_e2d}\begin{split}
&\prod_{k\neq\mu(v)}\!\!\!\!(\al_v'\!-\!\al_k)\!\!
\sum_{\Ga}q^{|\Ga|}\!\!\!\int_{\cZ_{\Ga}}\!
\frac{1}{\E(\N\cZ_{\Ga}^{\vir})}\!\!
\prod_{s=1}^{s=N}\!\!\frac{\ev_s^*\phi_{i_s}}{\hb_s\!-\!\psi_s}
=\sum_{I\in(\Z^{\ge0})^{g_v}}\sum_{\b\in\cA_{v;I}^{\star}(\Ga)}\!\!
\Bigg\{\wt\cZ_{\b,I}^{(g_v)}(\al_v',q)\\
&\hspace{.8in}\times\!\!\!\!
\prod_{\fp\in\ov\Fl_v^*(\Ga)}\!\!\!\!\!
\frac{(-1)^{|\fp|}}{|\fp|!}\!\!\!
\Res{(\hb_f=0)_{\!f\in\fp}}\!\Bigg\{\prod_{f\in\fp}\!
\frac{(-\hb_f)^{-b_f-1}}{b_f!}
\cZ^{(\fa(\Ga_{\fp}))}\!\big((\hb_f)_{\!f\in S_{\fp}},
(\al_v')_{\!f\in\fp},(\al_{i_f})_{\!f\in S_{\fp}^*},q\big)\!\!\Bigg\}.
\end{split}\EE
The sum on the left-hand side above is taken over all equivalence classes of
connected $[n]$-valued $N$-marked genus~$g$ weighted graphs~$\Ga$
as in~\e_ref{decortgraphdfn_e} determining a fixed element~\e_ref{Comp2typ_e}
of $\cP_{g,N}^{(m)}$ with $m\!=\!|\ov\Fl_v^*(\Ga)|$.
The analogue of~\e_ref{decomp2_e15} is~now
\BE{decomp2_e15b}\begin{split}   
&\prod_{k\neq\mu(v)}\!\!\!\!(\al_v'\!-\!\al_k)\!\!
\sum_{\Ga}q^{|\Ga|}\!\!\!\int_{\cZ_{\Ga}}\!
\frac{1}{\E(\N\cZ_{\Ga}^{\vir})}\!\!
\prod_{s=1}^{s=N}\!\!\frac{\ev_s^*\phi_{i_s}}{\hb_s\!-\!\psi_s}
=\sum_{(\b',\un\ep,\c,I)\in\wt\cA^{\star}_{g_v,(|\fp|)_{\fp\in\ov\Fl_v^*(\Ga)}}}\!\!
\Bigg\{\! \Psi_{I;\c}^{(g_v,\un\ep)}(\al_v',q)\\
&\hspace{.5in}
\times\!\prod_{\fp\in\ov\Fl_v^*(\Ga)}\!\!\frac{1}{|\fp|!}
\Res{(\hb_f=0)_{\!f\in\fp}}\!\Bigg\{\prod_{f\in\fp}
\frac{e^{-\frac{\ze(\al_v',q)}{\hb_f}}}
{b_f'!\hb_f^{b_f'+\ep_f+1}}\cZ^{(\fa(\Ga_{\fp}))}\!\big((\hb_f)_{\!f\in S_{\fp}},
(\al_v')_{\!f\in\fp},(\al_{i_f})_{\!f\in S_{\fp}^*},q\big)\!\!\Bigg\},
\end{split}\EE
with $\wt\cA^{\star}_{g_v,(|\fp|)_{\fp\in\ov\Fl_v^*(\Ga)}}$ as in~\e_ref{Phimcdfn_e0}.\\

\noindent
Suppose $\fp\!\in\!\ov\Fl_v^*(\Ga)$ and $2\fa(\Ga_{\fp})\!+\!|S_{\fp}|\!\ge\!3$.
Since $3\fa(\Ga_{\fp})\!+\!|S_{\fp}|\!<\!3g\!+\!N$,
by induction Theorem~\ref{equiv_thm} implies that 
\begin{equation*}\begin{split}
&\cZ^{(\fa(\Ga_{\fp}))}\!\big((\hb_f)_{f\in S_{\fp}},(\x_f)_{f\in S_{\fp}},q\big)\\
&\hspace{.5in}= \sum_{\p\in\nset^{S_{\fp}}}\sum_{\b\in(\Z^{\ge0})^{S_{\fp}}}
\sum_{d=0}^{\i}\cC_{\fa(\Ga_{\fp});\p,\b}^{(d)}q^d \!
\prod_{f\in S_{\fp}}\!\!\hb_f^{-b_f-1}\!\cZ_{p_f}\big(\hb_f,\x_f,q\big)
\end{split}\end{equation*}
with $\cC_{\fa(\Ga_{\fp});\p,\b}^{(d)}\!\in\!\Q[\al]$ satisfying~\e_ref{equivthm_e2}.
We set $\cC_{\fa(\Ga_{\fp});\p,\b}^{(d)}\!=\!0$ if
\hbox{$2\fa(\Ga_{\fp})\!+\!|S_{\fp}|\!\ge\!3$} and $b_f'\!<\!0$ 
for some  $f\!\in\!S_{\fp}$.
Along with~\e_ref{cZpexp_e}, the last equation implies that 
\BE{stresprod_e}\begin{split}
&\Res{(\hb_f=0)_{\!f\in\fp}}\!\Bigg\{\prod_{f\in\fp}
\frac{e^{-\frac{\ze(\al_v',q)}{\hb_f}}}
{\hb_f^{b_f'+\ep_f+1}}\cZ^{(\fa(\Ga_{\fp}))}\!\big((\hb_f)_{\!f\in S_{\fp}},
(\al_v')_{\!f\in\fp},(\al_{i_f})_{\!f\in S_{\fp}^*},q\big)\!\!\Bigg\}\\
&\hspace{.5in} =\sum_{\begin{subarray}{c}\p\in\nset^{S_{\fp}^*}\\
\b\in(\Z^{\ge0})^{S_{\fp}^*}\end{subarray}}
\Bigg\{\prod_{f\in S_{\fp}^*}\!\!\!\hb_f^{-b_f-1}\cZ_{p_f}\big(\hb_f,\al_{i_f},q\big)\\
&\hspace{1.4in}
\times\!\!\!\sum_{\begin{subarray}{c}\p'\in\nset^{S_{\fp}'}\\
\b''\in\Z^{S_{\fp}'}\end{subarray}}
\sum_{d=0}^{\i}\cC_{\fa(\Ga_{\fp});\p\p',\b\b''}^{(d)}q^d
\prod_{f\in\fp}\!\Psi_{p_f';b_f'+\ep_f+1+b_f''}\big(\al_v',q\big)\!\!\Bigg\},
\end{split}\EE
if $2\fa(\Ga_{\fp})\!+\!|S_{\fp}|\!\ge\!3$.
By Lemma~\ref{pt2_lmm} and Corollary~\ref{pt2_crl}, \e_ref{stresprod_e} with
\BE{cC2base_e}\cC_{0;(p_+,p_-),(b_+,b_-)}^{(d)}=\begin{cases}(-1)^{b_+}
\lrbr{\cC_{p_-p_+}\!(q)}_{q;d},&\hbox{if}~b_+\!\ge\!0,\,b_-\!+\!b_+\!=\!-1;\\
0,&\hbox{otherwise};\end{cases}\EE
holds if $\fa(\Ga_{\fp})\!=\!0$ and $|S_{\fp}|\!=\!2$.\\

\noindent
We note that 
$$\prod_{\fp\in\ov\Fl_v^*(\Ga)}
\prod_{f\in S_{\fp}^*}\!\!\hb_f^{-b_f-1}\!\cZ_{p_f}\big(\hb_f,\al_{i_f},q\big)
=\un\hb^{-\b}\!\cZ_{\p}\big(\un\hb,\al_{i_1\ldots i_N},q\big)
\quad\forall\,\p\!\in\!\nset^N,\,\b\!\in\!(\Z^{\ge0})^N.$$
We now divide both sides of~\e_ref{decomp2_e15b} by the first factor on
the left-hand side, plug in~\e_ref{stresprod_e}, and sum up 
over all possibilities for $\mu(v)\!\in\![n]$ using 
the Residue Theorem on~$\P^1$, as after~\e_ref{decomp2_e15}.
Summing up the result over all possibilities for~\e_ref{Comp2typ_e}, 
we obtain a recursion for the coefficients~$\cC_{g;\p,\b}^{(d)}$ in Theorem~\ref{equiv_thm}:
\BE{cCformula_e2}\begin{split}
&\cC_{g;\p,\b}^{(d)} =
\sum_{\begin{subarray}{c}m\in\Z^+\\ d'\in\Z^{\ge0}\end{subarray}}\!\frac{(-1)}{m!}
\sum_{\begin{subarray}{c}(g',\bfg,\bfS,\bfN)\in\cP_{g,N}^{(m)}\\
(\b',\un\ep,\c,I)\in\wt\cA^{\star}_{g',\bfN}\end{subarray}}
\sum_{\begin{subarray}{c}\p'\in\!\!\!\prod\limits_{i\in[m]}\!\!\!\!\nset^{N_i}\\  
\b''\in\!\!\!\prod\limits_{i\in[m]}\!\!\!\!(\Z^{\ge0})^{N_i}\end{subarray}} 
\sum_{\begin{subarray}{c}\bfd\in(\Z^{\ge0})^m\\ |\bfd|=d-d'\end{subarray}}  
\Bigg(\prod_{i=1}^{i=m}\frac{\cC_{g_i;\p|_{S_i}\p_i',\b|_{S_i}\b_i''}^{(d_i)}}{N_i!}\\
&\hspace{1.5in}\times\!\!\!\!
\Res{\x=0,\i}\Bigg\llbracket
\frac{\Psi_{I;\c}^{(g',\un\ep)}(\x,q)}{\prod\limits_{k=1}^{k=n}\!\!(\x\!-\!\al_k)} 
\prod_{i=1}^{i=m}\prod_{f\in[N_i]}\!\!\!\!
\frac{\Psi_{p_{i;f}';b_{i;f}'+\ep_{i;f}+1+b_{i;f}''}\!(\x,q)}{b_{i;f}'!}
\Bigg\rrbracket_{q;d'}\Bigg)\,.
\end{split}\EE
It remains to show that~\e_ref{equivthm_e2} with the summation over $t\!\in\!\Z$
instead of~$\Z^{\ge0}$ holds for some \hbox{$\nc_{g;\p,\b}^{(d,t)}\!\in\!\Q$} such that
$\nc_{g;\p,\b}^{(d,t)}$ with $t\!\ge\!0$ satisfy~\e_ref{coeffdfn_e}.\\ 

\noindent
By~\e_ref{cCformula_e2} and~\e_ref{cFres_e}, 
\BE{cCformula_e4}\begin{split}
&\cC_{g;\p,\b}^{(d)} \sim
\sum_{m,d'\in\Z^{\ge0}}\!\frac{1}{m!}
\sum_{\begin{subarray}{c}(g',\bfg,\bfS,\bfN)\in\cP_{g,N}^{(m)}\\
(\b',\un\ep,\c,I)\in\wt\cA^{\star}_{g',\bfN}\end{subarray}}
\sum_{\begin{subarray}{c}\p'\in\!\!\!\prod\limits_{i\in[m]}\!\!\!\!\nset^{N_i}\\  
\b''\in\!\!\!\prod\limits_{i\in[m]}\!\!\!\!(\Z^{\ge0})^{N_i}\end{subarray}} 
\sum_{\begin{subarray}{c}\bfd\in(\Z^{\ge0})^m\\ |\bfd|=d-d'\end{subarray}}  
\Bigg(\wh\si_n^t
\prod_{i=1}^{i=m}\frac{\cC_{g_i;\p|_{S_i}\p_i',\b|_{S_i}\b_i''}^{(d_i)}}{N_i!}\\
&\hspace{2.2in}\times
\Bigg\llbracket \Phi_{I;\c}^{(g',\un\ep)}(q)\!\!
\prod_{i=1}^{i=m}\prod_{f\in[N_i]}\!\!\!
\frac{\Phi_{p_{i;f}';b_{i;f}'+\ep_{i;f}+1+b_{i;f}''}\!(q)}{b_{i;f}'!\Phi_0(q)}\Bigg\rrbracket_{q;d'}
\Bigg)\,,
\end{split}\EE
with $t\!\in\!\Z$ defined by 
$$|\p'|-|\b''|=(n\!-\!4)(1\!-\!g')+2|\bfN|+n(d'\!+\!t)
\quad\Llra\quad (\p',\b'')\!\in\!\cS_{g',\bfN}(d',t);$$
if an integer $t$ satisfying the above condition does not exist, 
we define $\wh\si_n^t$  to be~0.\\

\noindent
We now confirm~\e_ref{equivthm_e2} by induction on $3g\!+\!N$.
For $(g,N)\!=\!(0,2)$, 
\e_ref{equivthm_e2} holds by \e_ref{cC2base_e}, \e_ref{p2cC_e}, and \e_ref{coeff2_e}.
Suppose $3g\!+\!N\!\ge\!0$ and~\e_ref{equivthm_e2} holds for smaller values of
this sum.
If $(g',\bfg,\bfS,\bfN)$ is an element of~$\cP_{g,N}^{(m)}$ for some $m\!\in\!\Z^+$,
then 
$$3g_i\!+\!|S_i|\!+\!N_i<3g\!+\!N \qquad\forall\,i\!\in\![m].$$
Thus, \e_ref{cCformula_e4} with the roles of~$\b'$ and~$\b''$ interchanged
and~\e_ref{equivthm_e2} with $(g,N)$ replaced by $(g_i,|S_i|\!+\!N_i)$ imply~that 
\begin{equation*}\begin{split}
&\cC_{g;\p,\b}^{(d)}\sim \sum_{t\in\Z}\wh\si_n^t
\sum_{\begin{subarray}{c}m,d'\in\Z^{\ge0}\\ t'\in\Z\end{subarray}}\!\frac{1}{m!}
\sum_{\begin{subarray}{c}(g',\bfg,\bfS,\bfN)\in\cP_{g,N}^{(m)}\\
(\b'',\un\ep,\c,I)\in\wt\cA^{\star}_{g',\bfN}\\
(\p',\b')\in\cS_{g',\bfN}(d',t')\end{subarray}}
\sum_{\begin{subarray}{c}\bfd,\bft\in(\Z^{\ge0})^m\\ 
|\bfd|=d-d',|\bft|=t-t'\end{subarray}}\!\!\!\!\!\!\!\!  
\Bigg(
\prod_{i=1}^{i=m}\frac{\nc_{g_i;\p|_{S_i}\p_i',\b|_{S_i}\b_i'}^{(d_i,t_i)}}{N_i!}\\
&\hspace{2.4in}\times
\Bigg\llbracket \Phi_{I;\c}^{(g',\un\ep)}(q)\!\!
\prod_{i=1}^{i=m}\prod_{f\in[N_i]}\!\!\!
\frac{\Phi_{p_{i;f}';b_{i;f}''+\ep_{i;f}+1+b_{i;f}'}\!(q)}{b_{i;f}''!\Phi_0(q)}\Bigg\rrbracket_{q;d'}
\Bigg)\,.
\end{split}\end{equation*}
By~\e_ref{coeffdfn_e}, this expression reduces to~\e_ref{equivthm_e2}.

\section{Key combinatorial identities}
\label{combin_sec}

\noindent
We now establish the key combinatorial statements, 
Propositions~\ref{HodgeInt_prp} and~\ref{HodgeIntGS_prp}, 
which characterize the Hodge integrals~\e_ref{DMdfn_e2} and
sums of residues of well-behaved generating series of rational functions
as in Proposition~\ref{pt1sum_prp}, respectively.
Their proofs make use of the basic identities of Lemma~\ref{comb_l0} below.
Lemma~\ref{grcomb_lmm} is used in the proof of Theorem~\ref{GWbound_thm}.

\begin{lmm}[{\cite[Lemma~B.1]{g0ci}}]\label{comb_l0}
The following identities hold:
\begin{alignat*}{2}
\sum_{\begin{subarray}{c}\b'\in(\Z^{\ge0})^m\\
|\b'|=b'\end{subarray}}
\prod_{k=1}^{k=m}\binom{b_k}{b_k'} &=\binom{b_1\!+\!\ldots\!+b_m}{b'}
&\qquad&\forall~m\!\in\!\Z^+,\,b_1,\ldots,b_m,b'\!\in\!\Z^{\ge0}\,,\\
\sum_{b=0}^{\i}(-1)^b\binom{p}{b}\binom{B\!+\!b}{s}
&=(-1)^p \binom{B}{s\!-\!p}
&\qquad&\forall~B,p,s\!\in\!\Z^{\ge0}\,,\\ 
\sum_{p=0}^{\i}(-1)^p\binom{m\!+\!p}{p}\Psi^p &=\frac{1}{(1\!+\!\Psi)^{m+1}}
&\qquad&\forall~m\!\in\!\Z^{\ge0}.
\end{alignat*}
\end{lmm}

\begin{lmm}\label{grcomb_lmm}
Let $g,N\!\in\!\Z^{\ge0}$ with $2g\!+\!N\!\ge\!3$ and 
$\Ga\!\in\!\cA_{g,N}$ be a connected trivalent $N$-marked genus~$g$ graph
as in~\e_ref{GaNdfn_e}.
Then,
\begin{equation*}\begin{split}
&\sum_{\c\in((\Z^{\ge0})^{\i})^{\Ver}}\hspace{-.15in}
2^{-\|\c\|}\!\!\!
\prod_{v\in\Ver}\!\!\Bigg(\!
\frac{(3(g_v\!-\!1)\!+\!|\ov\Fl_v(\Ga)|\!+\!|\c_v|)!}{(3(g_v\!-\!1)\!+\!|\ov\Fl_v(\Ga)|)!|\c_v|!}
\binom{|\c_v|}{\c_v}\!
\prod_{r=1}^{\i}\!\!\bigg(\frac{1}{r\!+\!1}\bigg)^{\!c_{v;r}}\!\!\Bigg)\\
&\hspace{3.5in}=\big(2(1\!-\!\ln2)\big)^{-(3g+N-2-g_{\Ga})}\,.
\end{split}\end{equation*}
\end{lmm}

\begin{proof} For each $v\!\in\!\Ver$, let $m_v\!=\!3(g_v\!-\!1)\!+\!|\ov\Fl_v(\Ga)|$
as before.
Since $\Ga$ is trivalent, $m_v\!\ge\!0$.
Since \hbox{$|2\ln2\!-\!1|\!<\!1$}, the Binomial Theorem gives
\begin{equation*}\begin{split}
&\sum_{\c_v\in(\Z^{\ge0})^{\i}}\!\!\!\!\!\!\!\!
2^{-\|\c_v\|}\!\frac{(m_v\!+\!|\c_v|)!}{m_v!|\c_v|!}\binom{|\c_v|}{\c_v}
\!\!\prod_{r=1}^{\i}\!\!\bigg(\frac{1}{r\!+\!1}\bigg)^{\!c_{v;r}}
=\sum_{c=0}^{\i}\!\frac{(m_v\!+\!c)!}{m_v!c!}
\bigg(\sum_{r=1}^{\i}\frac{2^{-r}}{r\!+\!1}\bigg)^{\!\!c}\\
&\hspace{2in}
=\sum_{c=0}^{\i}\!\frac{(m_v\!+\!c)!}{m_v!c!}\big(2\ln2\!-\!1\big)^{\!c}
=\big(2(1\!-\!\ln2)\big)^{-(m_v+1)}\,.
\end{split}\end{equation*}
Combining this with the second equality in~\e_ref{mvsum_e}
and with the first equality in~\e_ref{gGAcond_e}, we obtain the claim.
\end{proof}

\subsection{Proof of Proposition~\ref{HodgeInt_prp}}
\label{HodgeInt_subs}

\noindent
Let $m\!\in\!\Z^{\ge0}$. 
For a tuple $\b\!\equiv\!(b_k)_{k\in[m]}$ in $(\Z^{\ge0})^m$, we define
$$\binom{|\b|}{\b}=\binom{|\b|}{b_1,\ldots,b_m}
\equiv \frac{|\b|!}{b_1!\ldots\!b_m!}\,.$$
We extend this definition to arbitrary tuples~$\b$ in $\Z^m$ by setting
$$\binom{|\b|}{\b}=0 \qquad\hbox{if}\quad b_k\!<\!0~\hbox{for some}~k\!\in\![m].$$
For  $g,m\!\in\!\Z^{\ge0}$ with $2g\!+\!m\!\ge\!3$ and $I\!\in\!(\Z^{\ge0})^g$, 
we extend the definition in~\e_ref{DMdfn_e2} to arbitrary tuples~$\b$
in $(\Z^{\ge0})^m$ by setting
$$\bllrr{\la_{g;I};\tau_{\b}},\bllrr{\la_{g;I};\wt\tau_{\b}}= 0 
\qquad\hbox{if}\quad b_k\!<\!0~\hbox{for some}~k\!\in\![m].$$

\begin{lmm}\label{DM_lmm}
Let $g\!\in\!\Z^{\ge0}$ and $I\!\in\!(\Z^{\ge0})^g$.
There exists a collection
\BE{DMlmm_e0}
C_{I;\un\ep}^{(g)}\in\Q \qquad\hbox{with}\quad
\un\ep\in(\Z^{\ge0})^m,~m\!\in\!\Z^{\ge0},~2g\!+\!m\!\ge\!3,
\EE
which is invariant under the permutations of the components of~$\un\ep$ 
such~that
\BE{DMlmm_e}  \bllrr{\la_{g;I};\tau_{\b}} =
\sum_{\begin{subarray}{c}\un\ep\in (\Z^{\ge0})^m\\
|\un\ep|\le \mu_g(I)+m\end{subarray}} \hspace{-.22in}
C_{I;\un\ep}^{(g)} \! \binom{|\b|\!-\!|\un\ep|}{\b\!-\!\un\ep}\,\EE
for all $m\!\in\!\Z^{\ge0}$ and $\b\!\in\!\Z^m$  
 with $2g\!+\!m\!\ge\!3$ and $|\b|\!\ge\!\mu_g(I)\!+\!m$,
\BE{DMlmm_e1}
C_{I;\un\ep0}^{(g)} =C_{I;\un\ep}^{(g)}\,, \qquad
C_{I;\un\ep1}^{(g)} =\big(|\un\ep|\!+\!\|I\|\!-\!g\big)C_{I;\un\ep}^{(g)}
~~\hbox{if}~~(g,I)\!\neq\!(1,(0))~\hbox{or}~|\un\ep|\!>\!1\,,\EE
and $C_{I;\un\ep}^{(g)}\!=\!0$ if $|\un\ep|_S\!>\!\mu_g(I)\!+\!|S|$
for some subset $S\!\subset\![m]$ with $2g\!+\!|S|\!\ge\!3$.
If $(g,I)\!=\!(1,(0))$, then the numbers in~\e_ref{DMlmm_e0} can be chosen so 
that in addition
\BE{DMlmm_e1g1} C_{(0);(0)}^{(1)}=0, \qquad
C_{(0);(1,1)}^{(1)}=-C_{(0);(1)}^{(1)}\,.\EE
\end{lmm}

\begin{proof} We set the numbers in~\e_ref{DMlmm_e0} to be 0 whenever 
$|\un\ep|_S\!>\!\mu_g(I)\!+\!|S|$
for some subset $S\!\subset\![m]$ with $2g\!+\!|S|\!\ge\!3$.
Thus,
\begin{gather*}
C_{I;\un\ep}^{(g)}=0~~\hbox{if}~\|I\|\!>\!\max(g,3g\!-\!3), \\
C_{();\un\ve}^{(0)},C_{(1);\un\ve}^{(1)}=0~~\hbox{if}~\un\ve\!\in\!\Z^m\!-\!\{0^m\},
\quad
C_{(0);\un\ve}^{(1)}=0~~\hbox{if}~\un\ve\!\in\Z^m\!-\!\!\{0,1\}^m.
\end{gather*}
With $\0_m\!\in\!\Z^m$ denoting the zero vector, we also define
\BE{DMlmm_e1a}\begin{aligned}
C_{();\0_m}^{(0)}&=C_{();\0_3}^{(0)} &&\!\!\!\forall\,m\!>\!3,\\
C_{(1);\0_m}^{(1)}&=C_{(1);\0_1}^{(1)} &&\!\!\!\forall\,m\!>\!1,
\end{aligned}
\quad
C_{(0);\un\ve}^{(1)}=\begin{cases}
0,&\hbox{if}~\un\ve\!=\!\0_m;\\
-(|\un\ve|\!-\!2)!C_{(0);(1)}^{(1)},&\hbox{if}~\un\ve\!\in\!\{0,1\}^m,\,|\ve|\!\ge\!2.
\end{cases}\EE
The coefficients defined in this way are invariant under the permutations of 
the components of~$\un\ep$ and satisfy~\e_ref{DMlmm_e1},
the vanishing condition after~\e_ref{DMlmm_e1}, and
\e_ref{DMlmm_e1g1} if $(g,I)\!=\!(1,(0))$.\\

\noindent
Let $g\!\ge\!2$.
We~set 
\BE{DMlmm_e2a}C_{I;\un\ep}^{(g)}=0  \qquad\hbox{if}\quad
|\un\ep|\!<\!\mu_g(I)\!+\!m~~\hbox{and}~~ \ve_k\!\ge\!2~\forall\,k\!\in\![m].\EE
Suppose $m\!\in\!\Z^+$ and
the numbers $C_{I;\un\ep}^{(g)}\!\in\!\Q$ with $\un\ep\in(\Z^{\ge0})^{m-1}$
are invariant  under the permutations of the components of~$\un\ep$.
If $m\!\ge\!2$, assume in addition~that
\BE{DMlmm_e1b} C_{I;\un\ep1}^{(g)}=\big(|\un\ep|\!+\!\|I\|\!-\!g\big)C_{I;\un\ep0}^{(g)}
\qquad\forall~\un\ep\in(\Z^{\ge0})^{m-2}\,.\EE
The conditions~\e_ref{DMlmm_e1} and the symmetry requirement then determine the numbers 
\BE{DMlmm_e1d}C_{I;\un\ep}^{(g)}\in\Q \qquad\hbox{with}\quad 
\un\ep\!\equiv\!(\ep_1,\ldots,\ep_m)\in(\Z^{\ge0})^m~~\hbox{s.t.}~~
\ve_k\!<\!2~\hbox{for some}~k\!\in\![m].\EE
If $m\!\ge\!2$, the condition~\e_ref{DMlmm_e1b} holds if 
the numbers $C_{I;\un\ep}^{(g)}\!\in\!\Q$ with $\un\ep\in(\Z^{\ge0})^{m-1}$
satisfy~\e_ref{DMlmm_e1}.
Thus, the numbers~\e_ref{DMlmm_e0} invariant under the permutations of 
the components of~$\un\ep$ and satisfying~\e_ref{DMlmm_e1},
the vanishing condition after~\e_ref{DMlmm_e1}, and 
the additional condition~\e_ref{DMlmm_e2a} are determined by the numbers
\BE{DMlmm_e1e}
C_{I;\un\ep}^{(g)}\in\Q \quad\hbox{with}~~
\un\ep\!\equiv\!(\ep_1,\ldots,\ep_m)\in(\Z^{\ge0})^m,\,m\!\in\!\Z^{\ge0},~~
\ve_k\!\ge\!2~\forall\,k\!\in\![m],~|\un\ep|\!=\!\mu_g(I)\!+\!m\EE
invariant under the permutations of the components of~$\un\ep$.\\

\noindent
Let $g\!\in\!\Z^{\ge0}$ be arbitrary.
For $m\!\in\!\Z^{\ge0}$ and $i\!\in\![m]$, denote by $e_i\!\in\!\Z^m$
the $i$-th standard coordinate vector.
For all \hbox{$\b\!\in\!\Z^m\!-\!\{0\}$}, 
\BE{binomred_e} \binom{|\b|}{\b}=\sum_{i=1}^{i=m}\!
\binom{|\b|\!-\!1}{\b\!-\!e_i}\,.\EE
For $m\!\in\!\Z^{\ge0}$ and $\b\!\in\!(\Z^{\ge0})^m$  with $2g\!+\!m\!\ge\!3$ and
$|\b|\!=\!\mu_g(I)\!+\!m$, let 
$$\llrr{\b}_{g;I}=\sum_{\begin{subarray}{c}\un\ep\in (\Z^{\ge0})^m\\
|\un\ep|\le \mu_g(I)+m\end{subarray}} \hspace{-.2in}
C_{I;\un\ep}^{(g)} \! \binom{|\b|\!-\!|\un\ep|}{\b\!-\!\un\ep}\,.$$
By the first relation in~\e_ref{DMlmm_e1},
the vanishing condition after~\e_ref{DMlmm_e1}, and~\e_ref{binomred_e}, 
\BE{DMlmm_e2c}\llrr{\b0}_{g;I}=
\sum_{i=1}^{i=m}\bllrr{\b\!-\!e_i}_{g;I}\EE
for all $m\!\in\!\Z^{\ge0}$ and $\b\!\in\!(\Z^{\ge0})^m$ with $2g\!+\!m\!\ge\!3$ and
$|\b|\!=\!\mu_g(I)\!+\!m\!+\!1$.
By~\e_ref{DMlmm_e1} and the vanishing condition after~\e_ref{DMlmm_e1},
\BE{DMlmm_e2d}\llrr{\b1}_{g;I}=\big(|\b|\!+\!1\!-\!g\!+\|I\|\big)\llrr{\b}_{g;I}\EE
for all $m\!\in\!\Z^{\ge0}$ and $\b\!\in\!(\Z^{\ge0})^m$  with $2g\!+\!m\!\ge\!3$ and
$|\b|\!=\!\mu_g(I)\!+\!m$;
if $(g,I)\!=\!(1,(0))$, we also need to use~\e_ref{DMlmm_e1g1} and~\e_ref{binomred_e}
to obtain~\e_ref{DMlmm_e2d}.\\

\noindent
Let $m\!\in\!\Z^{\ge0}$ with $2g\!+\!m\!\ge\!3$.
By the string and dilaton equations \cite[Section~26.3]{MirSym},
\BE{StrDil_e}\begin{split}
\bllrr{\la_{g;I};\tau_{\b0}}&=
\sum_{i=1}^{i=m}\bllrr{\la_{g;I};\tau_{\b-e_i}}
\qquad\hbox{and}\\
\bllrr{\la_{g;I};\tau_{\b1}}&=\big(|\b|\!+\!1\!-\!g\!+\|I\|\big)\bllrr{\la_{g;I};\tau_{\b}},
\end{split}\EE
respectively.
By the first case in~\e_ref{DMdfn_e} and the first identity in~\e_ref{StrDil_e},
$$ \bllrr{\la_{g;I};\tau_{\b}}=\bllrr{\la_{g;I};\tau_{\b0}}
=\sum_{i=1}^{i=m}\!\bllrr{\la_{g;I};\tau_{\b-e_i}}
\quad\hbox{if}~~  |\b|\!>\!\mu_g(I)\!+\!m\,.$$
Along with~\e_ref{DMlmm_e2c}, this implies that 
\e_ref{DMlmm_e} for all $\b\!\in\!(\Z^{\ge0})^m$  
and $|\b|\!\ge\!\mu_g(I)\!+\!m$ is equivalent~to
\BE{DMlmm_e2} \bllrr{\la_{g;I};\tau_{\b}} = \llrr{\b}_{g;I}
\quad\forall~\b\!\in\!(\Z^{\ge0})^m~\hbox{s.t.}~|\b|\!=\!\mu_g(I)\!+\!m.\EE
By the symmetry of the numbers on the two sides of~\e_ref{DMlmm_e2} 
and \e_ref{DMlmm_e2c}-\e_ref{StrDil_e},
\e_ref{DMlmm_e2} for all $m\!\in\!\Z^{\ge0}$ with $2g\!+\!m\ge\!3$
is in turn equivalent
to~\e_ref{DMlmm_e2} for all $m\!\in\!\Z^{\ge0}$ and $\b\!\in\!(\Z^{\ge0})^m$ 
with  either \hbox{$2g\!+\!m\!=\!3$} or $2g\!+\!m\!>\!3$ and $b_k\!\ge\!2$ for all $k\!\in\![m]$.\\

\noindent
For $g\!=\!0,1$, \e_ref{DMlmm_e} thus reduces to 
\BE{g01base_e}\int_{\ov\cM_{0,3}}\!\!\!\!\!\!1=C^{(0)}_{();\0_3},\qquad
\int_{\ov\cM_{1,1}}\!\!\!\!\!\!\la_1=C^{(1)}_{(1);\0_1},\qquad
\int_{\ov\cM_{1,1}}\!\!\!\!\!\!\psi_1=C^{(1)}_{(0);(1)}\,.\EE
For $g\!\ge\!2$, \e_ref{DMlmm_e} reduces to 
\BE{DMlmm_e3}
\int_{\ov\cM_{g,m}}\!\!\!\!\!\!\la_{g;I}\!\!\prod_{k=1}^m\!\!\psi_k^{b_k}
=\sum_{\begin{subarray}{c}\un\ep\in (\Z^{\ge0})^m\\
|\un\ep|\le \mu_g(I)+m\end{subarray}} \hspace{-.2in}
C_{I;\un\ep}^{(g)} \! \binom{|\b|\!-\!|\un\ep|}{\b\!-\!\un\ep}\EE
for all $m\!\in\!\Z^{\ge0}$ and $\b\!\equiv\!(b_k)_{k\in[m]}\!\in\!(\Z^{\ge0})^m$ 
with $|\b|\!=\!\mu_g(I)\!+\!m$ and $b_k\!\ge\!2$ for all $k\!\in\![m]$.
We take $C^{(g)}_{I;()}$ to be the $m\!=\!0$ number on the left-hand side of~\e_ref{DMlmm_e3};
it vanishes unless $\mu_g(I)\!=\!0$.\\

\noindent
Suppose $g\!\ge\!2$, $\mu_g(I)\!\ge\!0$, $m^*\!\!\in\!\Z^+$, 
and we have constructed the numbers~\e_ref{DMlmm_e0} 
with the required properties for all~$m\!<\!m^*$.
The only numbers~\e_ref{DMlmm_e0} for $m\!=\!m^*$ that remain to be determined
are the numbers~$C_{I;\un\ep}^{(g)}$ in~\e_ref{DMlmm_e1e} with $m\!=\!m^*$.
The only one of these numbers appearing in the equation~\e_ref{DMlmm_e3} 
corresponding to $\b\!\in\!(\Z^{\ge0})^{m^*}$ 
with $|\b|\!=\!\mu_g(I)\!+\!m^*$
with a nonzero coefficient is the number indexed by $\un\ep\!=\!\b$.
Thus, the equations~\e_ref{DMlmm_e3} 
with $|\b|\!=\!\mu_g(I)\!+\!m^*$ determine the numbers~$C_{I;\un\ep}^{(g)}$ 
in~\e_ref{DMlmm_e1e}  with $m\!=\!m^*$.
By the invariance of the sides two of~\e_ref{DMlmm_e3} under the permutations 
of the components of~$\b$, the numbers~$C_{I;\un\ep}^{(g)}$ in~\e_ref{DMlmm_e1e} determined
by these equations are invariant 
 under the permutations of the components of~$\un\ep$.
This establishes the existence of the numbers~$C_{I;\un\ep}^{(g)}$ 
in~\e_ref{DMlmm_e0} satisfying all 
requirement of the lemma.
\end{proof}

\begin{lmm}\label{DM_lmm2}
Let $g\!\in\!\Z^{\ge0}$ and $I\!\in\!(\Z^{\ge0})^g$.
With $C_{I;\un\ep}^{(g)}$ as in Lemma~\ref{DM_lmm},
\begin{equation*}\begin{split}
A_{I,\b;\c}^{(g)}&=~~(-1)^{\|\c\|}\hspace{-.5in}
\sum_{\begin{subarray}{c}
(\un\ep,\un\ep')\in (\Z^{\ge0})^m\times (\Z^{\ge0})^{S(\c)}\\
|\un\ep|+|\un\ep'|\le\mu_g(I)+m+|\c|\end{subarray}} \!\!\!
\Bigg(C_{I;\un\ep\un\ep'}^{(g)}\,
\big(\mu_g(I)\!+\!m\!+\!|\c|\!-\!(|\un\ep|\!+\!|\un\ep'|)\big)!
\!\!\prod_{(r,j)\in S(\c)}\!\!\!\!\!\!\!\ep_{r,j}'!\binom{r}{\ep_{r,j}'}\\
&\hspace{2.3in}
\times
\binom{|\b|\!-\!|\un\ep|}{\mu_g(I)\!+\!m\!+\!|\c|\!-\!\|\c\|\!-\!|\un\ep|}
\prod_{k=1}^m\!\ep_k!\binom{b_k}{\ep_k}\Bigg)\,.
\end{split}\end{equation*}
for all $m\!\in\!\Z^{\ge0}$ with $2g\!+\!m\!\ge\!3$,
$\b\!\equiv\!(b_k)_{k\in[m]}\in\!(\Z^{\ge0})^m$, and
$\c\!\in\!(\Z^{\ge0})^{\i}$.
\end{lmm}

\begin{proof}
By ~\e_ref{DMlmm_e} and the second case in~\e_ref{DMdfn_e}, 
\BE{DM2_e1}\begin{split}
A_{I,\b;\c}^{(g)}&=\!\!\!
\sum_{\begin{subarray}{c}\b'\in(\Z^{\ge0})^{S(\c)}\\
(\un\ep,\un\ep')\in (\Z^{\ge0})^m\times (\Z^{\ge0})^{S(\c)}\\
|\un\ep|+|\un\ep'|\le \mu_g(I)+m+|\c|\le |\b|+|\b'|\end{subarray}}
\hspace{-.5in}\Bigg(
\frac{(-1)^{|\b'|}C_{I;\un\ep\un\ep'}^{(g)}\,(|\b|\!+\!|\b'|\!-\!(|\un\ep|\!+\!|\un\ep'|))!}
{(|\b|\!+\!|\b'|-(\mu_g(I)\!+\!m+\!|\c|))!}\\
&\hspace{1.5in}
\times\prod_{k=1}^m\!\ep_k!\binom{b_k}{\ep_k}
\prod_{(r,j)\in S(\c)}\!\!\!\frac{r!}{(r\!-\!\ep_{r,j}')!}
\binom{r\!-\!\ep_{r,j}'}{b_{r,j}'\!-\!\ep_{r,j}'}\!\!\Bigg).
\end{split}\EE
If $\c\!=\!\0$ and $|\b|\!\ge\!\mu_g(I)\!+\!m$, this expression becomes
\begin{equation*}\begin{split}
A_{I,\b;\0}^{(g)}=\!\!\!
\sum_{\begin{subarray}{c}
\un\ep\in (\Z^{\ge0})^m\\ |\ep|\le \mu_g(I)+m\end{subarray}}
\hspace{-.1in}
\frac{C_{I;\un\ep}^{(g)}\,(|\b|\!-\!|\un\ep|)!}
{(|\b|-(\mu_g(I)\!+\!m))!}
\prod_{k=1}^m\!\ep_k!\binom{b_k}{\ep_k}.
\end{split}\end{equation*}
If $\c\!=\!\0$ and $|\b|\!<\!\mu_g(I)\!+\!m$, the first sum vanishes.
This establishes the claim in both $\c\!=\!\0$ cases.\\

\noindent
Suppose $\c\!\neq\!\0$.
By~\e_ref{DM2_e1} and the first statement in Lemma~\ref{comb_l0},
\begin{equation*}\begin{split}
A_{I,\b;\c}^{(g)}&=\!\!\!
\sum_{\begin{subarray}{c}
(\un\ep,\un\ep')\in (\Z^{\ge0})^m\times (\Z^{\ge0})^{S(\c)}\\
|\un\ep|+|\un\ep'|\le \mu_g(I)+m+|\c|\end{subarray}} 
\hspace{-.07in}\Bigg(\!\!
(-1)^{|\un\ep'|}C_{I;\un\ep\un\ep'}^{(g)}
\big(\mu_g(I)\!+\!m\!+\!|\c|\!-\!(|\un\ep|\!+\!|\un\ep'|)\big)!
\!\!\!\prod_{(r,j)\in S(\c)}\!\!\!\!\!\!\!\!\ep_{r,j}'!\binom{r}{\ep_{r,j}'}\\
&\hspace{1.2in}\times
\prod_{k=1}^m\!\ep_k!\binom{b_k}{\ep_k}
\sum_{b'=0}^{\i}(-1)^{b'}\!
\binom{|\b|\!+\!b'\!-\!|\un\ep|}{\mu_g(I)\!+\!m\!+\!|\c|\!-\!|\ep|\!-\!|\ep'|}\!\!
\binom{\|\c\|\!-\!|\un\ep'|}{b'}\!\!\!\Bigg)\,.
\end{split}\end{equation*}
The claim now follows from the second statement in Lemma~\ref{comb_l0}.
\end{proof}

\begin{proof}[{\bf{\emph{Proof of Proposition~\ref{HodgeInt_prp}}}}]
By Lemma~\ref{DM_lmm2} and the vanishing statement after~\e_ref{DMlmm_e1}, 
\e_ref{DMstr_e} holds~with
\BE{DMcrl2_e1}A^{(g,\un\ep)}_{I;\c}= \prod_{k=1}^m\!\!\ep_k!
\!\!\!\!\!\!
\sum_{\begin{subarray}{c}
\un\ep'\in(\Z^{\ge0})^{S(\c)}\\
|\un\ep|+|\un\ep'|\le \mu_g(I)+m+|\c|\end{subarray}} \hspace{-.45in}
C_{I;\un\ep\un\ep'}^{(g)}\,
\big(\mu_g(I)\!+\!m\!+\!|\c|\!-\!(|\un\ep|\!+\!|\un\ep'|)\big)!
\!\!\prod_{(r,j)\in S(\c)}\!\!\!\!\!\!\!\ep_{r,j}'!\binom{r}{\ep_{r,j}'}\,.\EE
Since the coefficients $C_{I;\un\ep\un\ep'}^{(g)}$ provided by Lemma~\ref{DM_lmm}
are invariant under the permutations of the components of~$\un\ep$,
so are the numbers~\e_ref{DMcrl2_e1}.\\

\noindent
By the vanishing statement in Lemma~\ref{DM_lmm},
$$\sum_{\begin{subarray}{c}1\le i\le m\\ \ep_i\ge2\end{subarray}}
\!\!\!\!\ep_i\le 6g\!+\!3 
\qquad\hbox{if}~\un\ep\!\in\!(\Z^{\ge0})^m,~I\!\in\!(\Z^{\ge0})^g,~
2g\!+\!m\!\ge\!3,~C_{I;\un\ep}^{(g)}\!\neq\!0.$$
Along with~\e_ref{DMlmm_e1}, this implies that there exists 
$C_g\!\in\!\R$ such~that 
\BE{Abnd_e3}\begin{split}
\prod_{i=1}^{i=m}\!\!\ep_i!\,\big|C_{I;\un\ep}^{(g)}\big|\,
\big(3(g\!-\!1)\!+\!m\!-\!\|I\|\!-\!|\un\ep|\big)!\
&\le C_g\big(|\un\ep|\!+\!\|I\|\big)!\big(3(g\!-\!1)\!+\!m\!-\!\|I\|\!-\!|\un\ep|\big)!\\
&\le C_g\,\big(3(g\!-\!1)\!+\!m\big)!
\end{split}\EE
for all $I\!\in\!(\Z^{\ge0})^g$,   
$\un\ep\!\in\!(\Z^{\ge0})^m$, and $m\!\in\!\Z^{\ge0}$ with $2g\!+\!m\!\ge\!3$.
By~\e_ref{DMcrl2_e1}, \e_ref{Abnd_e3}, and the first statement in Lemma~\ref{comb_l0},
\begin{equation*}\begin{split}
\big|A_{I;\c}^{(g,\un\ep)}\big|&\le  C_g\,\big(3(g\!-\!1)\!+\!m\!+\!|\c|\big)!
\!\!\!
\sum_{\begin{subarray}{c}\un\ep'\in(\Z^{\ge0})^{S(\c)}\\
|\un\ep|+|\un\ep'|\le \mu_g(I)+m+|\c|\end{subarray}} 
\hspace{-.2in}\prod_{(r,j)\in S(\c)}\!\!\!\!\binom{r}{\ep_{r,j}'}\\
&\le 
C_g\,\big(3(g\!-\!1)\!+\!m\!+\!|\c|\big)!\!\!
\sum_{\ep'=0}^{\ep'=\|\c\|}\!\binom{\|\c\|}{\ep'}
= C_g\,2^{\|\c\|}
\big(3(g\!-\!1)\!+\!m\!+\!|\c|\big)!\,.
\end{split}\end{equation*}
This establishes~\e_ref{DMstr_e2}.
\end{proof}

\begin{crl}\label{DM_crl3}
Let $g\!\in\!\Z^{\ge0}$ and $I\!\in\!(\Z^{\ge0})^g$. Then,
$$A_{I,\b;(c_1,c_2,\ldots)}^{(g)}
=(-1)^{c_1}
\frac{(2g\!+\!m\!-\!3\!+\!c_1\!+\!c_2\!+\!\ldots)!}{(2g\!+\!m\!-\!3\!+\!c_2\!+\!\ldots)!}
A_{I,\b;(0,c_2,\ldots)}^{(g)}$$
for all $m\!\in\!\Z^{\ge0}$ with $2g\!+\!m\!\ge\!3$,
$\b\!\in\!(\Z^{\ge0})^m$, and
$\c\!\equiv\!(c_r)_{r\in\Z^+}\!\in\!(\Z^{\ge0})^{\i}$.
\end{crl}

\begin{proof}
Let $A_{I;\c}^{(g,\un\ep)}$ be as in~\e_ref{DMcrl2_e1}
and $e_1\!=\!(1,0,0,\ldots)\!\in\!(\Z^{\ge0})^{\i}$.
If $(g,I)\!\neq\!(1,(0))$ or $|\un\ve|\!\ge\!2$, then
\e_ref{DMlmm_e1} and the vanishing statement after~\e_ref{DMlmm_e1} imply~that 
$$A_{I;\c+e_1}^{(g,\un\ep)}= \big(2g\!-\!2\!+\!m\!+\!|\c|\big) A_{I;\c}^{(g,\un\ep)}\,.$$
If $(g,I)\!=\!(1,(0))$, then
\e_ref{DMlmm_e1}, the vanishing statement after~\e_ref{DMlmm_e1}, and
\e_ref{DMlmm_e1g1} imply that 
$$A_{I;\c+e_1}^{(g,\un\ep)}= \big(2g\!-\!2\!+\!m\!+\!|\c|\big) A_{I;\c}^{(g,\un\ep)}
+C^{(1)}_{(0);(1)}\big(m\!+\!|\c|\!-\!1\big)!\cdot
\begin{cases}
(-1),&\hbox{if}~|\un\ep|\!=\!1;\\
(m\!+\!|\c|\!-\!\|\c\|\big),&\hbox{if}~|\un\ep|\!=\!0.
\end{cases}$$
Combining these observations with~\e_ref{DMstr_e}, we obtain 
$$A_{I,\b;\c+e_1}^{(g)}=-\big(2g\!-\!2\!+\!m\!+\!|\c|\big) A_{I,\b;\c}^{(g)}\,.$$
The claim follows from this identity. 
\end{proof}

\begin{eg}\label{AgepIc_eg01} 
By~\e_ref{DMlmm_e1a} and~\e_ref{g01base_e}, 
$$C_{();\un\ve}^{(0)}=\begin{cases}1,&\hbox{if}~\un\ve\!=\!\0;\\
0,&\hbox{otherwise};\end{cases} \qquad
C_{(1);\un\ve}^{(1)}=\frac1{24}\cdot\begin{cases}1,&\hbox{if}~\un\ve\!=\!\0;\\
0,&\hbox{otherwise}.\end{cases}$$
Combining this with~\e_ref{DMcrl2_e1}, we obtain
$$A^{(0,\un\ve)}_{();\c}=\begin{cases}(m\!-\!3\!+\!|\c|)!,
&\hbox{if}~\un\ve\!=\!\0\!\in\!(\Z^{\ge0})^m;\\
0,&\hbox{if}~\un\ve\!\neq\!\0;\end{cases} \quad
A^{(1,\un\ve)}_{(1);\c}=\begin{cases}\frac{(m-1+|\c|)!}{24},
&\hbox{if}~\un\ve\!=\!\0\!\in\!(\Z^{\ge0})^m;\\
0,&\hbox{if}~\un\ve\!\neq\!\0.\end{cases}$$
Since
$$\bllrr{\la_{0;()};\wt\tau_{\b}}=
\begin{cases}|\b|!,&\hbox{if}~|\b|\!\ge\!m\!-\!3;\\
0,&\hbox{otherwise}; \end{cases}  \quad
\bllrr{\la_{1;(1)};\wt\tau_{\b}}=\frac1{24}\cdot
\begin{cases}|\b|!,&\hbox{if}~|\b|\!\ge\!m\!-\!1;\\
0,&\hbox{otherwise}; \end{cases}$$
the $(g,I)\!=\!(0,()),(1,(1))$ cases of~\e_ref{DMstr_e} reduce~to
\begin{equation*}\begin{split}
\sum_{\b'\in(\Z^{\ge0})^{S(\c)}}\!\!\!\Bigg(\!\! (-1)^{|\b'|}
\binom{|\b|+\!|\b'|}{m\!-\!3\!+\!|\c|}\!\!
\prod_{(r,j)\in S(\c)}\!\!\binom{r}{b_{r,j}'}\!\!\Bigg)
&=(-1)^{\|\c\|}\big(m\!+\!|\c|\!-\!3\big)!\binom{|\b|}{m\!-3\!+\!|\c|\!-\!\|\c\|},\\
\sum_{\b'\in(\Z^{\ge0})^{S(\c)}}\!\!\!\Bigg(\!\! (-1)^{|\b'|}
\binom{|\b|+\!|\b'|}{m\!-\!1\!+\!|\c|}\!\!
\prod_{(r,j)\in S(\c)}\!\!\binom{r}{b_{r,j}'}\!\!\Bigg)
&=(-1)^{\|\c\|}\big(m\!+\!|\c|\!-\!1\big)!\binom{|\b|}{m\!-1\!+\!|\c|\!-\!\|\c\|}
\end{split}\end{equation*}
for $m\!\ge\!3$ and $m\!\ge\!1$, respectively. 
These two identities are immediate consequences 
of the first two statements in Lemma~\ref{comb_l0}.
\end{eg}

\begin{eg}\label{AgepIc_eg1} 
By~\e_ref{DMlmm_e1a} and~\e_ref{g01base_e}, 
$$C_{(0);\un\ve}^{(1)}=\frac1{24}\cdot\begin{cases}
1,&\hbox{if}~|\un\ve|\!=\!1;\\
-(|\un\ve|\!-\!2)!,&\hbox{if}~\un\ve\!\in\!\{0,1\}^m,\,|\ve|\!\ge\!2;\\
0,&\hbox{otherwise}.
\end{cases}$$
Combining this with~\e_ref{DMcrl2_e1}, we obtain
$$A^{(1,\un\ep)}_{(0);\c}=-\frac{1}{24}\!\!\sum_{\c'\in(\Z^{\ge0})^{\i}}
\!\!\!\!\!\!\!\! \big(|\un\ep|\!+\!|\c'|\!-\!2\big)!
\big(m\!+\!|\c|\!-\!(|\un\ep|\!+\!|\c'|)\!\big)!
\prod_{r=1}^{\i}\!\binom{c_r}{c_r'}r^{c_r'}
\quad\hbox{if}~\un\ve\!\in\!\{0,1\}^m,$$
with $(-1)!\!\equiv\!-1$ and $(-2)!\!\equiv\!0$;
all other coefficients $A^{(g,\un\ep)}_{I;\c}$ with $(g,I)\!=\!(1,(0))$ provided
by~\e_ref{DMcrl2_e1}  vanish.
Thus,
$$A^{(1,\un\ve)}_{(0);\0}=\frac{(m\!-\!|\un\ve|)!}{24}\cdot\begin{cases}
1,&\hbox{if}~\un\ve\!\in\!\{0,1\}^m,\,|\un\ve|\!=\!1;\\
-(|\un\ve|\!-\!2)!,&\hbox{if}~\un\ve\!\in\!\{0,1\}^m,\,|\ve|\!\ge\!2;\\
0,&\hbox{otherwise}.\end{cases}$$
If $r\!\in\!\Z^+$ and $e_r\!\in\!(\Z^{\ge0})^{\i}$ denotes the $r$-th standard coordinate vector,
then
$$A^{(1,\un\ve)}_{(0);e_r}=\frac{(m\!-\!|\un\ve|)!}{24}\cdot\begin{cases}
r,&\hbox{if}~\un\ve\!\in\!\{0,1\}^m,\,|\un\ve|\!=\!0;\\
(m\!-\!r),&\hbox{if}~\un\ve\!\in\!\{0,1\}^m,\,|\un\ve|\!=\!1;\\
-(|\un\ve|\!-\!2)!((|\un\ve|\!-\!1)(r\!-\!1)\!+\!m),
&\hbox{if}~\un\ve\!\in\!\{0,1\}^m,\,|\ve|\!\ge\!2;\\
0,&\hbox{otherwise}.
\end{cases}$$
\end{eg}

\subsection{Sums of residues of generating series}
\label{HodgeIntGS_subs}

\noindent
For $g,m\!\in\!\Z^{\ge0}$ with $2g\!+\!m\!\ge\!3$, $I\!\in\!(\Z^{\ge0})^g$,
 $\c\!\in\!(\Z^{\ge0})^{\i}$, and $\un\ep\!\in\!(\Z^{\ge0})^m$, let
\hbox{$\wh{A}_{I;\c}^{(g,\un\ep)}\!\in\!\Q$} be as in~\e_ref{hatAdfn_e}.

\begin{prp}\label{HodgeIntGS_prp}
Let $g,m\!\in\!\Z^{\ge0}$ with $2g\!+\!m\!\ge\!3$, $I\!\in\!(\Z^{\ge0})^g$,
and $\b\!\in\!(\Z^{\ge0})^m$.
If
\BE{reslmm_e0}\ze,\Psi_0,\Psi_1,\ldots\!\in\!q\Q_{\al}(\hb)[[q]]
\quad\hbox{and}\quad
1+\cZ^*(\hb,q)=e^{\ze(q)/\hb}\bigg(1+\sum_{b=0}^{\i}\Psi_b(q)\hb^b\bigg),\EE
then
\BE{res_e}\begin{split}
&\sum_{m'=0}^{\i} 
\sum_{\begin{subarray}{c}\b'\in(\Z^{\ge0})^{m'} \\
|\b|+|\b'|=\mu_g(I)+m+m'\end{subarray}}
\hspace{-.37in}
\Bigg(\!\frac{\llrr{\la_{g;I};\wt\tau_{\b\b'}}}{m'!}
\prod_{k=1}^{k=m'}\!\!
\Res{\hb=0}\Big\{\frac{(-\hb)^{-b_k'}}{b_k'!}\cZ^*(\hb,q)\!\Big\}\!\!\Bigg)\\
&\hspace{.5in}=
\sum_{\c\in(\Z^{\ge0})^{\i}} \!\!\!
\sum_{\begin{subarray}{c} \un\ep\in (\Z^{\ge0})^m\\
|\un\ep|\le\mu_g(I)+m\\ \ep_k\le b_k~\forall k\in[m]\end{subarray}} \hspace{-.15in}
\Bigg(\!\frac{(-1)^{|\c|+\|\c\|}\wh{A}_{I;\c}^{(g,\un\ep)}}{(1\!+\!\Psi_0(q))^{2g-2+m}}
\prod_{r=1}^{\i}\frac{1}{c_r!}
\bigg(\!\frac{\Psi_r(q)}{(r\!+\!1)!\,(1\!+\!\Psi_0(q))}\!\bigg)^{\!c_r}\\
&\hspace{1.5in}
\times\ze(q)^{|\b|-(\mu_g(I)+m-\|\c\|)}
\binom{|\b|\!-\!|\un\ep|}{\mu_g(I)\!+\!m\!-\|\c\|\!-\!|\un\ep|}
\!\!\prod_{k=1}^m\!\!\frac{b_k!}{(b_k\!-\!\ep_k)!}
\!\Bigg)      
\end{split}\EE
in $\Q_{\al}[[q]]$.
\end{prp}

\begin{proof}
Fix $c_0\!\in\!\Z^{\ge0}$ and $\c\!\equiv\!(c_r)_{r\in\Z^+}\!\in\!(\Z^{\ge0})^{\i}$.
Let
\begin{gather*}
\Psi^{\c}=\prod_{r=1}^{\i}\!\Psi_r^{c_r}\,, \qquad
\om(\c)=\prod_{r=1}^{\i}\!\!\big(\!(r\!+\!1)!\big)^{c_r}\,,\\
S(c_0,\c)=\big\{(r,j)\!\in\!\Z^{\ge0}\!\times\!\Z^+\!\!:
(r,j)\!\in\!\{r\}\!\times\![c_r]~\forall\,r\!\in\!\Z^{\ge0}\big\},\\
A_{\b,I;c_0,\c}^{(g)}=
\sum_{\b'\in(\Z^{\ge0})^{S(c_0,\c)}}\!\!\!\!\!\!\!\!\!\!\! (-1)^{|\b'|}
\frac{\llrr{\la_{g;I};\wt\tau_{\b\b'}}}
{(|\b|\!+\!|\b'|\!-\!\mu_g(I)\!-\!m\!-\!c_0\!-\!|\c|)!}  \!\!
\prod_{(r,j)\in S(c_0,\c)}\!\!\!\binom{r\!+\!1}{b_{r,j}'}.
\end{gather*}
In particular, $|S(c_0,\c)|\!=\!c_0\!+\!|\c|$ and
the numerator above vanishes whenever
the argument of the factorial in the denominator is negative.
By Corollary~\ref{DM_crl3} and Proposition~\ref{HodgeInt_prp}, 
\BE{hatA_e}\begin{split}
A_{\b,I;c_0,\c}^{(g)}
&=(-1)^{c_0+|\c|+\|\c\|}
\frac{(2g\!+\!m\!-\!3\!+\!c_0\!+\!|\c|)!}{(2g\!+\!m\!-\!3\!+\!|\c|)!}\\
&\hspace{.8in}
\times\sum_{\begin{subarray}{c} \un\ep\in (\Z^{\ge0})^m\\
|\un\ep|\le \mu_g(I)+m\\ \ep_k\le b_k~\forall k\in[m]\end{subarray}} \hspace{-.25in}
\wh{A}_{I;\c}^{(g,\un\ep)}
\binom{|\b|\!-\!|\un\ep|}{\mu_g(I)\!+\!m\!-\|\c\|\!-\!|\un\ep|}
\!\!\prod_{k=1}^m\!\!\frac{b_k!}{(b_k\!-\!\ep_k)!}\,,
\end{split}\EE
with $\wh{A}_{I;\c}^{(g,\un\ep)}$ as in~\e_ref{hatAdfn_e}.\\

\noindent
We establish~\e_ref{res_e} by comparing  
the coefficients of $\Psi_0^{c_0}\Psi^{\c}$ on the two sides.
By~\e_ref{reslmm_e0}, 
\BE{HodgeIntGS_e5}\Res{\hb=0}\Big\{\hb^{-b} \cZ^*(\hb,q)\Big\}
=\sum_{r=\max(b-1,0)}^{\i}\!\frac{\ze(q)^{r+1-b}}{(r\!+\!1\!-\!b)!}\Psi_r(q)
+\begin{cases} \ze(q),&\hbox{if}~b\!=\!0;\\
0,&\hbox{if}~b\!\ge\!1.\end{cases}\EE
The coefficient $\LHS_{m'}(c_0,\c)$  of $\Psi_0^{c_0}\Psi^{\c}$ in the 
$m'$-th summand on the left-hand side of~\e_ref{res_e}
is a sum over the collections of disjoint subsets $S_0,S_1,\ldots$ of $[m']$
of cardinalities $c_0,c_1,\ldots$ and tuples $\b'\!\in\!(\Z^{\ge0})^{S(c_0,\c)}$ such~that  
$$|\b|\!+\!|\b'|=\mu_g(I)+m+m'\in\Z^{\ge0}\,.$$
The factors in the $m'$-fold product in \e_ref{res_e} 
that contribute $\Psi_r$ are indexed by the elements of~$S_r$;
the $j$-th such factor arises from $\Res{\hb=0}\{\hb^{-b_{r,j}'}\cZ^*(\hb,q)\}$
with $r\!\ge\!b_{r,j}'\!-\!1$.
This leaves $m'\!-\!c_0\!-\!|\c|$ factors that contribute~$\ze(q)$
from $\Res{\hb=0}\{\cZ^*(\hb,q)\}$.
The associated summand contributing to $\LHS_{m'}(c_0,\c)$ is~then
\begin{equation*}\begin{split}
&\frac{\llrr{\la_{g;I};\wt\tau_{\b\b'}}}{m'!}
\,\ze^{m'-c_0-|\c|} \!\!\!\!\!\!
\prod_{(r,j)\in S(c_0,\c)}\!\!\!
\bigg(\frac{(-1)^{b_{r,j}'}}{b_{r,j}'!}\cdot
\frac{\ze^{r+1-b_{r,j}'}}{(r\!+\!1\!-\!b_{r,j}')!}\bigg)\\
&\hspace{1in}
=\frac{\ze^{|\b|-\mu_g(I)-m+\|\c\|}}{\om(\c)}\,
(-1)^{|\b'|}\frac{\llrr{\la_{g;I};\wt\tau_{\b\b'}}}{m'!}
\prod_{(r,j)\in S(c_0,\c)}\!\!\!\binom{r\!+\!1}{b_{r,j}'} \,;
\end{split}\end{equation*}
the first expression above is defined to be~0 if 
$b_{r,j}'\!>\!r\!+\!1$ for some $(r,j)\!\in\!S(c_0,\c)$.
Since the number of collections of subsets above~is
$$\binom{m'}{c_0,\c,m'\!-\!c_0\!-\!|\c|}
\equiv \frac{m'!}{c_0!\c!(|\b|\!+\!|\b'|\!-\!\mu_g(I)\!-\!m\!-\!c_0\!-\!|\c|)!} \,,$$
it follows that the coefficient of $\Psi_0^{c_0}\Psi^{\c}$ 
on the left-hand side of~\e_ref{res_e} is
$$\sum_{m'=0}^{\i}\!\LHS_{m'}(c_0,\c)
=\frac{\ze^{|\b|-\mu_g(I)-m+\|\c\|}}{\om(\c)c_0!\c!}A_{\b,I;c_0,\c}^{(g)}.$$
The claim now follows from \e_ref{hatA_e} and the last statement in Lemma~\ref{comb_l0}.
\end{proof}

\vspace{.1in}

\noindent
{\it Department of Mathematics, Stony Brook University, Stony Brook, NY 11794\\
azinger@math.stonybrook.edu}

\end{document}